\documentclass[a4paper,11pt]{amsart}

\usepackage{amssymb,amsmath,amsthm,mathrsfs,enumerate,graphicx, color}
\usepackage[pdfpagelabels,colorlinks,linkcolor=blue,citecolor=black,urlcolor=blue]{hyperref}
\usepackage{esint}
\newtheorem{thm}{Theorem}[section]

\newtheorem{cor}[thm]{Corollary}
\newtheorem{lem}[thm]{Lemma}
\newtheorem{prop}[thm]{Proposition}
\newtheorem{defn}[thm]{Definition}
\newtheorem{rem}[thm]{Remark}

\newcommand{\SR}{\mathscr{S}(\mathbb{R})}
\newcommand{\SL}{\mathcal{S}_{\vc}}
\newcommand{\SLm}{\mathcal{S}'_{\vc}}

\newcommand{\lesi}{\lesssim}

\newcommand{\B}{\dot{B}}
\newcommand{\F}{\dot{F}}

\newcommand{\sL}{\sqrt{L}}
\newcommand{\dx}{d\mu(x)}
\newcommand{\dy}{d\mu(y)}
\newcommand{\dz}{d\mu(z)}

\newcommand{\supp}{\operatorname{supp}}
\newcommand{\f}{\frac}

\newcommand{\vc}{\infty}

\title[Weighted Besov and Triebel--Lizorkin spaces associated to operators]{Weighted Besov and Triebel--Lizorkin spaces associated to operators}         % Enter your title between curly braces
\author{Huy-Qui Bui}
\address{Department of Mathematics, University of Canterbury, Private Bag 4800, Christchurch 8140, New Zealand}
\email{huy-qui.bui@canterbury.ac.nz, quihuybui@gmail.com}

\author{The Anh Bui}
\address{Department of Mathematics, Macquarie University, NSW 2109,
Australia}
\email{the.bui@mq.edu.au, bt\_anh80@yahoo.com}
 
\author{Xuan Thinh Duong}
\address{Department of Mathematics, Macquarie University, NSW 2109,
Australia}

\email{xuan.duong@mq.edu.au}


\begin{document}

\begin{abstract}
	Let $X$ be a space of homogeneous type and $L$ be a nonnegative self-adjoint operator on $L^2(X)$ satisfying Gaussian upper bounds on its heat kernels.
	 In this paper we develop the theory of weighted Besov spaces $\B^{\alpha,L}_{p,q,w}(X)$ and weighted Triebel--Lizorkin spaces $\F^{\alpha,L}_{p,q,w}(X)$ associated to the operator $L$ for the full range $0<p,q\le \vc$, $\alpha\in \mathbb R$ and $w$ being in the Muckenhoupt weight class $A_\vc$. Similarly to the classical case in the Euclidean setting, we prove that our new spaces satisfy important features such as continuous charaterizations in terms of square functions, atomic decompositions and the identifications with some well known function spaces such as Hardy type spaces and Sobolev type spaces. Moreover, with extra assumptions on the operator $L$, we prove that the new function spaces associated to $L$ coincide with the classical function spaces. Finally we apply our results to prove the boundedness of the fractional power of $L$ and the spectral multiplier of $L$ in our new function spaces.
\end{abstract}
\date{}

\maketitle

\tableofcontents

\section{Introduction}
Let $X$ be a space of homogeneous type, with quasi distance $d$ and
$\mu$ is a nonnegative Borel measure on $X$, which satisfies the doubling property below. \emph{In this paper, we assume that $\mu(X)=\infty.$}\\

For $x\in X$ and $r>0$ we set $B(x, r)=\{y\in X:d(x,y)<r\}$ to be the open ball with radius $r >0$ and center $x\in
X$, and $V(x,r)=\mu(B(x, r))$. The doubling
property of $\mu$ provides a constant $C>0$ such that
\begin{equation}\label{doubling1}
V(x,2r)\leq CV(x,r)
\end{equation}
for all $x\in X$ and $r>0$.\\
The doubling property (\ref{doubling1}) yields a constant $n>0$ so that
\begin{equation}\label{doubling2}
V(x,\lambda r)\leq C\lambda^nV(x,r),
\end{equation}
for all $\lambda\geq 1, x\in X$ and $r>0$; and that
\begin{equation}\label{doubling3}
V(x,r)\leq C\Big(1+\frac{d(x,y)}{r}\Big)^{\tilde n}V(y,r),
\end{equation}
for all $x,y\in X$, $r>0$ and for some $\tilde n \in [0,n]$.\\

Let $L$ be a nonnegative self-adjoint  operator on $L^2(X)$ which generates a semigroup $\{e^{-tL}\}_{t>0}$. Denote by $p_t(x,y)$ the  kernel of the semigroup $e^{-tL}$. In this paper, we assume that the kernel $p_t(x,y)$ satisfies a Gaussian upper bound, i.e., there exist two positive constants $C$ and  $c$ so that for all $x,y\in X$ and $t>0$,
\begin{equation}
\tag{GE}\label{GE}
\displaystyle |p_t(x,y)|\leq \f{C}{\mu(B(x,\sqrt{t}))}\exp\Big(-\f{d(x,y)^2}{ct}\Big).
\end{equation}

The theory of classical Besov and Triebel-Lizorkin spaces has played an essential role in the theory of function approximation and partial differential equations, and has been developed by many mathematicians on Euclidean spaces $\mathbb{R}^n$. See \cite{FJ1, FJ2, P, Tr, Tr2, Besov1, Besov2, Besov3, Qui2, Qui3, Qui4, Qui5} and the references therein. Recently, this theory has been extended to several directions including the theory of these function spaces in the general setting of a metric measure space and the theory of new function spaces associated to operators. We now list some bodies of work related to these research directions.
\begin{enumerate}[(i)]
	\item Using the existence of the approximation of identity, the authors developed the theory of Besov spaces $\B^s_{p,q}$ and Triebel-Lizorkin spaces $\F^s_{p,q}$ for a range $1\le p,q\le \vc$ and $s\in (-\theta,\theta)$ for some $\theta\in (0,1)$ on metric measure spaces with polynomial volume growths \cite{HS} and 
	on  spaces with doubling and reverse doubling measures \cite{HMY}.
	
	\item In \cite{PX} the authors introduced new Besov and Triebel-Lizorkin spaces associated to the Hermite operator with a full range of indices. Similarly to the classical results, they proved the frame decompositions for these spaces by making use of estimates of eigenvectors of the Hermite operators. Similar results was also proved for the Laguerre operator in \cite{KP2}. The theory of these function spaces was further developed in \cite{BD1, BD2} where the authors proved molecular and atomic decompositions theorems and square function characterizations for these spaces. 
	
	\item In \cite{BDY} the authors introduced the theory of Besov spaces $\B^{s,L}_{p,q}$ associated to an operator $L$ satisfying Poisson estimates on metric spaces with a measure enjoying a polynomial upper bound on volume growth. However, the indices are only for $1\le p,q\le \vc$ and $-1<s<1$. The restriction of indices
	is due to some technical reasons and the absence of a suitable space of distributions.
	
	\item Recently, under the assumption that $L$ is a nonnegative self-adjoint operator satisfying Gaussian upper bounds, H\"older continuity and Markov semigroup properties, the frame decompositions of Besov and Triebel-Lizorkin spaces associated to $L$ with full range of indices  were studied in \cite{KP, G.etal}. This theory has a wide range of applications from the setting of Lie groups to Riemannian manifolds.
\end{enumerate}
Our main aim in this paper is to conduct the theory of \textit{weighted Besov and Triebel--Lizorkin spaces} associated to a nonnegative self-adjoint operator. In contrast to \cite{KP, G.etal}, we assume the Gaussian upper bound \eqref{GE}, \textit{but do not assume H\"older continuity on the heat kernels nor the Markov properties (the conservation property)}. This allows our theory to cover a wider range of applications. In this article, in order to introduce the weighted Besov and Triebel--Lizorkin spaces we adapt ideas in \cite{P, Tr, Tr2} to make use of  spectral decompositions of nonnegative self-adjoint operators. Our development on the new theory of weighted Besov and Triebel--Lizorkin spaces can be viewed as a generalization of the classical settings in $\mathbb{R}^n$. Our techniques can be considered as a combination of classical tools such as Calder\'on reproducing formulas and new tools such kernel estimates for spectral decompositions of nonnegative self-adjoint operators. These techniques enable us to obtain some useful maximal function estimates which play an essential role in the proofs of our main results. Our point of view and approach are motivated and influenced by a large body of works in the classical setting, especially  \cite{P,FJ2,ST,CT,Qui3,Qui4,BC}. 
For convenience we list the main contributions of the article:
%\begin{enumerate}[(i)]

	$\bullet$ Section 2 recalls the Fefferman--Stein maximal inequality and gives the definition of a new class of distributions. More importantly, we prove a number of Calder\'on reproducing formulas and some maximal function estimates which play a crucial role in the proofs of the main results.
	
	$\bullet$ Section 3 gives definitions of the weighted Besov and Triebel--Lizorkin spaces. Then we prove continuous characterizations including square function characterizations via heat kernels of the new weighted Besov and Triebel--Lizorkin spaces. In particular, the characterization via non-compactly supported spectral functional calculus is a significant advance of our paper (see Theorem~\ref{thm2}). Moreover, we also study the Triebel-Lizorkin space 
	$\F_{p,q,w}^{\alpha, L}$ with $p = \infty$, which does not seem to be considered in this setting in the literature before.
	
	$\bullet$ Section 4 establishes atomic decompositions of these new spaces where we prove similar results to the classical function spaces in the Euclidean setting.
	
	$\bullet$ Section 5 derives identifications of these new spaces with well known spaces such as Hardy spaces and  Sobolev Hardy spaces associated to operators.
	
	$\bullet$ Section 6 compares our new spaces with classical Besov and Triebel-Lizorkin spaces. 
	
	$\bullet$ Section 7 proves the boundedness of singular integrals including fractional power of $L$ and spectral multipliers $m(L)$ on the new function spaces.
%\end{enumerate} 

We note  that,  after a preliminary version of the manuscript was completed,
we learned  that the atomic decompositions have  been  obtained
independently in \cite{G.etal2} by using a different approach. It is worth noticing that in comparison with \cite{G.etal2}, our assumptions on the operator $L$ and the underlying spaces $X$ are weaker than those in  \cite{G.etal2}. In our paper, $L$ is assumed to be nonnegative, self--adjoint and has a Gaussian upper bound, whereas apart from these conditions, in \cite{G.etal2} the operator $L$ is assumed to satisfy some additional conditions such as H\"older continuity estimate and Markov semigroup
property. Moreover, the reverse doubling condition and non-collapsing condition for the underlying space are additionally required in \cite{G.etal2}. 
It is worth emphasizing that our approach can be adapted to study the generalized Besov and Triebel--Lizorkin spaces as in \cite{YSY} or Besov and Triebel--Lizorkin spaces with variable exponents which are associated to operators satisfying the Gaussian upper bounds of order $m$.

Throughout the paper, we usually use $C$ and $c$ to denote positive constants that are independent of the main parameters involved but whose values may differ from line to line. We will write
$A\lesi B$ if there is a universal constant $C$ so that $A\leq CB$ and $A\sim B$ if $A\lesi B$ and $B\lesi A$.  

\section{Preliminaries, a new class of distributions and related estimates}
\subsection{Muckenhoupt weights, Fefferman-Stein inequality and some estimates on spaces of homogeneous type}
To simplify notation, we will often use $B$ for $B(x_B, r_B)$.
Also given $\lambda > 0$, we will write $\lambda B$ for the
$\lambda$-dilated ball, which is the ball with the same center as
$B$ and with radius $r_{\lambda B} = \lambda r_B$. For each ball
$B\subset X$ we set
$$
S_0(B)=4B \ \text{and} \ S_j(B) = 2^jB\backslash 2^{j-1}B \
\text{for} \ j\geq 3.
$$
For $x,y\in X$ and $r>0$, we denote $V(x\wedge y,r)=\min\{V(x,r),V(y,r)\}$ and $V(x\vee y,r)=\max\{V(x,r),V(y,r)\}$.

\medskip

A weight $w$ is a non-negative measurable and locally integrable function on $X$. Let $w$ be a weight. For
any measurable set $E \subset X$, we denote $w(E) :=\int_E w(x)\dx$ and $V(E)=\mu(E)$. We denote
$$
\fint_E h(x)\dx=\f{1}{V(E)}\int_Eh(x)\dx.
$$
For $1 \leq p \leq \infty$ let $p'$
be the conjugate exponent of $p$, i.e. $1/p + 1/p' = 1$.
\medskip

We say that a weight $w \in A_p$, $1 < p < \infty$, if the following holds true:
\begin{equation}
\label{defn Ap}
[w]_{A_p}:=\sup_{B: {\rm balls}}\Big(\fint_B w(x)\dx\Big)^{1/p}\Big(\fint_B w(x)^{-1/(p-1)}\dx\Big)^{(p-1)/p}< \vc.
\end{equation}
For $p = 1$, we say that $w \in A_1$ if there is a constant $C$ such
that for every ball $B \subset X$,
$$
\fint_B w(y)\dy \leq Cw(x) \ \text{for a.e. $x\in B$}.
$$
We set $A_\vc=\cup_{p\geq 1}A_p$.\\

The reverse H\"older classes are defined in the following way: a weight $w
\in RH_q, 1 < q < \infty$, if there is a constant $C$ such that for
any ball $B \subset X$,
$$
\Big(\fint_B w(y)^q \dy\Big)^{1/q} \leq C \fint_B w\dx.
$$
The endpoint $q = \infty$ is given by the condition: $w \in
RH_\infty$, if there is a constant $C$ such that for any ball
$B \subset X$,
$$
w(x)\leq C \fint_B w(y)\dy  \ \text{for a.e. $x\in B$}.
$$

For $w \in A_\vc$ and $0< p <\infty$, the weighted space $L^p_w(X)$
is defined by
$$\Big\{f :\int_{X} |f(x)|^p w(x)\dx < \infty\Big\}$$
with the norm
$$\|f\|_{p,w}=\Big(\int_{X} |f(x)|^p w(x)\dx\Big)^{1/p}.$$

We sum up some of the standard properties of classes of weights in \cite{ST} in the following lemma.
\begin{lem}\label{weightedlemma1}
	The following properties hold:
	\begin{enumerate}[{\rm (i)}]
		\item $A_1\subset A_p\subset A_q$ for $1< p\leq q< \infty$.
		\item If $w \in A_p, 1 < p < \vc$, then there exists $1<r < p < \vc$ such that $w \in A_r$.
		\item If $w\in A_p, 1<p<\vc$, then $w^{1-p'}\in A_{p'}$.
		\item If $w\in A_p, 1\le p<\vc$, then there exists $C>0$ so that for any ball $B$ and any measurable subset $E\subset B$ we have
		\begin{equation}
		\label{eq- doubling w}
		\f{w(B)}{w(E)}\le C\Big(\f{V(B)}{V(E)}\Big)^p.
		\end{equation}
		\item If $w \in RH_p, 1 < p < \vc$, then there exists $1<p < q < \vc$ such that $w \in RH_q$.
		\item $\cup_{1<p<\vc} A_p =\cup_{1<q<\vc}RH_q$.
	\end{enumerate}
\end{lem}
For $w\in A_\vc$ we define $q_w=\inf\{q: w\in A_q\}$ and $r_w=\sup\{r: w\in RH_r\}$.

%The following simple estimate will be useful in the sequel.
%\begin{lem}
%	\label{lem2-weight}
%	Let $w\in A_p, 1<p<\vc$ and $N>np'$. Then we have
%	\[
%	\Big[\int_X \Big(1+\f{d(x,y)}{t}\Big)^{-N}w(y)^{-p'/p}\dy\Big]^{p/p'}\lesi \f{V(x,t)^p}{w(B(x,t))}
%	\]
%	for all $t>0$.
%\end{lem}
%\begin{proof}
%The proof of this lemma is simple and we omit the details.

%\end{proof}
\bigskip
Let $w\in A_\vc$ and $0<r<\vc$. The weighted Hardy--Littlewood maximal function $\mathcal{M}_{r,w}$ is defined by
$$
\mathcal{M}_{r,w} f(x)=\sup_{x\in B}\Big(\f{1}{w(B)}\int_B|f(y)|^rw(y)\dy\Big)^{1/r}
$$
where the sup is taken over all balls $B$ containing $x$. We will drop the subscripts $r$ or $w$ when either $r=1$ or $w\equiv 1$.

Let $w\in A_\vc$ and $0<r<\vc$. It is well-known that 
\begin{equation}
\label{boundedness maximal function}
\|\mathcal{M}_{r,w} f\|_{p,w}\lesi \|f\|_{p,w}
\end{equation}
for all $p>r$.

Moreover, let $0<r<\vc$ and $p>r$. Then we have
\begin{equation}
\label{boundedness maximal function 2}
\|\mathcal{M}_{r} f\|_{p,w}\lesi \|f\|_{p,w}
\end{equation}
for $w\in A_{p/r}$.

The following elementary estimates will be used frequently. See for example \cite{BDK}.
\begin{lem}\label{lem-elementary}
	Let $\epsilon >0$.
	\begin{enumerate}[{\rm (a)}]
		\item For any $p\in [1,\vc]$ we have
		$$
		\Big(\int_X\Big[\Big(1+\f{d(x,y)}{s}\Big)^{-n-\epsilon}\Big]^p\dy\Big)^{1/p}\lesi V(x,s)^{1/p},
		$$
		for all $x\in X$ and $s>0$.
		
		\item For any $f\in L^1_{\rm loc}(X)$ we have
		$$
		\int_X\f{1}{V(x\wedge y,s)}\Big(1+\f{d(x,y)}{s}\Big)^{-n-\epsilon}|f(y)|\dy\lesi \mathcal{M}f(x),
		$$
		for all $x\in X$ and $s>0$.		
	\end{enumerate}
\end{lem}

We recall the Fefferman-Stein vector-valued maximal inequality and its variant in \cite{GLY} and \cite{KP}. For $0<p<\vc$, $0<q\leq \vc$, $0<r<\min \{p,q\}$ and $w\in A_{p/r}$, we then have for any sequence of measurable functions $\{f_\nu\}$,
\begin{equation}\label{FSIn}
\Big\|\Big(\sum_{\nu}|\mathcal{M}_rf_\nu|^q\Big)^{1/q}\Big\|_{p,w}\lesi \Big\|\Big(\sum_{\nu}|f_\nu|^q\Big)^{1/q}\Big\|_{p,w}.
\end{equation}
For any measurable function $F: X\times \mathbb{R}_+ \to \mathbb{C}$ with respect to the product measure $d\mu\times du$ one has, for $0<p<\vc$, $0<q\leq \vc$, $w\in A_{p/r}$ and $0<r<\min \{p,q\}$,
\begin{equation}\label{FSIn1}
\Big\|\Big(\int_0^\vc \Big[\mathcal{M}_r(F(\cdot,u))(\cdot)\Big]^q\f{du}{u}\Big)^{1/q}\Big\|_{p,w}\lesi \Big\|\Big(\int_0^\vc |(F(\cdot,u))(\cdot)|^q\f{du}{u}\Big)^{1/q}\Big\|_{p,w}.
\end{equation}
The Young's inequality and \eqref{FSIn} imply the following  inequality: If $\{a_\nu\} \in \ell^{q}\cap \ell^{1}$, then 
\begin{equation}\label{YFSIn}
\Big\|\sum_{j}\Big(\sum_\nu|a_{j-\nu}\mathcal{M}_r f_\nu|^q\Big)^{1/q}\Big\|_{p,w}\lesi \Big\|\Big(\sum_{\nu}|f_\nu|^q\Big)^{1/q}\Big\|_{p,w}.
\end{equation}
%(Cf. also \cite[Proposition 2.7]{BC}.)
\bigskip

We will now recall  an important covering lemma in \cite{C}.
\begin{lem}\label{Christ'slemma} There
	exists a collection of open sets $\{Q_\tau^k\subset X: k\in
	\mathbb{Z}, \tau\in I_k\}$, where $I_k$ denotes certain (possibly
	finite) index set depending on $k$, and constants $\rho\in (0,1),
	a_0\in (0,1]$ and $\kappa_0\in (0,\vc)$ such that
	\begin{enumerate}[(i)]
		\item $\mu(X\backslash \cup_\tau Q_\tau^k)=0$ for all $k\in
		\mathbb{Z}$;
		\item if $i\geq k$, then either $Q_\tau^i \subset Q_\beta^k$ or $Q_\tau^i \cap
		Q_\beta^k=\emptyset$;
		\item for every $(k,\tau)$ and each $i<k$, there exists a unique $\tau'$
		such $Q_\tau^k\subset Q_{\tau'}^i$;
		\item the diameter ${\rm diam}\,(Q_\tau^k)\leq \kappa_0 \rho^k$;
		\item each $Q_\tau^k$ contains certain ball $B(x_{Q_\tau^k}, a_0\rho^k)$.
	\end{enumerate}
\end{lem}
\begin{rem}\label{rem1-Christ}
	Since the constants $\rho$ and $a_0$ are not essential in the paper, without loss of generality, we may assume that $\rho=a_0=1/2$.  We then fix a collection of open sets in Lemma \ref{Christ'slemma} and denote this collection by $\mathcal{D}$. We call open sets in $\mathcal{D}$ the dyadic cubes in $X$ and $x_{Q_\tau^k}$ the center of the cube $Q_\tau^k \in \mathcal{D}$. We also denote 
	$$\mathcal{D}_\nu:=\{Q_\tau^{\nu+1} \in \mathcal{D} : \tau\in I_{\nu+1}\}
	$$ 
	for each $\nu\in \mathbb{Z}$. Then for $Q\in \mathcal{D}_\nu$, we have $B(x_Q, c_02^{-\nu})\subset Q\subset B(x_Q, \kappa_0 2^{-\nu})=:B_Q$, where $c_0$ is a constant independent of $Q$.
\end{rem}

\subsection{Kernel estimates}

Denote by $E_L(\lambda)$ a spectral decomposition of $L$. Then by spectral theory, for any bounded Borel funtion $F:[0,\vc)\to \mathbb{C}$ we can define
$$
F(L)=\int_0^\vc F(\lambda)dE_L(\lambda)
$$
as a bounded operator on $L^2(X)$. It is well-known that the kernel $K_{\cos(t\sqrt{L})}$ of $\cos(t\sqrt{L})$ satisfies 
\begin{equation}\label{finitepropagation}
{\rm supp}\,K_{\cos(t\sqrt{L})}\subset \{(x,y)\in X\times X:
d(x,y)\leq t\}.
\end{equation}
See for example \cite{CS}.

We have the following useful lemma. See for example \cite{HLMMY}.
\begin{lem}\label{lem:finite propagation}
	Let $\varphi\in \SR$ be an even function with {\rm supp}\,$\varphi\subset (-1, 1)$ and $\int \varphi =2\pi$. Denote by $\Phi$ the Fourier transform of $\varphi$. Then for every $k \in \mathbb{N}$, the kernel $K_{(t^2L)^k\Phi(t\sqrt{L})}$ of $(t^2L)^k\Phi(t\sqrt{L})$ satisfies 
	\begin{equation}\label{eq1-lemPsiL}
	\displaystyle
	{\rm supp}\,K_{(t^2L)^k\Phi(t\sqrt{L})}\subset \{(x,y)\in X\times X:
	d(x,y)\leq t\},
	\end{equation}
	and
	\begin{equation}\label{eq2-lemPsiL}
	|K_{(t^2L)^k\Phi(t\sqrt{L})}(x,y)|\leq \f{C}{V(x,t)}.
	\end{equation}
\end{lem}

The following lemma gives some kernel estimates which play an important role in the proof of our main results.

\begin{lem}
	\label{lem1}
	\begin{enumerate}[{\rm (a)}]
		\item Let $\varphi\in \mathscr{S}(\mathbb{R})$ be an even function. Then for any $N>0$ there exists $C>0$ such that 
		\begin{equation}
		\label{eq1-lema1}
		|K_{\varphi(t\sqrt{L})}(x,y)|\leq \f{C}{V(x\vee y,t)}\Big(1+\f{d(x,y)}{t}\Big)^{-N},
		\end{equation}
		for all $t>0$ and $x,y\in X$.
		\item Let $\varphi_1, \varphi_2\in \mathscr{S}(\mathbb{R})$ be even functions. Then for any $N>0$ there exists $C>0$ such that
		\begin{equation}
		\label{eq2-lema1}
		|K_{\varphi_1(t\sqrt{L})\varphi_2(s\sqrt{L})}(x,y)|\leq C\f{1}{V(x\vee y,t)}\Big(1+\f{d(x,y)}{t}\Big)^{-N},
		\end{equation}
		for all $t\leq s<2t$ and $x,y\in X$.
		\item Let $\varphi_1, \varphi_2\in \mathscr{S}(\mathbb{R})$ be even functions with $\varphi^{(\nu)}_2(0)=0$ for $\nu=0,1,\ldots,2\ell$ for some $\ell\in\mathbb{Z}^+$. Then for any $N>0$ there exists 
		$C>0$ such that
		\begin{equation}
		\label{eq3-lema1}
		|K_{\varphi_1(t\sqrt{L})\varphi_2(s\sqrt{L})}(x,y)|\leq C\Big(\f{s}{t}\Big)^{2\ell} \f{1}{V(x\vee y,t)}\Big(1+\f{d(x,y)}{t}\Big)^{-N},
		\end{equation}
		for all $t\geq s>0$ and $x,y\in X$.
	\end{enumerate}
	
	Note that any function in $\mathscr{S}(\mathbb{R})$ with compact
	 support in $(0,\infty)$ can be extended to an even function in 
	 $\mathscr{S}(\mathbb{R})$ with derivatives of all orders vanish at $0$. Hence the results in each part (a), (b) and (c) hold for such functions. 
\end{lem}
\begin{proof}
	\noindent (a) The estimate \eqref{eq1-lema1} was proved in \cite[Lemma 2.3]{CD} in the particular case $X=\mathbb{R}^n$ but the proof is still valid in the case of spaces of homogeneous type.
	
	For the items (b) and (c) we refer to \cite{BDK}.
\end{proof}
\bigskip
\begin{rem}
	(i) From \eqref{doubling3}, the term $V(x\vee y, t)$ on the right hand side of estimates in Lemma \ref{lem1} can be replaced by $V(x\vee y, d(x,y))$. 
	
	(ii) We will often use the following inequality
	\[
	\Big(1+\f{d(x,y)}{t}\Big)^{-N}\Big(1+\f{d(x,z)}{t}\Big)^{-N}\leq \Big(1+\f{d(x,z)}{t}\Big)^{-N}
	\]
	for all $x,y,z\in X$ and all $t, N>0$.
	
	We may use these in the sequel without stating any reasons. 
\end{rem}

\bigskip

\subsection{A new class of distributions}
We fix $x_0\in X$ as a reference point in $X$. The class of test functions $\mathcal{S}$ associated to $L$ is defined as the set of all functions $\phi \in \cap_{m\geq 1}D(L^m)$ such that
\begin{equation}
\label{Pml norm}
\mathcal{P}_{m,\ell}(\phi)=\sup_{x\in X}(1+d(x,x_0))^m|L^\ell \phi(x)|<\vc, \ \ \forall m>0, \ell \in \mathbb{N}.
\end{equation}
It was proved in \cite{KP} that $\mathcal{S}$ is a complete locally convex space with topology generated by the family of semi-norms $\{\mathcal{P}_{m,\ell}: \, m>0, \ell \in \mathbb{N}\}$. As usual, we define  the space of distributions $\mathcal{S}'$ as the set of all continuous linear functional on $\mathcal{S}$ with the inner product defined by
\[
\langle f,\phi\rangle=f(\overline{\phi})
\]
for all $f\in \mathcal{S}'$ and $\phi\in \mathcal{S}$.

The space of distributions $\mathcal{S}'$ can be used to define the inhomogeneous Besov and Triebel-Lizorkin spaces. However, in order to study the homogeneous version of these spaces we need some modifications. 

Following \cite{G.etal} we define the space $\mathcal{S}_\vc$ as the set of all functions $\phi \in \mathcal{S}$ such that for each $k\in \mathbb{N}$ there exists $g_k\in \mathcal{S}$ so that $\phi=L^kg_k$. Note that such an $g_k$, if exists, is unique. See \cite{G.etal}.

The topology in $\mathcal{S}_\vc$ is generated by the following family of semi-norms 
\[
\mathcal{P}^*_{m,\ell,k}(\phi)=\mathcal{P}_{m,\ell}(g_k), \ \forall m>0; \ell, k\in \mathbb{N}
\]
where $\phi=L^k g_k$.

We then denote by $\mathcal{S}_\vc'$ the set of all continuous linear functionals on $\mathcal{S}_\vc$.

In order to have an insightful understanding about the distributions in $\mathcal{S}'_\vc$. We define
\[
\mathscr P_m =\{g\in\mathcal S': L^mg=0\}, m\in \mathbb{N}
\]
and set $\mathscr P =\cup_{m\in \mathbb N}\mathscr P_m$.

From Proposition 3.7 in \cite{G.etal}, we have:
\begin{prop}
	The
	following
	identification
	is
	valid $\mathcal S'/\mathscr P = \mathcal{S}_\vc'$.
\end{prop}

It was proved in \cite{G.etal} that with $L=-\Delta$, the Laplacian on $\mathbb{R}^n$,  the distributions in $\mathcal S'/\mathscr P = \mathcal{S}_\vc'$ are identical with the classical tempered distributions modulo polynomial.

\medskip

From Lemma \ref{lem1}, we can see that if $\varphi\in \mathscr{S}(\mathbb{R})$ with supp\,$\varphi\subset (0,\vc)$, then we have $K_{\varphi(t\sL)}(x,\cdot)\in \mathcal{S}_\vc$ and $K_{\varphi(t\sL)}(\cdot,y)\in \mathcal{S}_\vc$. Therefore, we can define
\begin{equation}\label{eq- s and s'}
\varphi(t\sL)f(x)=\langle f, K_{\varphi(t\sL)}(x,\cdot)\rangle
\end{equation}
for all $f\in \mathcal{S}'_\vc$.

The support condition supp\,$\varphi\subset (0,\vc)$ is essential to be able to define $\varphi(t\sL)f$ with $f\in \mathcal{S}'_\vc$. In general, if $\varphi\in \mathscr{S}(\mathbb{R})$, then we have $K_{\varphi(t\sL)}(x,\cdot)\in \mathcal{S}$ and $K_{\varphi(t\sL)}(\cdot,y)\in \mathcal{S}$. In this situation, it is possible to define $\varphi(t\sL)f$ when $f\in \mathcal{S}'$, but it is not clear how to define $\varphi(t\sL)f$ when $f\in \mathcal{S}'_\vc$.

\begin{lem}
	Let $f\in \mathcal S'$ and $\varphi\in \mathscr S(\mathbb{R})$ be an even function. Then there exist $m>0$ and $K>0$ such that 
	\begin{equation}\label{eq- bound of varphi and f}
	|\varphi(t\sL)f(x)|\lesi \f{(t\vee t^{-1})^m}{V(x_0,t)}(1+d(x,x_0))^K.
	\end{equation}
		The similar estimate holds true if $f\in \mathcal S'_\vc$ and $\varphi\in \mathscr S(\mathbb{R})$ supported in $[1/2,2]$.
\end{lem}
\begin{proof}
	Since $\varphi(t\sL)(x,\cdot)\in \mathcal{S}$, we have
	\[
	|\varphi(t\sL)f(x)|\lesi \mathcal P_{m',\ell}(K_{\varphi(t\sL)}(x,\cdot))
	\]
	for some $m'\in \mathbb{N}$ and $\ell\ge 0$.
	
	Using the kernel estimates in Lemma \ref{lem1} and simple calculations we can find $m,K>0$ so that
	\[
	P_{m',\ell}(K_{\varphi(t\sL)}(x,\cdot))\lesi \f{(t\vee t^{-1})^m}{V(x_0,t)}(1+d(x,x_0))^m.
	\] 
	This implies \eqref{eq- bound of varphi and f}.
	
	The proof is similar when $f\in \mathcal S'_\vc$ and $\varphi\in \mathscr S(\mathbb{R})$ supported in $[1/2,2]$. We omit the details.
\end{proof}
\subsection{Calder\'on reproducing formulas}
In what follows, by a ``partition of unity'' we shall mean a function $\psi\in \mathcal{S}(\mathbb{R})$ such that $\supp\psi\subset[1/2,2]$, $\int\psi(\xi)\,\f{d\xi}{\xi}\neq 0$ and
$$\sum_{j\in \mathbb{Z}}\psi_j(\lambda)=1 \textup{ on } (0,\infty), $$
where $\psi_j(\lambda):=\psi(2^{-j}\lambda)$ for each $j\in \mathbb{Z}$.

\begin{prop}
	\label{prop-Calderon1}
	Let $\psi$ be a partition of unity. Then for any $f\in \mathcal{S}'_\vc$ we have
	\[
	f= \sum_{j\in \mathbb{Z}}\psi_j(\sL)f \ \ \text{in $\mathcal{S}'_\vc$}.
	\]
\end{prop}
\begin{proof}
	By duality it suffices to prove that for each $f\in \mathcal{S}_\vc$,
	\[
	f= \sum_{j\in \mathbb{Z}}\psi_j(\sL)f \ \ \text{in $\mathcal{S}_\vc$}.
	\]
	Indeed, since $f\in \mathcal{S}_\vc$, for each $k\in \mathbb{N}$ there exists $g_k\in \mathcal{S}$ so that $f=L^k g_k$.
	
	For $m>0$ and $\ell, k\in \mathbb{N}$ we have
	$$
	\mathcal{P}^*_{m,\ell,k}(\psi_j(\sL)f)=\sup_{x\in X}(1+d(x,x_0))^m |L^\ell \psi_j(\sL)g_k(x)|.
	$$
	If $j\geq 0$, then we have
	$$
	(1+d(x,x_0))^m |L^\ell \psi_j(\sL)f(x)|=(1+d(x,x_0))^m |2^{-2j}\tilde{\psi}_j(\sL)L^{\ell+1}g_k(x)|
	$$
	where $\tilde{\psi}(\xi)=\xi^{-2}\psi(\xi)$.
	
	This, along with Lemma \ref{lem1} and Lemma \ref{lem-elementary}, implies that
	\begin{equation}\label{eq0-Calderon1-proof}
	\begin{aligned}
	(1+d(x,x_0))^m &|L^\ell \psi_j(\sL)f(x)|\\
	&\lesi 2^{-2j}(1+d(x,x_0))^m\int_X\f{1}{V(x,2^{-j})}\Big(1+\f{d(x,y)}{2^{-j}}\Big)^{-N-m}|L^{\ell+1}g_k(y)|\dy\\
	&\lesi 2^{-2j}\int_X\f{1}{V(x,2^{-j})}\Big(1+\f{d(x,y)}{2^{-j}}\Big)^{-N}(1+d(y,x_0))^m|L^{\ell+1}g_k(y)|\dy\\
	&\lesi 2^{-2j}\mathcal{P}_{m,\ell+1}(g_k)\int_X\f{1}{V(x,2^{-j})}\Big(1+\f{d(x,y)}{2^{-j}}\Big)^{-N}\dy\\
	&\lesi 2^{-2j}\mathcal{P}^*_{m,\ell+1,k}(f)
	\end{aligned}
	\end{equation}
	as long as $N>n$.
	
	Hence,
\begin{equation}
\label{eq1-Calder1-proof}
\mathcal{P}^*_{m,\ell,k}(\psi_j(\sL)f)\lesi 2^{-2j}\mathcal{P}^*_{m,\ell+1,k}(f), \  \ \forall j\geq 0.
\end{equation}

If $j< 0$, then we have
$$
(1+d(x,x_0))^m |L^\ell \psi_j(\sL)f(x)|=2^{2j(\ell+k+m+1)}(1+d(x,x_0))^m | {\varphi}_j(\sL)g_{k+m+1}(x)|
$$
where $\varphi(\lambda)=\lambda^{2(\ell+k+m+1)}\psi(\lambda)$.

Arguing similarly to \eqref{eq0-Calderon1-proof} we obtain
\begin{equation}
\label{eq2-Calder1-proof}
\mathcal{P}^*_{m,\ell,k}(\psi_j(\sL)f)\lesi 2^{2j}\mathcal{P}^*_{m,0,k+m+1}(f), \  \ \forall j< 0.
\end{equation}
From \eqref{eq1-Calder1-proof}, \eqref{eq2-Calder1-proof} and the fact that $\mathcal{S}_\vc$ is complete, we deduce that there exists $h\in \mathcal{S}_\vc$ so that 
\[
h= \sum_{j\in \mathbb{Z}}\psi_j(\sL)f \ \ \text{in $\mathcal{S}_\vc$}.
\]
On the other hand, by spectral theory we have
\[
f= \sum_{j\in \mathbb{Z}}\psi_j(\sL)f \ \ \text{in $L^2(X)$}.
\]
Therefore, $f\equiv h$ and this concludes the proposition.
\end{proof}
\bigskip
\begin{prop}
	
	\label{prop-Calderon2}
	Let $\psi\in \mathscr{S}(\mathbb{R})$ be such that ${\rm supp \psi}\subset [1/2,2]$ and $\int \psi \f{d\xi}{\xi}\neq 0$. Then for any $f\in \mathcal{S}'_\vc$ we have
	\begin{equation}\label{Cal formula S'vc}
	f= c_\psi\int_0^\vc \psi(t\sL)f\f{dt}{t} \ \ \text{in $\mathcal{S}'_\vc$}
	\end{equation}
	where $c_\psi =\Big[\int_0^\vc \psi(t)\f{dt}{t}\Big]^{-1}$.
	
	Moreover, if $f\in \mathcal{S}'$, then there exists $\rho\in \mathscr P$ so that 
	\begin{equation}\label{Cal formula S'}
	f-\rho= c_\psi\int_0^\vc \psi(t\sL)f\f{dt}{t} \ \ \text{in $\mathcal{S}'$}.
	\end{equation}
\end{prop}
\begin{proof}
	We first show that
	\[
	f= c_\psi\int_0^\vc \psi(t\sL)f\f{dt}{t} \ \ \text{in $\mathcal{S}_\vc$}
	\]
	Arguing similarly to the proof of Proposition \ref{prop-Calderon1}, we can prove  that 
	\[
	\lim_{N\to \vc}\int_0^{1/N} \psi(t\sL)f\f{dt}{t}=\lim_{N\to \vc}\int_N^\vc \psi(t\sL)f\f{dt}{t}=0 \ \ \text{in $\mathcal{S}_\vc$},
   \]
   and consequently, by using the completeness of $\mathcal{S}_\vc$ and the spectral theory as in the proof of Proposition \ref{prop-Calderon1}, we obtain the representation 
	\[
	f= c_\psi\int_0^\vc \psi(t\sL)f\f{dt}{t} \ \ \text{in $\mathcal{S}_\vc$}.
	\]
	It follows by duality that 
	\[
	f= c_\psi\int_0^\vc \psi(t\sL)f\f{dt}{t} \ \ \text{in $\mathcal{S}'_\vc$}.
	\]
	For the second part, we note that for $g\in \mathcal S$ we have 
	\[
	\begin{aligned}
	\int_0^\vc \psi(s\sL)g\f{ds}{s}&=\int_0^1 \psi(s\sL)g\f{ds}{s} +\int_1^\vc \psi(s\sL)g\f{ds}{s}\\
	&=\int_0^1 \psi(s\sL)g\f{ds}{s} + \tilde \psi(\sL)g
	\end{aligned}
	\]
	where $\displaystyle\tilde \psi(x)=\int_{|x|}^\vc \psi(s)\f{ds}{s}=\int_1^\vc \psi(sx)\f{ds}{s}$.
	
	Hence, for any $\ell\in \mathbb{N}$,  $k>0$ and $x\in X$ we have
	\[
	\begin{aligned}
	(1+d(x,x_0))^k &\Big|L^\ell \Big(\int_0^\vc \psi(s\sL)g(x)\f{ds}{s}\Big)\Big|\\
	&\le (1+d(x,x_0))^k\int_0^1 |\psi(s\sL)L^\ell g(x)|\f{ds}{s}+(1+d(x,x_0))^k|\tilde \psi(\sL)L^\ell g(x)|\\
	&\le (1+d(x,x_0))^k \int_0^1 s^N|(s\sL)^{-2N}\psi(s\sL)L^{\ell +N}g(x)|\f{ds}{s}+(1+d(x,x_0))^k|\tilde \psi(\sL)L^\ell g(x)|\\
	&=:I_1(g) + I_2(g)
	\end{aligned}
	\]
	where $N$ is a fixed number which is greater than $k +n$.
	
	Note that $x^{-2N}\psi(x)$ is a function in $\mathscr S(\mathbb{R})$ supported in $[1/2,2]$. Using Lemma \ref{lem1} and arguing similarly to the proof of Proposition \ref{prop-Calderon1} we can show that 
	\[
	I_1(g)\lesi \mathcal{P}_{k,\ell +N}(g).
	\]   
	For the same reason, since $\tilde \psi$ is an even function in $\mathscr S(\mathbb{R})$, we also have
	\[
	I_2(g)\lesi \mathcal{P}_{k,\ell}(g).
	\]
	As a consequence,
	\[
	(1+d(x,x_0))^k \Big|L^\ell \Big(\int_0^\vc \psi(s\sL)g(x)\f{ds}{s}\Big)\Big|\lesi  \mathcal{P}_{k,\ell+N}(g)+ \mathcal{P}_{k,\ell}(g)
	\]
	for all $\ell\in \mathbb{N}$,  $k>0$ and $x\in X$.
	
	This implies that $\displaystyle \int_0^\vc \psi(s\sL)g\f{ds}{s}\in \mathcal S$ whenever $g\in \mathcal S$. By duality, $\displaystyle \int_0^\vc \psi(s\sL)f\f{ds}{s}\in \mathcal S'$ whenever $f\in \mathcal S'$. This, along with \eqref{Cal formula S'vc} and the fact that $\mathcal S'_\vc =\mathcal S'/\mathscr P$, implies that there exists $\rho\in \mathscr P$ so that 
	\begin{equation*}
	f-\rho=c_\psi\int_0^\vc \psi(s\sL)f\f{ds}{s}	
	\end{equation*}
	in $\mathcal S'$. 
	
	This completes our proof.
\end{proof}
\bigskip
\begin{lem}
	\label{lem-DKP}
	Let $\phi\in\mathscr{S}(\mathbb{R})$ be an even function such that $\phi(\xi)\neq 0$ on $(-2,-1/2)\cup(1/2,2)$. Then there exist $a, b, c > 0$ and even functions $\Psi, \psi \in \mathscr{S}(\mathbb{R})$ with ${\rm supp}\, \Psi\subset [-a,a]$, ${\rm supp}\, \psi\subset [-c,-b]\cup [b,c]$, so that for every $f\in \mathcal S'$ and every $j\in \mathbb{Z}$ and $t\in [1,2]$ we have
	$$
	f=\Psi(2^{-j}t\sqrt{L})f+\sum_{k\geq 1}\phi(2^{-(k+j)}t\sqrt{L})\psi(2^{-(k+j)}t\sqrt{L})f \ \ \text{in $\mathcal S'$}.
	$$
\end{lem}
\begin{proof}
	It is well-known that there exist  an even function $\psi \in \mathscr{S}(\mathbb{R})$ with  ${\rm supp}\, \psi\subset [-c,-b]\cup[b,c]$ for some $c>b>0$ so that 
	\[
	\sum_{k=-\infty}^\infty\phi(2^{-k}\lambda)\psi(2^{-k}\lambda)=1, \ \ \forall \lambda\neq 0.
	\]
	See for example \cite{ST}. Define $\Psi\in \mathscr{S}(\mathbb{R})$ by $\Psi(0)=1$, and $\Psi(\lambda)=\sum_{k\le 0}\phi(2^{-k}\lambda)\psi(2^{-k}\lambda)$ for $\lambda\neq 0$. Then it is easy to see that ${\rm supp}\, \Psi\subset [-a,a]$ for some $a>0$, and that 
	\[
	\Psi(\lambda) + \sum_{k\ge 1}\phi(2^{-k}\lambda)\psi(2^{-k}\lambda) = 1,\;\; \forall \lambda \in \mathbb{R}.
	\]
	
	Using the above identity and  arguing similarly to the proof of Proposition \ref{prop-Calderon1} we conclude that 
	\[
	f=\Phi(2^{-j}t\sqrt{L})\Psi(2^{-j}t\sqrt{L})f+\sum_{k\geq 1}\phi(2^{-(k+j)}t\sqrt{L})\psi(2^{-(k+j)}t\sqrt{L})f  \ \ \text{in $\mathcal S'$}.
	\]

(A close inspection of the proof in \cite{ST} shows that we can take $a = 2$, and $1/2 < b < c/2 < 1$.)
\end{proof}
\bigskip
\subsection{Maximal function estimates}
We begin with some technical estimates.
\begin{lem}
	\label{lem2}
	Let $\psi, \varphi\in \mathscr{S}(\mathbb{R})$ be even functions. Assume that $0<a\le b<\vc$ and $\lambda\geq 0$. Then there exists $C>0$ such that
	\begin{equation}
	\label{eq-lem1-maximal}
	\sup_{y\in X}|\psi(s\sL)\varphi(t\sL)f(y)|\Big(1+\f{d(x,y)}{t}\Big)^{-\lambda}\leq C\sup_{y\in X}|\varphi(t\sL)f(y)|\Big(1+\f{d(x,y)}{t}\Big)^{-\lambda}
	\end{equation}
	for all $f\in \mathcal S'$, $x\in X$ and $s\in [at, bt]$. If both funtions $\psi$ and $\varphi$ are supported in $[1/2,2]$, then \eqref{eq-lem1-maximal} holds for all $f\in \mathcal{S}'_\vc$.
\end{lem}
\begin{proof}
	From Lemma \ref{lem1}, for $N>n$ and $f\in \mathcal S'$ we have
	\begin{align*}
	|\psi(s\sL)\varphi(t\sL)f(y)|&\Big(1+\f{d(x,y)}{t}\Big)^{-\lambda}\\
	&\lesi \Big(1+\f{d(x,y)}{t}\Big)^{-\lambda}\int_X\f{1}{V(y,s)}\Big(1+\f{d(y,z)}{s}\Big)^{-N-\lambda}|\varphi(t\sL)f(z)|\dz\\
	&\sim \Big(1+\f{d(x,y)}{t}\Big)^{-\lambda}\int_X\f{1}{V(y,t)}\Big(1+\f{d(y,z)}{t}\Big)^{-N-\lambda}|\varphi(t\sL)f(z)|\dz\\
	&\lesi \int_X\f{1}{V(y,t)}\Big(1+\f{d(y,z)}{t}\Big)^{-N}\Big(1+\f{d(x,z)}{t}\Big)^{-\lambda}|\varphi(t\sL)f(z)|\dz.
	\end{align*}
This along with Lemma \ref{lem-elementary} implies the desired estimate.

Note that in the case that both functions $\psi$ and $\varphi$ are supported in $[1/2,2]$, then we can define $\psi(s\sL)\varphi(t\sL)f$ for all $f\in \mathcal{S}'_\vc$. The above argument also gives \eqref{eq-lem1-maximal} in this case. 

This completes our proof.
\end{proof}

\bigskip
\begin{lem}
	\label{lem1-thm1-continuous cha}
	Let $\psi\in \SR$ with ${\rm supp}\, \psi\subset[1/2,2]$ and $\int \psi\f{d\xi}{\xi}\neq 0$. Then for any $r>0$ and $N>0$ we have
	\begin{equation}
	\label{eq-lem1-continuous cha}
	|\psi(t\sL)f(x)|^r\lesi \int_0^\vc\int_X\f{1}{V(x,s)}\Big(1+\f{d(x,y)}{s}\Big)^{-Nr}\Big(\f{s}{t}\wedge \f{t}{s}\Big)^{Nr}|\psi(s\sL)f(y)|^r\dy\f{ds}{s}
	\end{equation}
	for all $f\in \mathcal{S}'_\vc,x\in X$ and $t>0$.
\end{lem}
\begin{proof}
	By Proposition \ref{prop-Calderon2} we have
	\[
	\psi(u\sL)f=c_\psi\int_0^\vc\psi(u\sL)\psi(s\sL)f\f{ds}{s} \ \ \ \text{pointwise},
	\]
	where $c_\psi=\Big[\int_0^\vc\psi(\xi)\f{d\xi}{\xi}\Big]^{-1}$.

	This, along with the fact that supp\,$\psi\subset [1/2,2]$, yields
	\begin{equation}\label{eq1-proof lem1}
	\psi(u\sL)f=c_\psi\int_{u/4}^{4u}\psi(u\sL)\psi(s\sL)f\f{ds}{s}.
	\end{equation}
	Applying Lemma \ref{lem1} we deduce that, for $N>0$,
	\begin{equation}\label{eq2-proof lem1}
	\begin{aligned}
	|\psi(u\sL)f(y)|&\lesi\int_{u/4}^{4u}\int_X \f{1}{V(y,u)}\Big(1+\f{d(y,z)}{u}\Big)^{-N}|\psi(s\sL)f(z)|\dz\f{ds}{s}\\
	&\lesi\int_{u/4}^{4u}\int_X \f{1}{V(y,s)}\Big(1+\f{d(y,z)}{s}\Big)^{-N}\Big(\f{s}{u}\Big)^{N}|\psi(s\sL)f(z)|\dz\f{ds}{s}.
	\end{aligned}
	\end{equation}
	If $r\in [1,\vc)$, then using \eqref{eq2-proof lem1} for $A+N$, where $A>n$, together with H\"older's inequality and Lemma~\ref{lem-elementary} we obtain \eqref{eq-lem1-continuous cha}. 
	
	We now consider the case $r\in (0,1)$. For each $x\in X$ and $t>0$ we define
	\[
	\psi^{**}_{N}(t\sL)f(x)=\sup_{y\in X} \sup_{s>0}\f{|\psi(s\sL)f(y)|}{\Big(1+\f{d(x,y)}{s}\Big)^{N}}\Big(\f{s}{t}\wedge \f{t}{s}\Big)^{N}.
	\]
Note that by Lemma~\ref{lem1} \eqref{eq2-proof lem1} holds with $V(y,s)$ replaced by $V(z,s)$. It follows that 
	\begin{equation}\label{eq-maximal inequality}
	\begin{aligned}
	|\psi(u\sL)&f(y)|\Big(1+\f{d(x,y)}{u}\Big)^{-N}\Big(\f{u}{t}\wedge \f{t}{u}\Big)^{N}\\
	&\lesi\int_{u/4}^{4u}\int_X \f{1}{V(z,s)}\Big(1+\f{d(y,z)}{s}\Big)^{-N}\Big(1+\f{d(x,y)}{u}\Big)^{-N}\Big(\f{u}{t}\wedge \f{t}{u}\Big)^{N}|\psi(s\sL)f(z)|\dz\f{ds}{s}\\
	&\sim\int_{u/4}^{4u}\int_X \f{1}{V(z,s)}\Big(1+\f{d(y,z)}{s}\Big)^{-N}\Big(1+\f{d(x,y)}{s}\Big)^{-N}\Big(\f{s}{t}\wedge \f{t}{s}\Big)^{N}|\psi(s\sL)f(z)|\dz\f{ds}{s}\\
	&\lesi  \int_{u/4}^{4u}\int_X \f{1}{V(z,s)}\Big(1+\f{d(x,z)}{s}\Big)^{-N} \Big(\f{s}{t}\wedge \f{t}{s}\Big)^{N}|\psi(s\sL)f(z)|\dz\f{ds}{s}\\
	&\lesi \psi^{**}_N(t\sL)f(x)^{1-r}\int_{u/4}^{4u}\int_X \f{1}{V(z,s)}\Big(1+\f{d(x,z)}{s}\Big)^{-Nr} \Big(\f{s}{t}\wedge \f{t}{s}\Big)^{Nr}|\psi(s\sL)f(z)|^r \dz\f{ds}{s}.
	\end{aligned}
	\end{equation}
	 
By \eqref{eq- bound of varphi and f}, there exists $N_f>0$ such that $\psi^{**}_N(t\sL)f(x)<\infty$ for all $N\ge N_f$ and $x\in X$. For any such $N$, taking the supremum over $u>0$ and $y\in X$ in the left-hand side of \eqref{eq-maximal inequality} we obtain

 	\begin{equation*}
	\begin{aligned}
	\psi^{**}_{N}(t\sL)f(x)^r
	&\le C(f,N) \int_{u/4}^{4u}\int_X \f{1}{V(z,s)}\Big(1+\f{d(x,z)}{s}\Big)^{-Nr} \Big(\f{s}{t}\wedge \f{t}{s}\Big)^{Nr}|\psi(s\sL)f(z)|^r \dz\f{ds}{s}\\
	&\le C(f,N) \int_{0}^{\vc}\int_X \f{1}{V(z,s)}\Big(1+\f{d(x,z)}{s}\Big)^{-Nr} \Big(\f{s}{t}\wedge \f{t}{s}\Big)^{Nr}|\psi(s\sL)f(z)|^r \dz\f{ds}{s}.
	\end{aligned}
	\end{equation*}
Consequently,
 		
	\begin{equation}\label{eq- N large}
	\begin{aligned}
	|\psi(t\sL)f(x)|^r
	&\le C(f,N)\int_{0}^{\vc}\int_X \f{1}{V(z,s)}\Big(1+\f{d(x,z)}{s}\Big)^{-Nr} \Big(\f{s}{t}\wedge \f{t}{s}\Big)^{Nr}|\psi(s\sL)f(z)|^r \dz\f{ds}{s}.
	\end{aligned}
	\end{equation}	
 
Since the integral on the right hand side of the above gets larger when $N$ gets smaller, \eqref{eq- N large} holds true for all $N>0$ and all $x\in X$.
It follows that
	\begin{equation}\label{eq1- N large}
	\begin{aligned}
	|\psi(u\sL)f(y)|^r
	&\le C(f,N) \int_{0}^{\vc}\int_X \f{1}{V(z,s)}\Big(1+\f{d(y,z)}{s}\Big)^{-2Nr} \Big(\f{s}{u}\wedge \f{u}{s}\Big)^{2Nr}|\psi(s\sL)f(z)|^r \dz\f{ds}{s}\\
	&\le C(f,N) \int_{0}^{\vc}\int_X \f{1}{V(z,s)}\Big(1+\f{d(y,z)}{s}\Big)^{-Nr} \Big(\f{s}{u}\wedge \f{u}{s}\Big)^{2Nr}|\psi(s\sL)f(z)|^r \dz\f{ds}{s}
	\end{aligned}
	\end{equation}	
for all $N>0$, $u>0$ and $y\in X$.
	 
	 	Using the obvious inequalities 
	\[
	\Big(\f{u}{t}\wedge \f{t}{u}\Big)^{Ar}\Big(\f{s}{u}\wedge \f{u}{s}\Big)^{Ar}\le \Big(\f{s}{t}\wedge \f{t}{s}\Big)^{Ar}
	\]
	and
	\[
	\Big(1+\f{d(y,z)}{s}\Big)^{-Ar}\Big(1+\f{d(x,y)}{u}\Big)^{-Ar}\Big(\f{s}{u}\wedge \f{u}{s}\Big)^{Ar}\lesi \Big(1+\f{d(x,z)}{s}\Big)^{-Ar},
	\]
where $A=N+\tilde{n}/r$,	we obtain
	\[
	\begin{aligned}
	\f{|\psi(u\sL)f(y)|^r}{\Big(1+\f{d(x,y)}{u}\Big)^{Ar}}\Big(\f{u}{t}\wedge \f{t}{u}\Big)^{Ar}&\le C(f,N,r)\int_{0}^{\vc}\int_X \f{1}{V(z,s)}\Big(1+\f{d(x,z)}{s}\Big)^{-Ar} \Big(\f{s}{t}\wedge \f{t}{s}\Big)^{Ar}|\psi(s\sL)f(z)|^r \dz\f{ds}{s}.
	\end{aligned}
	\]
	Taking the supremum over all $u>0$ and $y\in X$ gives
	\[
	\begin{aligned}
	\psi^{**}_A(t\sL)f(x)^r &\le C(f,N,r) \int_{0}^{\vc}\int_X \f{1}{V(z,s)}\Big(1+\f{d(x,z)}{s}\Big)^{-Ar} \Big(\f{s}{t}\wedge \f{t}{s}\Big)^{Ar}|\psi(s\sL)f(z)|^r \dz\f{ds}{s}\\
&\le 	C(f,N,r) \int_{0}^{\vc}\int_X \f{1}{V(x,s)}\Big(1+\f{d(x,z)}{s}\Big)^{-Nr} \Big(\f{s}{t}\wedge \f{t}{s}\Big)^{Nr}|\psi(s\sL)f(z)|^r \dz\f{ds}{s}
	\end{aligned}
	\]
for all $N>0$, $x\in X$ and $t>0$, where we use \eqref{doubling3} in the last inequality.

Therefore, if the integral on the right hand side of the above is finite, then $\psi^{**}_{N+\tilde n/r}(t\sL)f(x)<\infty$. In this case we repeat the arguments in the first part of the proof and obtain the required inequality \eqref{eq-lem1-continuous cha}, where the constant in the inequality depends on $N$ and $r$ (but independent of $f$). On the other hand, if this integral is infinite, then \eqref{eq-lem1-continuous cha} holds trivially. Thus we have proved \eqref{eq-lem1-continuous cha} for all $N>0$.
 \end{proof}

\begin{rem}\label{remark-maximal function}
Actually we have proved a stronger statement: 
\begin{equation}\label{eq-maximal inequality 2}
	\psi^{**}_{N+\tilde n/r}(t\sL)f(x)^r \le C_{N,r} \int_{0}^{\vc}\int_X \f{1}{V(x,s)}\Big(1+\f{d(x,z)}{s}\Big)^{-Nr} \Big(\f{s}{t}\wedge \f{t}{s}\Big)^{Nr}|\psi(s\sL)f(z)|^r \dz\f{ds}{s}
\end{equation}
for all $N>0$, $f\in \mathcal{S}'_\infty$, $x\in X$ and $t>0$. Moreover, a close inspection of the proof of the lemma shows that we also have 
\begin{equation}\label{eq-maximal inequality 3}
	\psi^{**}_N(t\sL)f(x)^r \le C_{N,r} \int_{0}^{\vc}\int_X \f{1}{V(z,s)}\Big(1+\f{d(x,z)}{s}\Big)^{-Nr} \Big(\f{s}{t}\wedge \f{t}{s}\Big)^{Nr}|\psi(s\sL)f(z)|^r \dz\f{ds}{s}
\end{equation}
for all $N>0$, $f\in \mathcal{S}'_\infty$, $x\in X$ and $t>0$.

\end{rem}

\bigskip
For $\lambda>0, j\in \mathbb{Z}$ and $\varphi\in \mathscr{S}(\mathbb{R})$ the Peetre's type maximal function is defined, for $f\in \mathcal S'$, by
\begin{equation}
\label{eq-PetreeFunction}
\varphi_{j,\lambda}^*(\sqrt{L})f(x)=\sup_{y\in X}\f{|\varphi_j(\sqrt{L})f(y)|}{(1+2^jd(x,y))^\lambda}\, , x\in X,
\end{equation}
where $\varphi_j(\lambda)=\varphi(2^{-j}\lambda)$.

Obviously, we have
\[
\varphi_{j,\lambda}^*(\sqrt{L})f(x)\geq |\varphi_j(\sqrt{L})f(x)|, \ \ \ \ x\in X.
\]
Similarly, for $s, \lambda>0$ we set
\begin{equation}
\label{eq2-PetreeFunction}
\varphi_{\lambda}^*(s\sqrt{L})f(x)=\sup_{y\in X}\f{|\varphi(s\sqrt{L})f(y)|}{(1+d(x,y)/s)^\lambda}, \ \ \ f\in \mathcal S'.
\end{equation}

We note that in the particular case when $\varphi$ is supported in $(0,\vc)$, these maximal functions can defined for $f\in \mathcal{S}_\vc'$ via \eqref{eq- s and s'}.

Due to \eqref{eq- bound of varphi and f}, $\varphi_{\lambda}^*(s\sqrt{L})f(x)<\vc$ for all $x\in X$, provided that $\lambda$ is sufficiently large.
\begin{prop}
	\label{prop1-maximal function}
	Let $\psi\in \mathscr{S}(\mathbb{R})$ with ${\rm supp}\,\psi\subset [1/2,2]$ and  $\varphi\in \mathscr{S}(\mathbb{R})$ be a partition of unity. Then for any $\lambda>0$ and $j\in \mathbb{Z}$ we have
\begin{equation}
\label{eq-psistar vaphistar}
\sup_{s\in [2^{-j-1},2^{-j}]}\psi^*_{\lambda}(s\sL)f(x) \lesi \sum_{k=j-2}^{j+3} \varphi^*_{k,\lambda}(\sL)f(x), 
\end{equation}
for all $f\in \mathcal{S}_\vc'$ and $x\in X$.
\end{prop}
\begin{proof}
	Fix $j\in \mathbb{Z}$ and $s\in [2^{-j-1},2^{-j}]$. First note that
	\[
	\psi(s\sL)=\sum_{k=j-2}^{j+3}\psi(s\sL)\varphi_k(\sL).
	\]
Since $2^{-j}\sim s$, by Lemma \ref{lem1} we have,  for $y\in X$ and $N>n$,
	\begin{equation}\label{eq1-pro1}
	\begin{aligned}
	|\psi(s\sL)f(y)|&\leq \sum_{k=j-2}^{j+3}|\psi(s\sL)\varphi_k(\sL)f(y)|\\
	&\lesi  \sum_{k=j-2}^{j+3}\int_X \f{1}{V(y,2^{-j})}(1+2^jd(y,z))^{-N-\lambda}|\varphi_k(\sL)f(z)|\dz.
	\end{aligned}
	\end{equation}
It follows that
	\begin{equation}\label{eq2-pro1}
	\begin{aligned}
	\f{|\psi(s\sL)f(y)|}{(1+d(x,y)/s)^\lambda}&\lesi  \sum_{k=j-2}^{j+3}\int_X \f{1}{V(y,2^{-j})}(1+2^jd(y,z))^{-N-\lambda}\f{|\varphi_k(\sL)f(z)|}{(1+2^jd(x,y))^\lambda}\dz\\
	&\lesi  \sum_{k=j-2}^{j+3}\int_X \f{1}{V(y,2^{-j})}(1+2^jd(y,z))^{-N}\f{|\varphi_k(\sL)f(z)|}{(1+2^jd(x,y))^\lambda(1+2^jd(y,z))^\lambda}\dz\\
	&\lesi  \sum_{k=j-2}^{j+3}\int_X \f{1}{V(y,2^{-j})}(1+2^jd(y,z))^{-N}\f{|\varphi_k(\sL)f(z)|}{(1+2^jd(x,z))^\lambda}\dz.\\
	\end{aligned}
	\end{equation}
	Using the fact that $2^k\sim 2^j$, we obtain 
	\begin{equation}\label{eq2s-pro1}
	\begin{aligned}
	\f{|\psi(s\sL)f(y)|}{(1+d(x,y)/s)^\lambda}&\lesi  \sum_{k=j-2}^{j+3}\int_X \f{1}{V(y,2^{-j})}(1+2^jd(y,z))^{-N}\f{|\varphi_k(\sL)f(z)|}{(1+2^kd(x,z))^\lambda}\dz\\
	&\lesi \sum_{k=j-2}^{j+3} \varphi^*_{k,\lambda}(\sL)f(x)\int_X \f{1}{V(y,2^{-j})}(1+2^jd(y,z))^{-N}\dz\\
	&\lesi \sum_{k=j-2}^{j+3} \varphi^*_{k,\lambda}(\sL)f(x),
	\end{aligned}
	\end{equation}
	where in the last inequality we use Lemma \ref{lem-elementary}.
	Taking the supremum over $y\in X$, we derive \eqref{eq-psistar vaphistar}.
\end{proof}
\bigskip
\begin{prop}
	\label{prop2-maximal functions}
Let $\psi$ be a partition of unity. Then for any $\lambda, s>0$ and $r>0$ we have:
\begin{equation}\label{eq1-proof prop2 maximal}
	\psi^*_\lambda(s\sL)f(x)\lesi \Big[\int_X \f{1}{V(z,s)}\f{|\psi(s\sL)f(z)|^r}{(1+d(x,z)/s)^{\lambda r}}\dz\Big]^{1/r},
	\end{equation}
	\begin{equation}\label{eq1-proof prop2 maximal2}
	\psi^*_{\lambda+\tilde n/r}(s\sL)f(x)\lesi \Big[\int_X \f{1}{V(x,s)}\f{|\psi(s\sL)f(z)|^r}{(1+d(x,z)/s)^{\lambda r}}\dz\Big]^{1/r}
	\end{equation}
	for all $f\in \SLm$ and $x\in X$.
\end{prop}
\begin{proof}
We start with the proof of \eqref{eq1-proof prop2 maximal2}. 
Similarly to \eqref{eq1-proof lem1}, we have
\begin{equation}\label{eq1-proof lem2}
\psi(s\sL)f=c_\psi\int_{s/4}^{4s}\psi(u\sL)\psi(s\sL)f\f{du}{u}.
\end{equation}
	If $r\in (0,1)$, then using Lemma \ref{lem1} and the fact $u\sim s$, we see that, for $\lambda>0$ and $x,y\in X$, 
	\begin{equation*}
	\begin{aligned}
	\f{|\psi(s\sL)f(y)|}{(1+d(x,y)/s)^\lambda}&\lesi  \int_X \f{1}{V(z,s)}{\Big(1+\f{d(y,z)}{s}\Big)^{-\lambda}}\f{|\psi(s\sL)f(z)|}{(1+d(x,y)/s)^\lambda}\dz\\
	&\lesi  \int_X \f{1}{V(z,s)}\f{|\psi(s\sL)f(z)|}{(1+d(x,z)/s)^\lambda}\dz\\
	&\lesi [\psi^*_{\lambda}(s\sL)f(x)]^{1-r}\int_X \f{1}{V(z,s)}\f{|\psi(s\sL)f(z)|^r}{(1+d(x,z)/s)^{\lambda r}}\dz.
	\end{aligned}
	\end{equation*}
By \eqref{eq- bound of varphi and f} there exists $\lambda_f >0$ such that $\psi^*_{\lambda}(s\sL)f(x)<\infty$ for all $\lambda\ge \lambda_f$, $s>0$ and $x\in X$. Using a bootstrap argument similarly  to the proof of Lemma~\ref{lem1-thm1-continuous cha} (see \eqref{eq-maximal inequality 2}), in which \eqref{eq-maximal inequality} is replaced by the above inequality, we deduce the required inequality \eqref{eq1-proof prop2 maximal2} in this case. We omit the details.

On the other hand, if $r\ge 1$, we similarly have
	\begin{equation*}
	\begin{aligned}
	\f{|\psi(s\sL)f(y)|}{(1+d(x,y)/s)^{\lambda+\tilde n/r}}&\lesi  \int_X \f{1}{V(z,s)}{\Big(1+\f{d(y,z)}{s}\Big)^{-n-\lambda-\tilde n/r}}\f{|\psi(s\sL)f(z)|}{(1+d(x,y)/s)^{\lambda+\tilde n/r}}\dz\\
	&\lesi  \int_X \f{1}{V(z,s)}\Big(1+\f{d(y,z)}{s}\Big)^{-n}\f{|\psi(s\sL)f(z)|}{(1+d(x,z)/s)^{\lambda+\tilde n/r}}\dz.
	\end{aligned}
	\end{equation*}
	We now apply H\"older's inequality, Lemma \ref{lem-elementary} and \eqref{doubling3} to deduce that
	\[
	\begin{aligned}
	\f{|\psi(s\sL)f(y)|}{(1+d(x,y)/s)^{\lambda+\tilde n/r}} &\lesi \Big[\int_X \f{1}{V(z,s)}\f{|\psi(s\sL)f(z)|^r}{(1+d(x,z)/s)^{\lambda r+\tilde n}}\dz\Big]^{1/r}\\
	&\lesi \Big[\int_X \f{1}{V(x,s)}\f{|\psi(s\sL)f(z)|^r}{(1+d(x,z)/s)^{\lambda r}}\dz\Big]^{1/r}.
	\end{aligned}
	\]
This implies the required inequality \eqref{eq1-proof prop2 maximal2} when $r\ge 1$.

The proof of \eqref{eq1-proof prop2 maximal} can be done in a similar fashion, but slightly simpler, where \eqref{doubling3} is not used.
\end{proof}

\bigskip
\begin{prop}
	\label{prop3-maximal functions}
Let $\psi$ be a partition of unity and $\varphi\in \mathscr{S}(\mathbb{R})$ be an even function such that $\varphi\neq 0$ on $[1/2,2]$. Then for any $\lambda>0, j\in \mathbb{Z}$ and $r>0$ we have	
	\begin{equation}\label{eq-prop3 maximal}
	|\psi_j(\sL)f(x)|\lesi \Big(\int_{2^{-j-2}}^{2^{-j+2}}|\varphi^*_{\lambda}(s\sL)f(x)|^r\f{ds}{s}\Big)^{1/r}
	\end{equation}
	for every $f\in \mathcal S'$.
\end{prop}

\begin{proof}
Since  $\varphi\neq 0$ on $[1/2,2]$, there exists $\phi\in \SR$ supported in $[1/2,2]$ so that
\[
c_{\varphi,\phi}=\int_0^\vc\varphi(\xi)\phi(\xi)\f{d\xi}{\xi}\neq 0.
\] 
It follows that, for each $j\in \mathbb{Z}$, 
\[
\begin{aligned}
\psi_j(\sL)f&=c_{\varphi,\phi}^{-1}\int_0^\vc\varphi(t\sL)\phi(t\sL)\psi_j(\sL)f\f{d\xi}{\xi}\\
&=c_{\varphi,\phi}^{-1}\int_{2^{j-2}}^{2^{j+2}}\varphi(t\sL)\phi(t\sL)\psi_j(\sL)f\f{dt}{t}.
\end{aligned}
\]
Hence, for any $r>0$,
\begin{equation}\label{eq1-proof prop3 maximal}
|\psi_j(\sL)f(x)|^r\lesi \sup_{t\in [2^{-j-2},2^{-j+2}]}|\varphi(t\sL)\phi(t\sL)\psi_j(\sL)f(x)|^r
\end{equation} 

Fix $t\in [2^{-j-2},2^{-j+2}]$. By Lemma \ref{lem1-thm1-continuous cha} and the fact that supp\,$\phi\subset [1/2,2]$ we have, for $\lambda>0$,
	\[
	\begin{aligned}
	|\phi(t\sL)\varphi(t\sL)&[\psi_j(\sL)f](x)|^r\\
	&\lesi \int_0^\vc\int_X\f{1}{V(x,s)}\Big(1+\f{d(x,y)}{s}\Big)^{-\lambda r}\Big(\f{s}{t}\wedge \f{t}{s}\Big)^{\lambda r}|\phi(s\sL)\varphi(s\sL)\psi_j(\sL)f(y)|^r\dy\f{ds}{s}\\
	&\lesi \int_{2^{-j-2}}^{2^{-j+2}}\int_X\f{1}{V(x,s)}\Big(1+\f{d(x,y)}{s}\Big)^{-\lambda r}\Big(\f{s}{t}\wedge \f{t}{s}\Big)^{\lambda r}|\phi(s\sL)\varphi(s\sL)\psi_j(\sL)f(y)|^r\dy\f{ds}{s}\\
	&\sim \int_{2^{-j-2}}^{2^{-j+2}}\int_X\f{1}{V(x,s)}\Big(1+\f{d(x,y)}{s}\Big)^{-\lambda r}|\phi(s\sL)\varphi(s\sL)\psi_j(\sL)f(y)|^r\dy\f{ds}{s}.	
	\end{aligned}
	\]
	 
	Note that $ s/4\le 2^{-j}\le 4s$ when $2^{-j-2} \leq s\leq 2^{-j+2}$ . Hence, applying Lemma \ref{lem2} we see that 
	\[
	\begin{aligned}
	\Big(1+\f{d(x,y)}{s}\Big)^{-\lambda}&|\phi(s\sL)\varphi(s\sL)\psi_j(\sL)f(y)|\\
	&\lesi \sup_{y\in X}|\varphi(s\sL)f(y)|\Big(1+\f{d(x,y)}{s}\Big)^{-\lambda}=: \varphi^*_{\lambda}(s\sL)f(x).
	\end{aligned}
	\]
	As a consequence,
	\[
	\begin{aligned}
	\sup_{t\in [2^{-j-2},2^{-j+2}]}|\phi(t\sL)\psi(t\sL)\psi_j(\sL)f(x)|^r\lesi \int_{2^{-j-2}}^{2^{-j+2}}|\varphi^*_{\lambda}(s\sL)f(x)|^r\f{ds}{s}.	
	\end{aligned}
	\]
	This and \eqref{eq1-proof prop3 maximal}  yield \eqref{eq-prop3 maximal}.
	
\end{proof}

\bigskip

\section{Besov and Triebel--Lizorkin spaces associated to $L$: Properties and Characterizations}

\subsection{Definitions of Besov and Triebel--Lizorkin spaces associated to $L$}
\begin{defn}
Let $\psi$ be  a partition of unity. For $0< p, q\leq \vc$, $\alpha\in \mathbb{R}$ and $w\in A_\vc$, we define the weighted homogeneous Besov space $\B^{\alpha, \psi, L}_{p,q,w}(X)$ as follows 
$$
\B^{\alpha, \psi, L}_{p,q,w}(X) = 
\Big\{f\in \mathcal{S}'_\vc:  \|f\|_{\B^{\alpha, \psi, L}_{p,q,w}(X)}<\vc\},
$$
where
\[
\|f\|_{\B^{\alpha, \psi, L}_{p,q,w}(X)}= \Big\{\sum_{j\in \mathbb{Z}}\left(2^{j\alpha}\|\psi_j(\sqrt{L})f\|_{p,w}\right)^q\Big\}^{1/q}.
\]

Similarly, for $0< p<\vc$, $0<q\le \vc$, $\alpha\in \mathbb{R}$ and $w\in A_\vc$, the weighted homogeneous Triebel-Lizorkin space $\F^{\alpha, \psi, L}_{p,q,w}(X)$ is defined by 
$$
\F^{\alpha, \psi, L}_{p,q,w}(X) = 
\Big\{f\in \mathcal{S}'_\vc:  \|f\|_{\F^{\alpha, \psi, L}_{p,q,w}(X)}<\vc\},
$$
where
\[
\|f\|_{\F^{\alpha, \psi, L}_{p,q,w}(X)}= \Big\|\Big[\sum_{j\in \mathbb{Z}}(2^{j\alpha}|\psi_j(\sqrt{L})f|)^q\Big]^{1/q}\Big\|_{p,w}.
\]
\end{defn}

We now claim that $\psi_j(\sqrt{L})f = 0$ for all $j\in \mathbb{Z}$ if and only if $f\in \mathscr{P}$. Indeed, since $\mathscr{P}=\cup_{m\in \mathbb{N}}\mathscr{P}_m$, it is obvious that if $f\in \mathscr{P}$ then $\psi_j(\sqrt{L})f = 0$ for all $j$.

For the reverse direction, we assume that $\psi_j(\sqrt{L})f = 0$ for all $j\in \mathbb{Z}$. Since $f\in \mathcal S'_\vc =\mathcal S'/\mathscr P$, we have $f = f_{\mathcal S'} +f_{\mathscr P}$ for some  $f_{\mathcal S'} \in \mathcal S'$ and $f_{\mathscr P}\in \mathscr P$. It follows that, for every $j\in \mathbb{Z}$,
\[
\psi_j(\sqrt{L})f = \psi_j(\sqrt{L})f_{\mathcal S'}+\psi_j(\sqrt{L})f_{\mathscr P}=\psi_j(\sqrt{L})f_{\mathcal S'}.
\]
Therefore, 
\[
\sum_{j\in \mathbb{Z}}\psi_j(\sqrt{L})f = \sum_{j\in \mathbb{Z}}\psi_j(\sqrt{L})f_{\mathcal S'}
\]
which implies
\[
\sum_{j\in \mathbb{Z}}\psi_j(\sqrt{L})f_{\mathcal S'} =0.
\]
On the other hand, arguing similarly to the proof of Proposition \ref{prop-Calderon2}, we can find $\rho\in \mathscr P$ such that 
\[
\sum_{j\in \mathbb{Z}}\psi_j(\sqrt{L})f_{\mathcal S'} =f_{\mathcal S'} -\rho \ \ \ \text{in $\mathcal S'$}.
\]
From the last two identities we obtain $f_{\mathcal S'} =\rho\in \mathscr P$ and hence $f\in \mathscr{P}$. The statement is proved.

Therefore, each of the above spaces is a quasi-normed linear space (normed linear space when $p,q\ge 1$).

Note that like the classical case, the Triebel-Lizorkin spaces for $p=\vc$ would be defined in a different way. See Subsection \ref{sub-Fvc}.
\bigskip

From Proposition \ref{prop1-maximal function} we have:
\begin{prop}
	\label{prop1-thm1}
	Let $\psi, \varphi$ be  partitions of unity and assume ${\rm supp\,}\psi, {\rm supp\,}\varphi \subset [1/2,2]$.  Let $w\in A_\vc$, $\alpha\in \mathbb{R}$ and $\lambda>0$.
	Then the following norm equivalence holds: For  all $f\in \mathcal{S}'_\infty$
	\begin{enumerate}[{\rm (a)}]
		\item 
		 $$\displaystyle \Big\{\sum_{j\in \mathbb{Z}}\left(2^{j\alpha}\|\psi^*_{j,\lambda}(\sqrt{L})f\|_{p,w}\right)^q\Big\}^{1/q}\sim \Big\{\sum_{j\in \mathbb{Z}}\left(2^{j\alpha}\|\varphi^*_{j,\lambda}(\sqrt{L})f\|_{p,w}\right)^q\Big\}^{1/q},\\\ 0< p, q\le \vc;$$

		\item 
		$$\displaystyle \Big\|\Big[\sum_{j\in \mathbb{Z}}(2^{j\alpha}|\psi^*_{j,\lambda}(\sqrt{L})f|)^q\Big]^{1/q}\Big\|_{p,w}\sim \Big\|\Big[\sum_{j\in \mathbb{Z}}(2^{j\alpha}|\varphi^*_{j,\lambda}(\sqrt{L})f|)^q\Big]^{1/q}\Big\|_{p,w},\\\ 0< p<\vc, 0<q\le \vc.$$
	\end{enumerate}
\end{prop}

We next prove the following result.
\begin{prop}
	\label{prop2-thm1}
	Let $\psi$ be a partition of unity. Then we have:
	\begin{enumerate}[{\rm (a)}]
		\item For $0< p, q\le \vc$, $\alpha\in \mathbb{R}$ and $\lambda>nq_w/p$,
		$$\displaystyle \Big\{\sum_{j\in \mathbb{Z}}\left(2^{j\alpha}\|\psi^*_{j,\lambda}(\sqrt{L})f\|_{p,w}\right)^q\Big\}^{1/q}\sim \|f\|_{\B^{\alpha, \psi, L}_{p,q,w}(X)}.
		$$
		
		\item For $0< p<\vc$, $0<q\le \vc$, $\alpha\in \mathbb{R}$ and $\lambda>\max\{n/q, nq_w/p\}$,
		$$\displaystyle \Big\|\Big[\sum_{j\in \mathbb{Z}}(2^{j\alpha}|\psi^*_{j,\lambda}(\sqrt{L})f|)^q\Big]^{1/q}\Big\|_{p,w}\sim \|f\|_{\F^{\alpha, \psi, L}_{p,q,w}(X)}.$$
	\end{enumerate}
\end{prop}
\begin{proof}
	We will provide the proof for (b), since the proof of  (a) is similar and even easier.
	
	Observe that from Proposition \ref{prop1-thm1} it suffices to prove that 
	\begin{equation}
	\label{eq1-proof prop2 thm1}
	\Big\|\Big[\sum_{j\in \mathbb{Z}}(2^{j\alpha}|\psi^*_{j,\lambda}(\sqrt{L})f|)^q\Big]^{1/q}\Big\|_{p,w}\lesi \Big\|\Big[\sum_{j\in \mathbb{Z}}(2^{j\alpha}|\psi_{j}(\sqrt{L})f|)^q\Big]^{1/q}\Big\|_{p,w}.
	\end{equation}
	Indeed, taking $r<\min\{p,q, p/q_w\}=\min\{q, p/q_w\}$ so that $\lambda>n/r$ and $w\in A_{p/r}$, then applying \eqref{eq1-proof prop2 maximal} we have
	\begin{equation*}
	\begin{aligned}
	\psi^*_{j,\lambda}(\sL)f(x)&\lesi \Big[\int_X \f{1}{V(z,2^{-j})}\f{|\psi_j(\sL)f(z)|^r}{(1+2^jd(x,z))^{\lambda r}}\dz\Big]^{1/r}\\
	&\lesi \mathcal{M}_r(|\psi_j(\sL)f)(x),
	\end{aligned}
	\end{equation*}
	where we use Lemma \ref{lem-elementary} in the last inequality.
	The desired inequality \eqref{eq1-proof prop2 thm1} then follows by using the weighted Fefferman-Stein maximal inequality \eqref{FSIn}.
	\end{proof}

\bigskip
As a consequence of Proposition \ref{prop1-thm1} and Proposition \ref{prop2-thm1}, we obtain the following theorem.
\begin{thm}
	\label{thm1}
	Let $\psi$ and $\varphi$ be partitions of unity. Then following statements hold:
	\begin{enumerate}[(a)]
		\item  The spaces $\B^{\alpha, \psi, L}_{p,q,w}(X)$ and $\B^{\alpha, \varphi, L}_{p,q,w}(X)$ coincide with equivalent norms for all $0< p, q\le \vc$, $\alpha\in \mathbb{R}$ and $w\in A_\vc$.
		
		\item  The spaces $\F^{\alpha, \psi, L}_{p,q,w}(X)$ and $\F^{\alpha, \varphi, L}_{p,q,w}(X)$ coincide with equivalent norms for all $0< p< \vc$, $0<q\le \vc$, $\alpha\in \mathbb{R}$ and $w\in A_\vc$.		
	\end{enumerate}
For this reason, we define the spaces  $\B^{\alpha, L}_{p,q,w}(X)$ and  $\F^{\alpha, L}_{p,q,w}(X)$ to be any spaces  $\B^{\alpha, \psi, L}_{p,q,w}(X)$ and  $\B^{\alpha, \psi, L}_{p,q,w}(X)$ with any partitions of unity $\psi$, respectively. 
\end{thm}

 It is routine to show that the spaces $\B^{\alpha, L}_{p,q,w}(X)$ and  $\F^{\alpha, L}_{p,q,w}(X)$ are complete, and each is continously embedded into $\mathcal{S}'_\vc$. We omit the details.

\subsection{Continuous characterizations by functions with compact supports }
In this section, we will prove continuous characterizations for new Besov and Triebel--Lizorkin spaces including those using Lusin functions and the Littlewood-Paley functions. 
\begin{thm}
	\label{thm1-continuouscharacter}
	Let $\psi$ be a partition of unity. Then we have:
	\begin{enumerate}[(a)]
		\item For $w\in A_\vc$, $0< p, q\le \vc$, $\alpha\in \mathbb{R}$ and $\lambda>nq_w/p$,
		\begin{equation}\label{eq-B continuous cha}
		\|f\|_{\B_{p,q,w}^{\alpha,L}}\sim \Big(\int_0^\vc \Big[t^{-\alpha}\|\psi(t\sL)f\|_{p,w}\Big]^q\f{dt}{t}\Big)^{1/q}\sim \Big(\int_0^\vc \Big[t^{-\alpha}\|\psi^*_\lambda(t\sL)f\|_{p,w}\Big]^q\f{dt}{t}\Big)^{1/q}
		\end{equation}
		for all $f\in \mathcal{S}'_\vc$.
		\item For $w\in A_\vc$, $0< p< \vc$, $0<q\le \vc$, $\alpha\in \mathbb{R}$, and $\lambda>\max\{n/q, nq_w/p\}$,
		\begin{equation}\label{eq-F continuous cha}
		\|f\|_{\F_{p,q,w}^{\alpha,L}}\sim \Big\|\Big(\int_0^\vc \Big[t^{-\alpha}|\psi(t\sL)f|\Big]^q\f{dt}{t}\Big)\Big\|_{p,w}\sim \Big\|\Big(\int_0^\vc \Big[t^{-\alpha}\psi^*_\lambda(t\sL)f\Big]^q\f{dt}{t}\Big)\Big\|_{p,w}
	   \end{equation}
	   for all $f\in \mathcal{S}'_\vc$.
	\end{enumerate}
\end{thm}

\begin{proof} We give the proof of (b) only, since the proof of (a) can be done in the same manner.\\

We divide the proof of (b) into three steps.
	
	\medskip
	
	\noindent{\it \textbf{Step 1:}} We first prove that 
	\begin{equation}
	\label{first ineq continuous}
	\Big\|\Big(\int_0^\vc \Big[t^{-\alpha}|\psi(t\sL)f|\Big]^q\f{dt}{t}\Big)^{1/q}\Big\|_{p,w}\lesi \|f\|_{\F_{p,q,w}^{\alpha,L}}.
	\end{equation}
	Indeed, for $t\in [2^{-j-1},2^{-j}]$ with $j\in \mathbb{Z}$, from \eqref{eq-psistar vaphistar}, we see that
		\[
	\sup_{t\in [2^{-j-1},2^{-j}]}|\psi(t\sL)f(x)|\lesi \sum_{k=j-2}^{j+3}\psi^*_{k,\lambda}(\sL)f(x).
	\]
	Hence, the estimate \eqref{first ineq continuous} follows from the above inequality and Proposition \ref{prop1-thm1}.
	
	\medskip

	\noindent{\it \textbf{Step 2:}} We next show that
	\begin{equation}
	\label{reverse ineq continuous}
	\|f\|_{\F_{p,q,w}^{\alpha,L}}\lesi \Big\|\Big(\int_0^\vc \Big[t^{-\alpha}\psi^*_\lambda(t\sL)f\Big]^q\f{dt}{t}\Big)^{1/q}\Big\|_{p,w}.
	\end{equation}
	Using Proposition \ref{prop3-maximal functions} we have
	\begin{equation*}%\label{eq-key estimate}
	|\psi_j(\sL)f(x)|\lesi \Big(\int_{2^{-j-2}}^{2^{-j+2}}|\psi^*_{\lambda}(s\sL)f(x)|^q\f{ds}{s}\Big)^{1/q}.
	\end{equation*}
	This implies the desired inequality \eqref{reverse ineq continuous}.
	 
	\medskip

	\noindent{\it \textbf{Step 3:}} To complete the proof of the theorem, we need to prove that
	\begin{equation}\label{psi*-psi inequality}
	\Big\|\Big(\int_0^\vc \Big[t^{-\alpha}\psi^*_\lambda(t\sL)f\Big]^q\f{dt}{t}\Big)^{1/q}\Big\|_{p,w}\lesi \Big\|\Big(\int_0^\vc \Big[t^{-\alpha}|\psi(t\sL)f|\Big]^q\f{dt}{t}\Big)^{1/q}\Big\|_{p,w}.
	\end{equation}
	To see this, taking $r<\min\{p,q, p/q_w\}=\min\{q, p/q_w\}$ so that $\lambda>n/r$ and $w\in A_{p/r}$, then applying \eqref{eq1-proof prop2 maximal} we have, for all $t\in [1,2]$,
	\begin{equation*}
	\begin{aligned}
	|\psi^*_{\lambda}(2^{-j}t\sL)f(x)|^r&\lesi \int_X \f{1}{V(x,2^{-j})}\f{|\psi(2^{-j}t\sL)f(z)|^r}{(1+2^jd(x,z))^{\lambda r}}\dz.
	\end{aligned}
	\end{equation*}
	
		Since $r<q$, we use Minkowski's inequality to get the bound:
	\[
	\Big(\int_{1}^2|\psi_{\lambda}^{*}(2^{-j}t\sL)f(x)|^{q}\f{dt}{t}\Big)^{r/q}\lesi \int_X \f{1}{V(z,2^{-j})}\f{\Big(\int_{1}^2|\psi(2^{-j}t\sL)f(z)|^q\f{dt}{t}\Big)^{r/q}}{(1+2^{j}d(x,z))^{\lambda r}}\dz.
	\]
	By a change of variables,
	\[
	\Big[\int_{2^{-j}}^{2^{-j+1}}(t^{-\alpha}|\psi_{\lambda}^{*}(t\sL)f(x)|)^{q}\f{dt}{t}\Big]^{r/q}\lesi \int_X \f{1}{V(z,2^{-j})}\f{\Big[\int_{2^{-j}}^{2^{-j+1}}(t^{-\alpha}|\psi(t\sL)f(z)|)^q\f{dt}{t}\Big]^{r/q}}{(1+2^{j}d(x,z))^{\lambda r}}\dz.
	\]
	Hence, applying Lemma \ref{lem-elementary} we obtain
	\[
	\Big(\int_{2^{-j}}^{2^{-j+1}}|\psi_{\lambda}^{*}(t\sL)f(x)|^{q}\f{dt}{t}\Big)^{1/q}\lesi  \mathcal{M}_r\Big[\Big(\int_{2^{-j}}^{2^{-j+1}}|\psi(t\sL)f|^q\f{dt}{t}\Big)^{1/q}\Big](x)
	\]
 as long as $\lambda r>n$. Using \eqref{FSIn}, we deduced the required estimate \eqref{psi*-psi inequality}.

The proof of our theorem is thus complete.
\end{proof}
\bigskip

\subsection{Continuous characterizations by functions in $\mathscr{S}_m(\mathbb{R})$.}
For each $m\in \mathbb{N}$ we denote by $\mathscr{S}_m(\mathbb{R})$ the set of all even functions $\varphi\in \mathscr{S}(\mathbb{R})$ such that $\varphi(\xi)=\xi^{2m}\phi(\xi)$ for some $\phi\in \mathscr{S}(\mathbb{R})$, and $\varphi(\xi)\neq 0$ on $(-2,-1/2)\cup(1/2,2)$.

We have the following characterization for  the new  Besov and Triebel-Lizorkin spaces via functions in $\mathscr{S}_m(\mathbb{R})$.
\begin{thm}
	\label{thm2}
	Let $w\in A_\vc$, $\alpha\in \mathbb{R}$, $m>\alpha/2$ and let  $\varphi\in \mathscr{S}_m(\mathbb{R})$. Then the following statements hold:
	\begin{enumerate}[(a)]
		\item For $0< p, q\le \vc$, $\lambda>nq_w/p$, and $f\in \mathcal S'$ there exists $\rho\in \mathscr P$ so that
		\begin{equation}
		\label{eq-Besov equivalent norms square functions}
		\Big(\int_0^\vc\Big[t^{-\alpha}\|\varphi^*_\lambda(t\sL)(f-\rho)\|_{p,w}\Big]^q\f{dt}{t}\Big)^{1/q}\lesi \|f\|_{\B^{\alpha, L}_{p,q,w}}\lesi \Big(\int_0^\vc\Big[t^{-\alpha}\|\varphi(t\sL)f\|_{p,w}\Big]^q\f{dt}{t}\Big)^{1/q}.
		\end{equation}

		\item For $0< p< \vc$, $0<q\le \vc$, $\lambda>\max\{n/q, nq_w/p\}$, and $f\in \mathcal S'$ there exists $\rho\in \mathscr P$ so that
		\begin{equation}
		\label{eq-TL equivalent norms square functions}
		\Big\| \Big(\int_0^\vc\Big[t^{-\alpha}\varphi^*_\lambda(t\sL)(f-\rho)\Big]^q\f{dt}{t}\Big)^{1/q}\Big\|_{p,w}\lesi \|f\|_{\F^{\alpha, L}_{p,q,w}}\lesi \Big\| \Big(\int_0^\vc\Big[t^{-\alpha}\varphi(t\sL)f\Big]^q\f{dt}{t}\Big)^{1/q}\Big\|_{p,w}.
		\end{equation}
	\end{enumerate}
\end{thm}

\begin{proof}
	In comparison with the proof of Theorem \ref{thm1-continuouscharacter}, the proof of Theorem \ref{thm2} is much more difficult, due to the lack of compact support condition for the functions in $\mathscr{S}_m(\mathbb{R})$. 
	A significant difference  with Theorem~\ref{thm1-continuouscharacter} is that the results are formulated for $f\in \mathcal{S}'$ (as opposed to $f\in\mathcal{S}'_\infty$). The reason being that, as $K_{\varphi(t\sL)}(x,\cdot)$ may not be in $\mathcal{S}_\infty$ for $\varphi\in \mathscr{S}_m(\mathbb{R})$, $\varphi(t\sL)f$ may not be defined when $f\in \mathcal{S}'_\infty$. Although each $f\in \mathcal{S}'_\infty$ has an extension to an element in $\mathcal{S}'$, the extension is not unique; that is, $\varphi(t\sL)f$ will depend on the chosen representative of $f$. Our theorem says that there exists a representative so that the left-hand side inequality in \eqref{eq-Besov equivalent norms square functions} (or in \eqref{eq-TL equivalent norms square functions}) holds.
	
	We will prove \eqref{eq-TL equivalent norms square functions} only, because the proof of \eqref{eq-Besov equivalent norms square functions} can be done in the same manner. We divide the proof into a number of steps.
	
	\medskip

\noindent{\it \textbf{Step 1:}}	Let $\psi$ be a partition of unity. From Proposition \ref{prop-Calderon2}, 
	there exists $\rho\in \mathscr P$ so that 
	\begin{equation*}
	f-\rho=c_\psi\int_0^\vc \psi(s\sL)f\f{ds}{s}	 \quad
	\text{in $\mathcal S'$}.
	\end{equation*} 
We will show that 
	\begin{equation}
	\label{eq- first direc square}
	\Big\| \Big(\int_0^\vc\Big[t^{-\alpha}\varphi^*_\lambda(t\sL)(f-\rho)\Big]^q\f{dt}{t}\Big)^{1/q}\Big\|_{p,w}\lesi \|f\|_{\F^{\alpha, L}_{p,q,w}}.
	\end{equation}
	 
	First note that the Calder\'{o}n reproducing formula above implies the pointwise representation
	\begin{equation}\label{eq- f Svc S}
	\varphi(t\sL)(f-\rho)=c_\psi\int_0^\vc \varphi(t\sL)\psi(s\sL)f\f{ds}{s}	
	\end{equation}
	for all $t>0$.
	
	Let $\lambda>0$, $t\in [2^{-\nu-1},2^{-\nu}]$ for some $\nu \in \mathbb{Z}$ and $M>m+\lambda/2$.  For simplicity of writing we let $c_\psi = 1$. We then have
	\[
	\begin{aligned}
	\varphi(t\sL)(f-\rho)&=\int_0^\vc \psi(s\sL)\varphi(t\sL)f\f{ds}{s}\\
	&=\sum_{j\ge \nu}\int_{2^{-j-1}}^{2^{-j}} \psi(s\sL)\varphi(t\sL)f\f{ds}{s}+\sum_{j< \nu}\int_{2^{-j-1}}^{2^{-j}} \psi(s\sL)\varphi(t\sL)f\f{ds}{s}\\
	&=\sum_{j\ge \nu}\int_{2^{-j-1}}^{2^{-j}}\Big(
	\f{s}{t}\Big)^{2M} (s^2L)^{-M}\psi(s\sL)(t^2L)^{M}\varphi(t\sL)f\f{ds}{s}\\
	& \ \ \ +\sum_{j< \nu}\int_{2^{-j-1}}^{2^{-j}} \Big(
	\f{t}{s}\Big)^{2m} (s^2L)^{m}\psi(s\sL)(t^2L)^{-m}\varphi(t\sL)f\f{ds}{s}.
	\end{aligned}
	\]
	Setting $\psi_{s,M}(x)=x^{-2M}\psi(x)$ and $\tilde\psi_{s,m}(x)=x^{2m}\psi(x)$, we rewrite the above as
	\[
	\begin{aligned}
	\varphi(t\sL)(f-\rho)&=\sum_{j\ge \nu}\int_{2^{-j-1}}^{2^{-j}}\Big(
	\f{s}{t}\Big)^{2M} (t^2L)^{M}\varphi(t\sL)\psi_{s,M}(\sL)f\f{ds}{s}\\
	& \ \ \ \ +\sum_{j< \nu}\int_{2^{-j-1}}^{2^{-j}} \Big(
	\f{t}{s}\Big)^{2m} (t^2L)^{-m}\varphi(t\sL)\tilde\psi_{s,m}(\sL)f\f{ds}{s}.
	\end{aligned}
	\]

	Since $\xi^{2M}\varphi(\xi)$ is an even function in $\mathscr{S}(\mathbb{R})$, we use Lemma \ref{lem1} to deduce that 
	\[
	\begin{aligned}
	|(t^2L)^{M}\varphi(t\sL)\psi_{s,N}(\sL)f(y)|\lesi \int_X\f{1}{V(y,t)}\Big(1+\f{d(y,z)}{t}\Big)^{-\lambda-N}|\psi_{s,N}(\sL)f(z)|\dz,
	\end{aligned}
	\]
	where $N>n$.
	
	It follows that
	\[
	\begin{aligned}
	\f{|(t^2L)^{M}\varphi(t\sL)\psi_{s,N}(\sL)f(y)|}{(1+d(x,y)/t)^\lambda}\lesi \int_X\f{1}{V(y,t)}\Big(1+\f{d(y,z)}{t}\Big)^{-N}\f{|\psi_{s,N}(\sL)f(z)|}{(1+d(x,z)/t)^\lambda}\dz
	\end{aligned}
	\]
	for all $x,y\in X$.
	
	Hence, for $j\ge \nu$, $t\in [2^{-\nu-1},2^{-\nu}]$ and $s\in [2^{-j-1},2^{-j}]$ we have
	\[
	\begin{aligned}
	\f{|(t^2L)^{M}\varphi(t\sL)\psi_{s,N}(\sL)f(y)|}{(1+d(x,y)/t)^\lambda}&\lesi 2^{\lambda(j-\nu)}\psi^*_{s,N,\lambda}(\sL)f(x)\int_X\f{1}{V(y,t)}\Big(1+\f{d(y,z)}{t}\Big)^{-N}\dy\\
	&\lesi 2^{\lambda(j-\nu)}\psi^*_{s,N,\lambda}(\sL)f(x).
	\end{aligned}
	\]
	Since $\psi \in \mathscr{S}_m(\mathbb{R})$, $x^{-2m}\psi(x)\in \mathscr{S}(\mathbb{R})$. Using Lemma \ref{lem1} and an  argument similar to the above estimate for $\psi_{s,M}$, we obtain, for $j< \nu$, $t\in [2^{-\nu-1},2^{-\nu}]$ and $s\in [2^{-j-1},2^{-j}]$,
	\[
	\f{|(t^2L)^{-m}\varphi(t\sL)\tilde\psi_{s,m}(\sL)f(y)|}{(1+d(x,y)/t)^\lambda}\lesi \tilde{\psi}^*_{s,m,\lambda}(\sL)f(x).
	\]
	
	Combining the above two estimates we deduce that 
	\[
	\begin{aligned}
	|\varphi^*_\lambda(t\sL)(f-\rho)|&\le\sum_{j\ge \nu}2^{-(j-\nu)(2M-\lambda)}\sup_{s\in (2^{-j-1},2^{-j}]}\psi^*_{s,N,\lambda}(\sL)f\\
	& \ \ \ \ + \sum_{j< \nu}2^{-2m(\nu-j)}\sup_{s\in (2^{-j-1},2^{-j}]}\tilde{\psi}^*_{s,m,\lambda}(\sL)f.
	\end{aligned}
	\]
	This, along with Proposition \ref{prop1-maximal function}, implies that
	\begin{equation}\label{eq1-heat kernel charac}
	\begin{aligned}
	|\varphi^*_\lambda(t\sL)(f-\rho)|&\lesi \sum_{j\ge \nu -1}2^{-(2M-\lambda)(j-\nu)}\psi^*_{j,\lambda}(\sL)f +\sum_{j< \nu +3}2^{-2m(\nu-j)}\psi^*_{j,\lambda}(\sL)f\\
	&\lesi \sum_{j\in \mathbb{Z}}2^{-2m|\nu-j|} \psi^*_{j,\lambda}(\sL)f
	\end{aligned}
	\end{equation}
	for all $t\in [2^{-\nu-1},2^{-\nu}]$ and $M>m+\lambda/2$.
	
	Therefore, if $q\leq 1$, we have
	$$
	\begin{aligned}
	\int_{2^{-\nu-1}}^{2^{-\nu}} (t^{-\alpha}|\varphi^*_\lambda(t\sL)(f-\rho)|)^q\f{dt}{t}
	&\lesi \sum_{j\in \mathbb{Z}}2^{-q(2m-\alpha)|\nu-j|} (2^{j\alpha}\psi^*_{j,\lambda}(\sL)f)^q.
	\end{aligned}
	$$
	It follows that 
	$$
	\begin{aligned}
	\Big(\int_{0}^{\vc} (t^{-\alpha}|\varphi^*_\lambda(t\sL)(f-\rho)|)^q\f{dt}{t}\Big)^{1/q}
	&\lesi \Big(\sum_{\nu\in \mathbb{Z}}\sum_{j\in \mathbb{Z}}2^{-q(2m-\alpha)|\nu-j|} (2^{j\alpha}\psi^*_{j,\lambda}(\sL)f)^q\Big)^{1/q}\\
	&\lesi \Big(\sum_{j\in \mathbb{Z}}(2^{j\alpha}\psi^*_{j,\lambda}(\sL)f)^q\Big)^{1/q},
	\end{aligned}
	$$
	which, along with Proposition \ref{prop2-thm1}, yields \eqref{eq- first direc square}.
	
	On the other hand, if $q>1$, then we use Young's inequality to get the bound:
	$$
	\begin{aligned}
	\Big(\int_{0}^{\vc} (t^{-\alpha}|\varphi^*_\lambda(t\sL)(f-\rho)|)^q\f{dt}{t}\Big)^{1/q}
	&\lesi \Big(\sum_{\nu\in \mathbb{Z}}\Big[\sum_{j\in \mathbb{Z}}2^{-(2m-\alpha)|\nu-j|} 2^{j\alpha}\psi^*_{j,\lambda}(\sL)f\Big]^q\Big)^{1/q}\\
	&\lesi \Big(\sum_{j\in \mathbb{Z}}(2^{j\alpha}\psi^*_{j,\lambda}(\sL)f)^q\Big)^{1/q}.
	\end{aligned}
	$$
	Hence, \eqref{eq- first direc square} follows from this and  Proposition \ref{prop2-thm1}.
	
	\bigskip

	\noindent{\it \textbf{Step 2:}} We next show that
	\[
	\|f\|_{\F^{\alpha, L}_{p,q,w}}\lesi \Big\| \Big(\int_0^\vc\Big[t^{-\alpha}\varphi^*_\lambda(t\sL)f\Big]^q\f{dt}{t}\Big)^{1/q}\Big\|_{p,w}.
	\]
	Let $\psi$ be a partition of unity. By Proposition \ref{prop3-maximal functions}, 
	\begin{equation*} 
		|\psi_j(\sL)f(x)|^q\lesi \int_{2^{-j-2}}^{2^{-j+2}}|\varphi^*_\lambda(t\sL)f(x)|^q\f{ds}{s}.
	\end{equation*}
	Hence the desired inequality follows.
	 
	\bigskip
	
	\noindent{\it \textbf{Step 3:}} This is the most elaborate step in the proof, where we will prove that
	\[
	\Big\| \Big(\int_0^\vc\Big[t^{-\alpha}\varphi^*_\lambda(t\sL)f\Big]^q\f{dt}{t}\Big)^{1/q}\Big\|_{p,w}\lesi \Big\| \Big(\int_0^\vc\Big[t^{-\alpha}\varphi(t\sL)f\Big]^q\f{dt}{t}\Big)^{1/q}\Big\|_{p,w}.
	\]
	To this end, we apply Lemma \ref{lem-DKP} to find  $a,b,c>0$ and even functions $\phi,\eta\in \mathscr{S}(\mathbb{R})$ with ${\rm supp}\,\phi \subset [-a,a]$, ${\rm supp}\,\eta \subset [-c,-b]\cup[b,c]$, and  so that for every $\ell\in \mathbb{Z}$, $t\in [1,2]$ and $f\in \mathcal S'$ we have
	$$
	\begin{aligned}
	f=&\phi(2^{-\ell}t\sqrt{L})f+\sum_{k\geq 1}\varphi(2^{-k-\ell}t\sL)\eta(2^{-k-\ell}t\sL)f \quad \text{in $\mathcal S'$}.
	\end{aligned}
	$$
	This implies that 
	$$
	\begin{aligned}
	\varphi(2^{-\ell}t\sL)f = \ &\phi(2^{-\ell}t\sqrt{L})\varphi(2^{-\ell}t\sL)f+\sum_{k\geq 1}\varphi(2^{-k-\ell}t\sL)\eta(2^{-k-\ell}t\sL)\varphi(2^{-\ell}t\sL)f
	\end{aligned}
	$$
pointwise.	

Let $\lambda,A>0$. Put $M=\lambda+A$. Using the above pointwise representation together with Lemma~\ref{lem1}, we get
	\[
	\begin{aligned}
	|\varphi(2^{-\ell}t\sL)f(y)|\lesi&\int_X\f{1}{V(z,2^{-\ell})}(1+2^\ell d(y,z))^{-\lambda}|\varphi(2^{-\ell}t\sL)f(z)|\dz\\
	&+\sum_{k\geq 1}2^{-kM}\int_X\f{1}{V(z,2^{-\ell})}(1+2^\ell d(y,z))^{-\lambda}|\varphi(2^{-\ell-k}t\sL)f(z)|\dz\\
	&= \sum_{k\geq 0}2^{-kM}\int_X\f{1}{V(z,2^{-\ell})}(1+2^\ell d(y,z))^{-\lambda}|\varphi(2^{-\ell-k}t\sL)f(z)|\dz\\
	&= 2^{(\ell-j)M}\sum_{k\geq \ell}2^{(j-k)M}\int_X\f{1}{V(z,2^{-\ell})}(1+2^\ell d(y,z))^{-\lambda}|\varphi(2^{-k}t\sL)f(z)|\dz\\
	&\leq 2^{(\ell-j)M}\sum_{k\geq \ell}2^{(j-k)M}\int_X\f{1}{V(z,2^{-k})}\frac{|\varphi(2^{-k}t\sL)f(z)|}
	{(1+2^\ell d(y,z))^{\lambda}}\dz,
	\end{aligned}
	\]
where $j,\ell \in \mathbb{Z}$ and $\ell \geq j$.	
It follows that, for any $0<r\le 1$,

	\begin{equation}\label{eq2-proof thm2}
	\begin{aligned}
	2^{(j-\ell)M}\f{|\varphi(2^{-\ell}t\sL)f(y)|}{(1+2^j d(x,y))^\lambda}\lesi&
	\sum_{k\geq \ell}2^{(j-k)M}\int_X\f{1}{V(z,2^{-k})}\frac{|\varphi(2^{-k}t\sL)f(z)|}{(1+2^\ell d(y,z))^{\lambda}(1+2^jd(x,y))^\lambda}\dz\\
	\le& \sum_{k\geq \ell}2^{(j-k)M}\int_X\f{1}{V(z,2^{-k})}\frac{|\varphi(2^{-k}t\sL)f(z)|}{(1+2^j d(y,z))^{\lambda}(1+2^jd(x,y))^\lambda}\dz\\
	\le& \sum_{k\geq \ell}2^{(j-k)M}\int_X\f{1}{V(z,2^{-k})}\frac{|\varphi(2^{-k}t\sL)f(z)|}{(1+2^j d(x,z))^{\lambda}}\dz\\
	\le& \varphi_{A,\lambda}^{**}(2^{-j}t\sL)f(x)^{1-r}
\sum_{k\geq j}2^{(j-k)Mr}\int_X\f{1}{V(z,2^{-k})}\frac{|\varphi(2^{-k}t\sL)f(z)|^r}{(1+2^j d(x,z))^{\lambda r}}\dz,
	\end{aligned}
	\end{equation}
    where, for each $j\in \mathbb{Z}$ and $t\in [1,2] $, we define the   Peetre-type maximal function by  
	$$
	\varphi_{A,\lambda}^{**}(2^{-j}t\sL)f(x)=\sup_{k\ge j}\sup_{y\in X}2^{(j-k)M}\f{|\varphi(2^{-k}t\sL)f(y)|}{(1+2^{j}d(x,y))^\lambda}.
	$$

By \eqref{eq- bound of varphi and f}, $\varphi_{A,\lambda}^{**}(2^{-j}t\sL)f(x) < \infty$ , for all sufficiently large $\lambda$ (depending on $f$), all $x\in X$ and $A>0$. Hence, for any such $\lambda$, by taking the supremum of the LHS of \eqref{eq2-proof thm2} over $\ell\ge j$ and $y\in X$, and using the obvious inequality $(1+2^jd(x,z))\ge 2^{j-k}(1+2^kd(x,z))$ on the RHS, we obtain 
\begin{equation}\label{eq- vrphi **}
\begin{aligned}
\varphi_{A,\lambda}^{**}(2^{-j}t\sL)f(x)^r &\lesi
\sum_{k\geq j}2^{(j-k)(M-\lambda)r}\int_X\f{1}{V(z,2^{-k})}\frac{|\varphi(2^{-k}t\sL)f(z)|^r}{(1+2^k d(x,z))^{\lambda r}}\dz\\
&=\sum_{k\geq j}2^{(j-k)Ar}\int_X\f{1}{V(z,2^{-k})}\frac{|\varphi(2^{-k}t\sL)f(z)|^r}{(1+2^k d(x,z))^{\lambda r}}\dz \quad 
\text{(as $M-\lambda=A$)}.
\end{aligned}
\end{equation}
Since clearly 
$|\varphi(\cdots)f| \le \varphi_{A,\lambda}^{**}(\cdots)f$, the above implies that
\[
|\varphi(2^{-\ell}t\sL)f(y)|^r \lesi \sum_{k\geq \ell}2^{(\ell-k)Ar}\int_X\f{1}{V(z,2^{-k})}\frac{|\varphi(2^{-k}t\sL)f(z)|^r}{(1+2^k d(x,z))^{\lambda r}}\dz 
\]
for all sufficiently large $\lambda$,  $A>0$, $y\in X$ and $\ell \in \mathbb{Z}$. But the right-hand side of the above inequality increases as $\lambda$ decreases, and hence this inequality holds for all $\lambda>0$ and $A>0$, with the inequality constant also depending on $f$. It follows that, for $\ell \ge j$,

\begin{equation*} 
	\begin{aligned}
	2^{(j-\ell)Mr}\f{|\varphi(2^{-\ell}t\sL)f(y)|^r}{(1+2^j d(x,y))^{\lambda r}} \lesi&
	\sum_{k\geq \ell}2^{(j-k)Ar}\int_X\f{2^{(j-\ell)\lambda r}}{V(z,2^{-k})}\frac{|\varphi(2^{-k}t\sL)f(z)|^r}{(1+2^k d(y,z))^{\lambda r}(1+2^jd(x,y))^{\lambda r}}\dz \\
	=& \sum_{k\geq \ell}2^{(j-k)Ar}\int_X\f{2^{(k-\ell)\lambda r}}{V(z,2^{-k})}\frac{|\varphi(2^{-k}t\sL)f(z)|^r}{(1+2^k d(y,z))^{\lambda r}(2^{k-j}+2^kd(x,y))^{\lambda r}}\dz\\
	\leq& \sum_{k\geq j}2^{(j-k)(A-\lambda)r}\int_X\f{1}{V(z,2^{-k})}\frac{|\varphi(2^{-k}t\sL)f(z)|^r}{(1+2^k d(x,z))^{\lambda r}}\dz.
	\end{aligned}
	\end{equation*}
	(Recall that $M=\lambda+A.$) Taking the supremum over $\ell \ge j$ and $y\in X$ gives
	\[
	\varphi_{A,\lambda}^{**}(2^{-j}t\sL)f(x)^r \lesi
	\sum_{k\geq j}2^{(j-k)(A-\lambda)r}\int_X\f{1}{V(z,2^{-k})}\frac{|\varphi(2^{-k}t\sL)f(z)|^r}{(1+2^k d(x,z))^{\lambda r}}\dz.
	\]
Therefore, if the right-hand side of the above is finite, then	 $\varphi_{A,\lambda}^{**}(2^{-j}t\sL)f(x) < \infty$.  Repeating the proof of \eqref{eq- vrphi **}, we obtain
\begin{equation*}
\begin{aligned}
\varphi_{A,\lambda}^{**}(2^{-j}t\sL)f(x)^r &\lesi \sum_{k\geq j}2^{(j-k)Ar}\int_X\f{1}{V(z,2^{-k})}\frac{|\varphi(2^{-k}t\sL)f(z)|^r}{(1+2^k d(x,z))^{\lambda r}}\dz\\
&\lesi
\sum_{k\geq j}2^{(j-k)(A-\lambda)r}\int_X\f{1}{V(z,2^{-k})}\frac{|\varphi(2^{-k}t\sL)f(z)|^r}{(1+2^k d(x,z))^{\lambda r}}\dz
\end{aligned}
\end{equation*}
(with inequality constant independent of $f$). Since clearly $\varphi^{*}_\lambda \lesi \varphi_{A,\lambda}^{**}$,  it follows that
\begin{equation}\label{eq-varphi*}
\varphi^{*}_\lambda(2^{-j}t\sL)f(x)^r \lesi \sum_{k\geq j}2^{(j-k)(A-\lambda)r}\int_X\f{1}{V(z,2^{-k})}\frac{|\varphi(2^{-k}t\sL)f(z)|^r}{(1+2^k d(x,z))^{\lambda r}}\dz,
\end{equation}
provided the sum in the right-hand side is finite.  Since \eqref{eq-varphi*} is obviously true when this sum is infinite, we conclude that it holds for all $\lambda>0, A>0, j\in \mathbb{Z}, t\in [1,2]$, and $0<r\le 1$.

Assume now that $r>1$. Let $\lambda>0,A>0$ and $N>n$. Using the Calder\'{on} reproducing formula and Lemma~\ref{lem1} as at the beginning of this step, we see that, for any $j\in \mathbb{Z}$, $t\in [1,2]$ and $y\in X$,
\[
	\begin{aligned}
	|\varphi(2^{-j}t\sL)f(y)| \lesi&
	 \sum_{k\geq j}2^{(j-k)A}\int_X\f{(1+2^jd(y,z))^{-N}}{V(z,2^{-j})}\frac{|\varphi(2^{-k}t\sL)f(z)|}
	{(1+2^j d(y,z))^{\lambda}}\dz\\
	\lesi& \left(\sum_{k\geq j}2^{(j-k)Ar}\int_X\f{1}{V(z,2^{-j})}\frac{|\varphi(2^{-k}t\sL)f(z)|^r}{(1+2^j d(y,z))^{\lambda r}}\dz\right)^{1/r}\\
	\le& \left(\sum_{k\geq j}2^{(j-k)Ar}\int_X\f{1}{V(z,2^{-k})}\frac{|\varphi(2^{-k}t\sL)f(z)|^r}{(1+2^j d(y,z))^{\lambda r}}\dz\right)^{1/r},
	\end{aligned}
	\]
	where we have also used Lemma~\ref{lem-elementary} and 
	 H\"older's inequality in the second last inequality in the above. It follows that 
	\begin{equation}
	\label{eq-59}
	\begin{aligned}
	\frac{|\varphi(2^{-j}t\sL)f(y)|^r}{(1+2^j d(x,y)/t)^{\lambda r}} %\lesi& \sum_{k\geq j}2^{(j-k)Ar}\int_X\f{1}{V(z,2^{-k})}\frac{|\varphi(2^{-k}t\sL)f(z)|^r}{(1+2^j d(y,z))^{\lambda r}}\dz\\
	\lesi& \sum_{k\geq j}2^{(j-k)Ar}\int_X\f{1}{V(z,2^{-k})}\frac{|\varphi(2^{-k}t\sL)f(z)|^r}{(1+2^j d(y,z))^{\lambda r}(1+2^j d(x,y))^{\lambda r}}\dz\\
	\lesi& \sum_{k\geq j}2^{(j-k)Ar}\int_X\f{1}{V(z,2^{-k})}\frac{|\varphi(2^{-k}t\sL)f(z)|^r}{(1+2^j d(x,z))^{\lambda r}}\dz\\
	\lesi& \sum_{k\geq j}2^{(j-k)(A-\lambda)r}\int_X\f{1}{V(z,2^{-k})}\frac{|\varphi(2^{-k}t\sL)f(z)|^r}{(1+2^k d(x,z))^{\lambda r}}\dz.
	\end{aligned}
	\end{equation}
	Taking the supremum of the LHS over all $y\in X$, we deduce that \eqref{eq-varphi*} also holds for $r>1$. Hence \eqref{eq-varphi*} holds for all $r>0$. Consequently,
	\begin{equation*}
	\begin{aligned}
	\left[(2^{-j}t)^{-\alpha}\varphi^{*}_\lambda(2^{-j}t\sL)f(x)\right]^{r}
	\lesi\sum_{k\ge j}2^{(j-k)(A-\lambda+\alpha) r}\int_X\f{1}{V(z,2^{-k})}\f{\left[(2^{-k}t)^{-\alpha}|\varphi(2^{-k}t\sL)f(z)|\right]^r}{(1+2^k d(x,z))^{\lambda r}}\dz.
	\end{aligned}
	\end{equation*}
	
	We now choose $r>0$ such that $\max\{n/p,n/q, nq_w/p\} =
	\max\{n/q, nq_w/p\}< n/r < \lambda$ and then choose $A>0$ such that $A-\lambda+\alpha>0$ . Minkowski's inequality and the above inequality then imply that
	\begin{equation*}
	\begin{aligned}
	\Big(\int_{1}^2&\left[(2^{-j}t)^{-\alpha}\varphi^{*}_\lambda(2^{-j}t\sL)f(x)\right]^{q}\f{dt}{t}\Big)^{r/q}\\
	&\lesi \sum_{k\ge j}\int_X \f{2^{(j-k)(A-\lambda+\alpha)r}}{V(z,2^{-k})}\f{\Big(\displaystyle \int_{1}^2\left[(2^{-k}t)^{-\alpha}|\varphi(2^{-k}t\sL)f(z)|\right]^q\f{dt}{t}\Big)^{r/q}}{(1+2^{k}d(x,z))^{\lambda r}}\dz.
	\end{aligned}
	\end{equation*}
	By a change of variables, we get that
	\begin{equation}
	\label{eq-Mink-heat}
	\begin{aligned}
	\Big[\int_{2^{-j}}^{2^{-j+1}}&(t^{-\alpha}|\varphi^{*}_\lambda(t\sL)f(x)|)^{q}\f{dt}{t}\Big]^{r/q}\\
	&\lesi \sum_{k\ge j}\int_X \f{2^{(j-k)(A-\lambda+\alpha)r}}{V(z,2^{-k})}\f{\Big[\displaystyle\int_{2^{-k}}^{2^{-k+1}}(t^{-\alpha}|\varphi(t\sL))f(z)|)^q\f{dt}{t}\Big]^{r/q}}{(1+2^{k}d(x,z))^{\lambda r}}\dz.
	\end{aligned}
	\end{equation}
	Setting $F_k = {\Big[\displaystyle\int_{2^{-k}}^{2^{-k+1}}(t^{-\alpha}|\varphi(t\sL))f(z)|)^q\f{dt}{t}\Big]^{1/q}}$ and using Lemma~\ref{lem-elementary}, we deduce that
	\[
	\begin{aligned}
	\Big[\int_{2^{-j}}^{2^{-j+1}}(t^{-\alpha}|\varphi^{*}_\lambda(t\sL)f(x)|)^{q}\f{dt}{t}\Big]^{r/q}
	&\lesi \sum_{k\ge j} 2^{(j-k)(A-\lambda+\alpha)r} \mathcal{M}_r(F_k)(x)^r\\
	&\lesi \Big[\sum_{k\ge j} 2^{(j-k)(A-\lambda+\alpha)r} \mathcal{M}_r(F_k)(x)^q\Big]^{r/q},
	\end{aligned}
	\]
	where in the last inequality we use  H\"older's inequality and the fact that $A-\lambda+\alpha>0$.   Hence, by applying \eqref{YFSIn}, we obtain the desired estimate
	$$
	\Big\|\Big(\int_0^\vc \Big[t^{-\alpha}|\varphi^{*}_\lambda(t\sL)f|\Big]^q\f{dt}{t}\Big)\Big\|_{p,w}\lesi \Big\|\Big(\int_0^\vc \Big[t^{-\alpha}|\varphi(t\sL)f|\Big]^q\f{dt}{t}\Big)\Big\|_{p,w}.
	$$
	Combining the results in Step 1, Step 2 and Step 3, we complete the proof of the Theorem.  
\end{proof}

\begin{rem}
	\label{rem theorem 4.2}
	\noindent {\rm (a)} It is interesting to note that when $L=-\Delta$ the Laplacian on $\mathbb{R}^n$, $\mathscr P$ is the set of all polynomials. In this situation, Theorem \ref{thm2} are in line with findings in \cite{Qui3}.
	
	\noindent {\rm (b)} The presence of the polynomial $\rho$ in \eqref{eq-Besov equivalent norms square functions} and \eqref{eq-TL equivalent norms square functions} can be omitted if $f\in L^2$. Indeed, in this case instead of \eqref{eq- f Svc S} we have, by the spectral theory,
	\[
	\varphi(t\sL)f=c_\psi\int_0^\vc \varphi(t\sL)\psi(s\sL)f\f{ds}{s} \quad \text{(in $L^2$ and hence in $\mathcal{S}'$)}.
	\]
	Arguing similarly to the proof of the first inequalities in \eqref{eq-Besov equivalent norms square functions} and \eqref{eq-TL equivalent norms square functions} we get the desired estimates.
\end{rem}
\bigskip

We denote $\Psi_{m,t}(L)=(t^2L)^me^{-t^2L}$ for $t>0$ and $m\in \mathbb{N}$. For $\lambda>0$ we define
\begin{equation}
\label{eq-PetreeFunction-heatkernel}
\Psi_{m,t,\lambda}^*(L)f(x)=\sup_{y\in X}\f{|\Psi_{m,t}(L)f(y)|}{(1+d(x,y)/t)^\lambda}
\end{equation}
for $f\in \mathcal S'$.

Applying Theorem \ref{thm2} and Remark \ref{rem theorem 4.2} for $\varphi(\xi)=\xi^{2m}e^{-\xi^2}$, we have the following heat kernel characterizations for the new Besov and Triebel-Lizorkin spaces:
\begin{cor}
	\label{cor1}
	Let $w\in A_\vc$, $\alpha\in \mathbb{R}$ and $m>\alpha/2$. Then  we the following norm equivalences hold:
	\begin{enumerate}[(a)]
		\item For $0< p, q\le \vc$, $\lambda>nq_w/p$, and $f\in \mathcal S'$ there exists $\rho\in \mathscr P$ so that 
		\begin{equation}
		\label{eq-Besov equivalent norms square functions-heatkernels}
		\Big(\int_0^\vc\Big[t^{-\alpha}\|\Psi_{m,t,\lambda}^*(L)(f-\rho)\|_{p,w}\Big]^q\f{dt}{t}\Big)^{1/q}\lesi \|f\|_{\B^{\alpha, L}_{p,q,w}}\lesi \Big(\int_0^\vc\Big[t^{-\alpha}\|\Psi_{m,t}(L)f\|_{p,w}\Big]^q\f{dt}{t}\Big)^{1/q}.
		\end{equation}
		
		\item For $0< p<0$, $0< q\le \vc$, $\lambda>\max\{n/q, nq_w/p\}$, and $f\in \mathcal S'$ there exists $\rho\in \mathscr P$ so that 
		\begin{equation}
		\label{eq-TL equivalent norms square functions-heatkernel}
		\Big\| \Big(\int_0^\vc\Big[t^{-\alpha}\Psi_{m,t,\lambda}^*(L)(f-\rho)\Big]^q\f{dt}{t}\Big)^{1/q}\Big\|_{p,w}\lesi \|f\|_{\F^{\alpha, L}_{p,q,w}}\lesi \Big\| \Big(\int_0^\vc\Big[t^{-\alpha}\Psi_{m,t}(L)f\Big]^q\f{dt}{t}\Big)^{1/q}\Big\|_{p,w}.
		\end{equation}
		Moreover, if $f\in L^2$, $\rho$ can be removed in \eqref{eq-Besov equivalent norms square functions-heatkernels} and \eqref{eq-TL equivalent norms square functions-heatkernel}.
	\end{enumerate}
\end{cor}

\begin{rem}
	Note that in \cite{KP, G.etal}, the authors proved \eqref{eq-Besov equivalent norms square functions-heatkernels} for $s\in \mathbb{R}, 1\le p\le \vc, 0<q\le \vc$, and \eqref{eq-TL equivalent norms square functions-heatkernel} for $s\in \mathbb{R}, 1< p< \vc, 1< q\le \vc$ for inhomogeneous and homogeneous Besov and Triebel-Lizorkin spaces under additional conditions of H\"older continuity and Markov property of the heat kernel $p_t(x,y)$. Moreover, their results are formulated for distributions in $\mathcal{S}'_\infty$. 
	Hence, Corollary \ref{cor1} can be viewed as a significant extension of those results in \cite{KP, G.etal}. 
\end{rem}

We end this subsection by a remark on an interesting extension of Theorem~\ref{thm2}.

\begin{rem}
The non-degeneracy condition of the function $\varphi$ in the definition of the class $\mathscr{S}_m(\mathbb{R})$, $\varphi(\xi)\neq 0$ on $(-2,-1/2)\cup (1/2,2)$, can be weakened to $\varphi(\lambda)\neq 0$ for some $\lambda > 0$ (as $\varphi$ is even, this implies also $\varphi(-\lambda)\neq 0$). Then Theorem~\ref{thm2} holds under this weaker condition on $\varphi$. The proof of this stronger result can be done by modifying the proof for Theorem~\ref{thm2} and using the following two observations:

(a) There exist $a>0$, $c>b>0$, and even functions $\phi,\eta\in \mathscr{S}(\mathbb{R})$, such that $\rm{supp}\, \phi \subset [-a,a], \rm{supp}\,\eta \subset [-c,-b]\cup[b,c]$, and 
$$
\phi(\lambda) + 
\int_1^\infty \varphi(s\lambda)\eta(s\lambda)
\frac{ds}{s} = 1  
\quad  \forall \lambda \in \mathbb{R}.
$$
See, for example \cite{Hei,ST}.

(b) Part (a) and an argument similar to the proof of Proposition~\ref{prop-Calderon1} imply a Calder\'{o}n reproducing formula
$$
f=\phi(t\sqrt{L}f) + \int_1^\infty \varphi(ts\sqrt{L}f)\eta(ts\sqrt{L})\frac{ds}{s} \quad \text{in $\mathcal{S}'$},
$$
for all $f\in \mathcal{S}'$ and all $t>0$. 

We leave the details to the interested reader, and also refer to \cite{Qui3,Qui4} for the proof in the classical case.

\end{rem}

\subsection{Characterizations for Triebel-Lizorkin spaces via Lusin functions and the Littlewood--Paley functions.}

For $\alpha\in \mathbb{R}, \lambda, a>0$ and $0<q<\vc$ we define the Lusin function and the Littlewood--Paley function by setting
\begin{equation}\label{g-function}
\mathcal{G}^{\alpha}_{\lambda, q}F(x)=\left[\int_0^\vc\int_X(t^{-\alpha}|F(y,t)|)^q\Big(1+\f{d(x,y)}{t}\Big)^{-\lambda q} \f{\dy dt}{tV(x,t)}\right]^{1/q}
\end{equation}
and
\begin{equation}\label{lusin-function}
\mathcal{S}^\alpha_{a,q}F(x)=\left[\int_0^\vc\int_{d(x,y)<at}(t^{-\alpha}|F(y,t)|)^q\f{\dy dt}{tV(x,t)}\right]^{1/q},
\end{equation}
respectively.

When either $\alpha=0$ or $a=1$ we will drop them in the notation of $\mathcal{S}^\alpha_{a,q}$ and $\mathcal{G}^\alpha_{\lambda,q}$. We now have the following result regarding the estimates on the change of the angles for the function $\mathcal{S}^\alpha_{a,q}$.

\begin{prop}
	\label{pro-change the angle}
	Let $a>1$, $w\in A_r, 1\leq r<\vc, 0<q<\vc, \alpha\in \mathbb{R}$ and $0<p<\vc$. Then there exists a constant $C$ so that
	\[
	\|\mathcal{S}^\alpha_{a,q}F\|_{p,w}\leq Ca^{\f{rn}{p\wedge q}}\|\mathcal{S}^\alpha_{q}F\|_{p,w}
		\]
	for all $F$.
\end{prop}
\begin{proof}
	It suffices  to prove the proposition for $\alpha=0$ and $q=2$, since in the general case of $\alpha$ and $q$ we set $\tilde F(y,t)=(t^{-\alpha} |F(y,t)|)^{q/2}$ and then apply the result for the case $\alpha=0$ and $q=2$ we will get the desired estimate. In this situation, we can adapt the proof of   \cite[Theorem 1]{AS} to our setting easily. Hence we omit the details.
	
	We need to give the proof for the case $p\geq 2$. Indeed, for a positive function $g\in L^{(p/2)'}_w$ we have
	\[
	\begin{aligned}
	\langle [\mathcal{S}_{a,q}f]^2,g\rangle_w:&=\int_X\int_0^\vc\int_{d(x,y)<at}|f(y,t)|^2\f{\dy dt}{tV(x,t)}g(x)w(x)\dx\\
	&\sim \int_X\int_0^\vc\int_{d(x,y)<at}|f(y,t)|^2\f{\dy dt}{tV(y,t)}g(x)w(x)\dx\\
	&=\int_X\int_0^\vc|f(y,t)|^2M_{at,w}g(y)w(B(y,at))\f{\dy dt}{tV(y,t)}
	\end{aligned}
	\] 
	where 
	\[
	M_{at,w}g(y)=\f{1}{w(B(y,at))}\int_{B(y,at)}g(x)w(x)\dx.
	\]
	We now observe that
	\[
	\chi_{B(y,at)}(x)\leq \f{c_n}{w(B(y,t))}\int_{B(y,t)}\chi_{B(x,at)}(z)w(z)\dz.
	\]
	It follows
	\[
	\begin{aligned}
	\f{1}{w(B(y,at))}&\int_{B(y,at)}\chi_{B(y,at)}(x)g(x)w(x)\dx\\
	&\leq \f{c_n}{w(B(y,t))}\int_{B(y,t)}\f{1}{w(B(y,at))}\int_{B(y,at)}\chi_{B(x,at)}(z)g(x)w(x)\dx w(z)\dz .
	\end{aligned}
	\]
	Note that in this situation, we have $d(y,z)\leq 3at$, and hence $B(y,at)\subset B(z,3at)$ and $w(B(y,at))\sim w(B(z,3at))$. As a consequence, we have
	\[
	\begin{aligned}
	\f{1}{w(B(y,at))}\int_{B(y,at)}\chi_{B(x,at)}(z)g(x)w(x)\dx
	&\lesi \f{1}{w(B(z,3at))}\int_{B(z,3at)}g(x)w(x)\dx\\
	&\lesi \mathcal{M}_wg(z).
	\end{aligned}
	\]
	This implies that
	\[
	\begin{aligned}
	M_{at,w}g(y)\leq c_n M_{t,w}[\mathcal{M}_w g](y).
	\end{aligned}
	\]
	Hence,
	\[
	\begin{aligned}
	\langle [\mathcal{S}_{a,q}f]^2,g\rangle_w	&\leq c_n \int_X\int_0^\vc|f(y,t)|^2M_{t,w}[\mathcal{M}_w g](y)w(B(y,at))\f{\dy dt}{tV(y,t)}\\
	&\sim c_n \int_X\int_0^\vc\int_{d(x,y)<t}|f(y,t)|^2\f{w(B(y,at))}{w(B(y,t))}\f{\dy dt}{tV(y,t)}\mathcal{M}_wg(x)w(x)\dx\\
	&\sim c_na^{nr} \int_X\int_0^\vc\int_{d(x,y)<t}|f(y,t)|^2\f{\dy dt}{tV(y,t)}\mathcal{M}_wg(x)w(x)\dx\\
	&\sim c_na^{nr} \langle [\mathcal{S}_{q}f]^2,\mathcal{M}_wg\rangle_w\\
	&\leq c_na^{nr} \|[\mathcal{S}_{q}f]^2\|_{\f{p}{2},w}\|\mathcal{M}_wg\|_{(\f{p}{2})',w}\leq c_na^{nr}\|\mathcal{S}_{q}f\|^2_{p,w}\|g\|_{(\f{p}{2})',w}.
	\end{aligned}
	\]
	Taking the supremum over all  $g\in L^{(p/2)'}_w$ with $\|g\|_{(p/2)',w}\le 1$ we obtain
	\[
	\|\mathcal{S}_{a,q}f\|_{p,w}\leq \sqrt{c_n}a^{nr/2}\|\mathcal{S}_{q}f\|_{p,w}.
	\]
	This completes our proof.
\end{proof} 
\bigskip

We have the following corollary.
\begin{cor}
	\label{cor-change angle}
	Let $a\ge 1$, $w\in A_r, 1\leq r<\vc, 0<q<\vc, \alpha\in \mathbb{R}$ and $0<p<\vc$. Then there exists a constant $C$ so that
	\[
	C^{-1}\|\mathcal{G}^{\alpha}_{\lambda,q}F\|_{p,w}\le \|\mathcal{S}^\alpha_{a,q}F\|_{p,w}\le C \|\mathcal{G}^{\alpha}_{\lambda,q}F\|_{p,w}
	\]
	for all $F$ provided that  $\lambda>\f{nr}{p\wedge q}$.
\end{cor}
\begin{proof}
	Due to Proposition \ref{pro-change the angle}, we need only to prove the corollary for $a=1$.\\
	
	Since $\mathcal{S}^\alpha_{a,q}F\leq \mathcal{G}^{\alpha}_{\lambda,q}F$ for any $\lambda>0$, it suffices to prove that
	\[
	\|\mathcal{G}^{\alpha}_{\lambda,q}F\|_{p,w}\lesi \|\mathcal{S}^\alpha_{q}F\|_{p,w}.
	\]
	Indeed, it is easy to see that
	\begin{equation}\label{eq1-proof change angle}
	[\mathcal{G}^{\alpha}_{\lambda,q}F]^q\leq \sum_{k\ge 0}2^{-kq\lambda}[\mathcal{S}^\alpha_{2^k,q}F]^q.
	\end{equation}
	Hence,
	\begin{equation*}
	[\mathcal{G}^{\alpha}_{\lambda,q}F]^p\leq \Big[\sum_{k\ge 0}2^{-kq\lambda}[\mathcal{S}^\alpha_{2^k,q}F]^q\Big]^{p/q}.
	\end{equation*}
	If $p/q<1$, we then have
	\[
	[\mathcal{G}^{\alpha}_{\lambda,q}F]^p\leq \sum_{k\ge 0}2^{-kp\lambda}[\mathcal{S}^\alpha_{2^k,q}F]^p.
	\]
	This, along with Proposition \ref{pro-change the angle}, implies that
	\[
	\begin{aligned}
	\|\mathcal{G}^{\alpha}_{\lambda,q}F\|_{p,w}^p&\leq \sum_{k\ge 0}2^{-kp\lambda}\|\mathcal{S}^\alpha_{2^k,q}F\|_{p,w}^p\leq c\sum_{k\ge 0}2^{-kp\lambda}2^{knr}\|\mathcal{S}^\alpha_{q}F\|_{p,w}^p\\
	&\lesi \|\mathcal{S}^\alpha_{q}F\|_{p,w}^p
	\end{aligned}
	\]
	as long as $\lambda>nr/p$.

	\medskip
	
	If $p/q\ge 1$, then from \eqref{eq1-proof change angle} we have
	\[
	\begin{aligned}
	\|[\mathcal{G}^{\alpha}_{\lambda,q}F]^q\|_{p/q,w}\leq \sum_{k\ge 0}2^{-kq\lambda}\|[\mathcal{S}^\alpha_{2^k,q}F]^q\|_{p/q,w}.
	\end{aligned}
	\]
	Applying Proposition \ref{pro-change the angle}, we get 
	\[
	\begin{aligned}
	\|\mathcal{G}^{\alpha}_{\lambda,q}F\|^q_{p,w}&\leq \sum_{k\ge 0}2^{-kq\lambda}\|\mathcal{S}^\alpha_{2^k,q}F\|^q_{p,w}\lesi \sum_{k\ge 0}2^{-kq\lambda}2^{knr}\|\mathcal{S}^\alpha_{q}F\|^q_{p,w}\\
	&\lesi \|\mathcal{S}^\alpha_{q}F\|^q_{p,w}
	\end{aligned}
	\]
	provided $\lambda>nr/q$.
	
	This completes our proof of the corollary.
\end{proof}

\bigskip

We have the following characterization.
\begin{prop}\label{prop4.1b}
	Let $\psi$ be a partition of unity. Then for  $w\in A_\vc$, $0<p, q<\vc$, and $\alpha\in \mathbb{R}$, we have
	$$
	\|f\|_{\F^{\alpha,L}_{p,q,w}}\sim  \|\mathcal{G}^{\alpha}_{\lambda, q}(\psi(t\sL)f)\|_{p,w}\sim \|\mathcal{S}^{\alpha}_{q}(\psi(t\sL)f)\|_{p,w}
	$$
	for all $f\in \mathcal{S}'_\vc$, provided that $\lambda>\f{nq_w}{p\wedge q}$.
\end{prop}
\begin{proof}
	We first prove that $\|\mathcal{S}^{\alpha}_{q}(\psi(t\sL)f)\|_{p,w}\lesi \|f\|_{\F^{\alpha,L}_{p,q,w}}$. First observe that  
	\[
	|\psi(t\sL)f(y)|\leq \psi^*_\lambda(t\sL)f(x)
	\]
	for all $\lambda>0$ and $d(x,y)<t$.
	
	Therefore,
	\[
	\begin{aligned}
	\mathcal{S}^{\alpha}_{q}(\psi(t\sL)f)(x)&\leq \left[\int_0^\vc\int_{d(x,y)<t}(t^{-\alpha}|\psi^*_\lambda(t\sL)f(x)|)^q\f{\dy dt}{tV(x,t)}\right]^{1/q}\\
	&\lesi \left[\int_0^\vc(t^{-\alpha}|\psi^*_\lambda(t\sL)f(x)|)^q\f{dt}{t}\right]^{1/q}.
	\end{aligned}
	\]
	This, along with Theorem  \ref{thm1-continuouscharacter}, implies that
	\[
	\|\mathcal{S}^{\alpha}_{q}(\psi(t\sL)f)\|_{p,w}\lesi \|f\|_{\F^{\alpha,L}_{p,q,w}}.
	\]
	Due to the above estimate and Corollary \ref{cor-change angle}, it remains to show that
	\[
	\|f\|_{\F^{\alpha,L}_{p,q,w}}\lesi \|\mathcal{G}^{\alpha}_{\lambda, q}(\psi(t\sL)f)\|_{p,w}.
	\] 
	By Proposition \ref{prop2-maximal functions} we have 
	\begin{equation*}
	|\psi(t\sL)f(x)|\lesi \Big[\int_X \f{1}{V(x,t)}\f{|\psi(t\sL)f(z)|^q}{(1+d(x,z)/t)^{\lambda q}}\dz\Big]^{1/q}
	\end{equation*}
	for all $x\in X$, $\lambda>0$ and $t>0$.
	
	This implies that
	\[
	\Big\|\Big(\int_0^\vc (t^{-\alpha}|\psi(t\sL)f|)^q\f{dt}{t}\Big)^{1/q}\Big\|_{p,w}\lesi \|\mathcal{G}^{\alpha}_{\lambda, q}(\psi(t\sL)f)\|_{p,w}.
	\]
	Using Theorem \ref{thm1-continuouscharacter} we obtain the bound:
	\[
	\|f\|_{\F^{\alpha,L}_{p,q,w}}\lesi \|\mathcal{G}^{\alpha}_{\lambda, q}(\psi(t\sL)f)\|_{p,w}
	\]
	as desired.
\end{proof}

\bigskip

We also have a similar square function characterization for the new Triebel-Lizorkin spaces via functions in $\mathscr{S}_m(\mathbb{R})$.
\begin{prop}\label{prop4.1-heat kernel}
	Let  $w\in A_\vc$, $0<p<\vc$, $0<q< \vc$, $\alpha\in \mathbb{R}$, $\lambda>\f{nq_w}{p\wedge q}$ and $\varphi\in \mathscr{S}_m(\mathbb{R})$ with $m>\alpha/2$. Then for each  $f\in \mathcal S'$  there exists $\rho\in \mathscr P$ such that 
	$$
	\begin{aligned}
	\|\mathcal{G}^{\alpha}_{\lambda, q}(\varphi(t\sL)(f-\rho))\|_{p,w}&\sim \|\mathcal{S}^{\alpha}_{q}(\varphi(t\sL)(f-\rho))\|_{p,w} \\
	&\lesi \|f\|_{\F^{\alpha,L}_{p,q,w}}\lesi  \|\mathcal{G}^{\alpha}_{\lambda, q}(\varphi(t\sL)f)\|_{p,w}\sim \|\mathcal{S}^{\alpha}_{q}(\varphi(t\sL)f)\|_{p,w}.
	\end{aligned}	
	$$
	Moreover, the distribution $\rho$ in the inequalities above can be removed if $f\in L^2$.
\end{prop} 
\begin{proof}
	Arguing similarly to the proof of Proposition \ref{prop4.1b} and using Theorem \ref{thm2} we show that
	\[
	\|\mathcal{S}^{\alpha}_{q}(\varphi(t\sL)(f-\rho))\|_{p,w}\lesi \|f\|_{\F^{\alpha,L}_{p,q,w}}.
	\]
	It remains to prove that
	\begin{equation}\label{eq2-proof square}
	\|f\|_{\F^{\alpha,L}_{p,q,w}}\lesi \|\mathcal{G}^{\alpha}_{\lambda, q}(\varphi(t\sL)f)\|_{p,w}.
	\end{equation}
	To do this, we divide into 2 cases. 
	
		If $q\in (0,1]$, then applying \eqref{eq-Mink-heat} with $\lambda$ replaced by $\lambda+\tilde{n}/q$, $r=q$  and $A-\lambda-\tilde{n}/q+\alpha>0$, and using \eqref{doubling3}, we deduce that 

	\begin{equation*}
	\sum_{j\in \mathbb{Z}}\int_{2^{-j}}^{2^{-j+1}}(t^{-\alpha}|\varphi(t\sL)f(x)|)^{q}\f{dt}{t}\lesi \sum_{j\in \mathbb{Z}}\sum_{k\in \mathbb{Z}}\int_{2^{-k}}^{2^{-k+1}}\int_X \f{2^{-|k-j|(A-\lambda-\tilde n/q +\alpha) q}}{V(x,t)}\f{(t^{-\alpha}|\varphi(t\sL)f(z)|)^q}{(1+d(x,z)/t)^{\lambda q}}\dz\f{dt}{t}.
	\end{equation*}
	Hence,
	\begin{equation*}
	\int_{0}^{\vc}(t^{-\alpha}|\varphi(t\sL)f(x)|)^{q}\f{dt}{t}\lesi \int_{0}^{\vc}\int_X \f{1}{V(x,t)}\f{(t^{-\alpha}|\varphi(t\sL)f(z)|)^q}{(1+d(x,z)/t)^{\lambda q}}\dz\f{dt}{t}.
	\end{equation*}
	As a consequence,
	\[
	\Big[\int_{0}^{\vc}(t^{-\alpha}|\varphi(t\sL)f(x)|)^{q}\f{dt}{t}\Big]^{1/q}\lesi \mathcal{G}^{\alpha}_{\lambda, q}(\varphi(t\sL)f)(x).
	\]
	This, along with Theorem \ref{thm2}, yields \eqref{eq2-proof square}.
	
	If $q>1$, then for a partition of unity  function $\psi$ we have
	\[
	f= c\int_0^\vc \psi(s\sL)\varphi(s\sL)f\f{ds}{s} \ \ \text{in $\mathcal S'_\vc$}.
	\]
	
	It follows that, for each $t>0$ and $x\in X$,
	\[
	\psi(t\sL)f(x) =c\int_{t/4}^{4t} \psi(t\sL)\psi(s\sL)\varphi(s\sL)f(x)\f{ds}{s}.
	\]
	Using Lemma \ref{lem1}, for $\lambda>\f{nq_w}{p\wedge q}$, we see that 
	\begin{equation*}
	\begin{aligned}
	|\psi(t\sL)f(x)| &\lesi \int_{t/4}^{4t} \int_X \f{1}{V(x\vee y,s)}\Big(\f{s}{s+d(x,y)}\Big)^{\lambda q}|\varphi(s\sL)f(y)|\dy\f{ds}{s}.
	\end{aligned}
	\end{equation*}
	
	H\"older's inequality and  Lemma \ref{lem-elementary}  then imply that
	\begin{equation*}
	\begin{aligned}
	|\psi(t\sL)f(x)|^q&\lesi \int_{t/4}^{4t} \int_X \f{1}{V(x\vee y,s)}\Big(\f{s}{s+d(x,y)}\Big)^{\lambda q}|\varphi(s\sL)f(y)|^q\dy\f{ds}{s}\\
	& \ \ \times \Big[\int_{t/4}^{4t} \int_X \f{1}{V(x\vee y,s)}\Big(\f{s}{s+d(x,y)}\Big)^{\lambda q}\dy\f{ds}{s}\Big]^{q/q'}\\
	&\lesi \int_{t/4}^{4t} \int_X \f{1}{V(x\vee y,s)}\Big(\f{s}{s+d(x,y)}\Big)^{\lambda q}|\varphi(s\sL)f(y)|^q\dy\f{ds}{s}.
	\end{aligned}
	\end{equation*}
	Consequently, 
	\[
	\begin{aligned}
	\int_0^\vc (t^{-\alpha}|\psi(t\sL)f(x)|)^q\f{dt}{t}&\lesi \int_0^\vc\int_{t/4}^{4t} \int_X \f{1}{V(x\vee y,s)}\Big(\f{s}{s+d(x,y)}\Big)^{\lambda q}(s^{-\alpha}|\varphi(s\sL)f(y)|)^q\dy\f{ds}{s}\f{dt}{t}\\
	&\lesi \int_0^\vc\int_{s/4}^{4s} \int_X \f{1}{V(x\vee y,s)}\Big(\f{s}{s+d(x,y)}\Big)^{\lambda q}(s^{-\alpha}|\varphi(s\sL)f(y)|)^q\dy\f{dt}{t}\f{ds}{s}\\
	&\lesi \int_0^\vc\int_X \f{1}{V(x\vee y,s)}\Big(\f{s}{s+d(x,y)}\Big)^{\lambda q}(s^{-\alpha}|\varphi(s\sL)f(y)|)^q\dy\f{ds}{s}\\
	&\lesi  \mathcal{G}^{\alpha}_{\lambda, q}(\varphi(t\sL)f)(x)^q.
	\end{aligned}
	\]
	This, along with Theorem \ref{thm1-continuouscharacter}, yields \eqref{eq2-proof square}.
	
	This completes our proof.
\end{proof}

\bigskip

\begin{cor}\label{cor-area function}
	Let  $w\in A_\vc$, $0<p<\vc$, $0<q< \vc$, $\alpha\in \mathbb{R}$, $\lambda>\f{nq_w}{p\wedge q}$ and  $m>\alpha/2$. Then for each $f\in \mathcal S'$ there exists $\rho\in \mathscr P$ such that
	$$
	\begin{aligned}
	\|\mathcal{G}^{\alpha}_{\lambda, q}(\Psi_{m,t}(L)(f-\rho))\|_{p,w}&\sim \|\mathcal{S}^{\alpha}_{q}(\Psi_{m,t}(L)(f-\rho))\|_{p,w}\\
	&\lesi \|f\|_{\F^{\alpha,L}_{p,q,w}}\lesi  \|\mathcal{G}^{\alpha}_{\lambda, q}(\Psi_{m,t}(L)f)\|_{p,w}\sim \|\mathcal{S}^{\alpha}_{q}(\Psi_{m,t}(L)f)\|_{p,w}
	\end{aligned}
	$$
	where $\Psi_{m,t}(L)=(t^2L)^me^{-t^2L}$.
	
	Moreover, the distribution $\rho$ can be removed if $f\in L^2$.
\end{cor}

\subsection{Weighted Triebel-Lizorkin space $\F_{\vc,q,w}^{\alpha,L}$ and its characterizations}\label{sub-Fvc}
In this section we will give the definition of the weighted Triebel-Lizorkin space $\F_{\vc,q,w}^{\alpha,L}$ and prove some characterizations for this space.
\begin{defn}\label{de pvc}
	Let $0<q\leq \vc, \alpha\in \mathbb{R}$ and $w\in A_\vc$. Let $\psi$ be a partition of unity. The space $\F_{\vc,q,w}^{\alpha,\psi,L}$ is defined as the set of all $f\in \mathcal{S}'_\vc$ so that
	\[
	\|f\|_{\F_{\vc,q,w}^{\alpha,\psi,L}}=
	\sup_{Q: \ {\rm balls}}\Big(\f{V(Q)}{w(Q)^2}\int_{Q}\sum_{j\ge -\log_2 r_Q}^\vc (2^{j\alpha}|\psi_j(\sL)f(x)|)^q \dx\Big)^{1/q},
	\]
	where the supremum is taken over all balls $Q$ in $X$ with radius $r_Q$, with the interpretation that when $q=\infty$, one has
\[
	\|f\|_{\F_{\vc,\vc,w}^{\alpha,\psi,L}}=
	\sup_{Q: \ {\rm balls}}\sup_{j\ge -\log_2 r_Q}
	\Big(\f{V(Q)}{w(Q)^2}\int_{Q} 2^{j\alpha}|\psi_j(\sL)f(x)| \dx\Big).
	\]
\end{defn}

\begin{prop}
	Let $0<q\leq \vc, \alpha\in \mathbb{R}$ and $w\in A_\vc$. Let $\psi$ and $\varphi$ be partitions of unity. Then the spaces $\F_{\vc,q,w}^{\alpha,\psi,L}$ and $\F_{\vc,q,w}^{\alpha,\varphi,L}$ are equivalent. Hence we define the space $\F_{\vc,q,w}^{\alpha,L}$ as any space $\F_{\vc,q,w}^{\alpha,\varphi,L}$ with $\varphi$ being a partition of unity.
\end{prop}

The proposition is a direct consequence of Lemma \ref{lem1-Fvc} and Lemma \ref{lem2-Fvc} below.
\begin{lem}\label{lem1-Fvc}
	Let $0<q\leq \vc, \alpha\in \mathbb{R}$ and $w\in A_\vc$. Let $\psi$ and $\varphi$ be partitions of unity. Then for each $\lambda>np/q+2np^2/q$ we have
	\[
	\begin{aligned}
	\sup_{Q: \ {\rm balls}}&\Big(\f{V(Q)}{w(Q)^2}\int_{Q}\sum_{j\ge -\log_2 r_Q}^\vc (2^{j\alpha}|\psi^*_{j, \lambda}(\sL)f(x)|)^q\dx\Big)^{1/q}\\
	&\sim \sup_{Q: \ {\rm balls}}\Big(\f{V(Q)}{w(Q)^2}\int_{Q}\sum_{j\ge -\log_2r_Q}^\vc (2^{j\alpha}|\varphi^*_{j, \lambda}(\sL)f(x)|)^q \dx\Big)^{1/q} 
	\end{aligned}
	\]
\end{lem}
\begin{proof}
	The proposition follows easily from \eqref{eq-psistar vaphistar}. 
\end{proof}

\begin{lem}\label{lem2-Fvc}
	Let $0<q\leq \vc, \alpha\in \mathbb{R}$ and $w\in A_p, 1<p<\vc$. Let $\psi$ be a partition of unity. Then for each $\lambda>np/q+2np^2/q$ we have the norm equivalence:
	$$
	\begin{aligned}
	\sup_{Q: \ {\rm balls}}\Big(\f{V(Q)}{w(Q)^2}\int_{Q}\sum_{j\ge -\log_2 r_Q}^\vc (2^{j\alpha}|\psi^{*}_{j, \lambda}(\sL)f(x)|)^q \dx\Big)^{1/q}\sim \|f\|_{\F_{\vc,q,w}^{\alpha,\psi,L}}.
	\end{aligned}
	$$
\end{lem}
\begin{proof}
	We consider the case $0<q<\vc$. It suffices to prove 
	\[
	\sup_{Q: \ {\rm balls}}\Big(\f{V(Q)}{w(Q)^2}\int_{Q}\sum_{j\ge -\log_2 r_Q}^\vc (2^{j\alpha}|\psi^{*}_{j, \lambda}(\sL)f(x)|)^q \dx\Big)^{1/q}\lesi \|f\|_{\F_{\vc,q,w}^{\alpha,\psi,L}}.
	\]
	Indeed, fix a ball $Q$ and let $x\in Q$. For $j\geq \nu$, applying Proposition \ref{prop2-maximal functions} with $r=q/p$, we obtain
	\begin{equation*}
	\begin{aligned}
	|\psi^*_{j,\lambda}(\sL)f(x)|^{q}&\lesi \Big[\int_X \f{1}{V(z,2^{-j})}\f{|\psi_j(\sL)f(z)|^{q/p}}{(1+2^jd(x,z))^{\lambda q/p}}\dz\Big]^p\\
	&\lesi \Big[\int_{4Q} \f{1}{V(z,2^{-j})}\f{|\psi_j(\sL)f(z)|^{q/p}}{(1+2^jd(x,z))^{\lambda q/p}}\dz\Big]^p\\
	& \ \ \ \ + \Big[\int_{X\backslash 4Q} \f{1}{V(z,2^{-j})}\f{|\psi_j(\sL)f(z)|^{q/p}}{(1+2^jd(x,z))^{\lambda q/p}}\dz\Big]^p:=E_{j,1}(x) +E_{j,2}(x).
	\end{aligned}
	\end{equation*}
	Using Lemma \ref{lem-elementary} we can bound $E_{j,1}$ as follows:
	\[
	E_{j,1}(x)\lesi \left[\mathcal{M}_r(|\psi_j(\sL)f|\chi_{4Q})(x)\right]^q.
	\]

	Therefore,
	$$
	\begin{aligned}
	\f{V(Q)}{w(Q)^2}\int_{Q}&\sum_{j\ge -\log_2 r_Q}^\vc 2^{j\alpha q} E_{j,1}(x)\dx\\
	&\lesi \f{V(Q)}{w(Q)^2}\int_{Q} \sum_{j\ge -\log_2 r_Q}^\vc  2^{j\alpha q}\left[\mathcal{M}_r(|\psi_j(\sL)f|\chi_{4Q})(x)\right]^q\dx.
	\end{aligned}
	$$
	Since $\mathcal{M}_r$ is bounded on $L^q_w(X)$, the inequality above and Fefferman--Stein's inequality imply that
	$$
	\begin{aligned}
	\f{V(Q)}{w(Q)^2}\int_{Q}\sum_{j\ge -\log_2 r_Q}^\vc 2^{j\alpha q}E_{j,1}(x) dx
	&\lesi \f{V(Q)}{w(Q)^2}\int_{4Q} \sum_{j\ge -\log_2 r_Q}^\vc  (|2^{j\alpha}|\psi_j(\sL)f(x)|)^q \dx\\
	&\lesi \|f\|^q_{\F_{\vc,q,w}^{\alpha,\psi,L}}.
	\end{aligned}
	$$
	We now use H\"older's inequality to dominate $E_{j,2}(x)$ by 
	\begin{equation}\label{eq1-proof Lem2 Fvc}
	\begin{aligned}
	\Big[\int_{X\backslash 4Q} \f{1}{V(z,2^{-j})}&\f{|\psi_j(\sL)f(z)|^q}{(1+2^jd(x,z))^{\lambda q/p}}\dz\Big]\Big[\int_{X\backslash 4Q}\f{1}{V(z,2^{-j})}(1+2^jd(x,z))^{-\lambda q/p}\dz\Big]^{p/p'}.
	\end{aligned}
	\end{equation}
	Recalling Lemma \ref{lem-elementary}, 
	\[
	\begin{aligned}
	\Big[\int_{X\backslash 4Q}\f{1}{V(z,2^{-j})}(1+2^jd(x,z))^{-\lambda q/p} \dz\Big]^{p/p'}&\lesi 1
	\end{aligned}
	\]
	provided $\lambda >np/q$.
	
	Plugging this into \eqref{eq1-proof Lem2 Fvc} yields
	\begin{equation*}
	\begin{aligned}
	E_{j,2}(x)
	&\lesi \sum_{k=3}^\vc \int_{S_k(Q)}\f{1}{V(z,2^{-j})} \f{|\psi_j(\sL)f(z)|^{q}}{(1+2^jd(x,z))^{\lambda q/p}}\dz\\
	&\lesi\sum_{k=3}^\vc 2^{-(j+k-\nu)\lambda q/p}\int_{S_k(Q)}\f{1}{V(z,2^{-j})} |\psi_j(\sL)f(z)|^{q}\dz\\
	&\lesi\sum_{k=3}^\vc \f{2^{-(j+k-\nu)(\lambda q/p-n)}}{V(2^kQ)}\int_{S_k(Q)} |\psi_j(\sL)f(z)|^{q}\dz\\
	&\lesi\sum_{k=3}^\vc \f{2^{-k(\lambda q/p-n)}}{V(2^kQ)}\int_{S_k(Q)} |\psi_j(\sL)f(z)|^{q}\dz
	\end{aligned}
	\end{equation*}
	where $\nu$ is the integer part of $-\log_2 r_Q$, and in the last inequality we use the fact that $j\ge \nu$. Hence, using \eqref{eq- doubling w}, we deduce that

	\begin{equation*} 
	\begin{aligned}
	\f{V(Q)}{w(Q)^2}\int_{Q}\sum_{j\ge -\log_2 r_Q}^\vc &2^{j\alpha q}E_{j,2}(x)\dx\\
	&\lesi \sum_{k=3}^\vc \f{2^{-k(\lambda q/p-n)}}{V(2^kQ)}\f{V(Q)^2}{w(Q)^2}\int_{2^kQ} \sum_{j\ge \nu}^\vc  (2^{j\alpha}|\psi_j(\sL)f(z)|)^q\dz\\
	&\lesi \sum_{k=3}^\vc 2^{-k(\lambda q/p-n-2np)}\f{V(2^kQ)}{w(2^kQ)^2}\int_{2^kQ} \sum_{j\ge \nu-k}^\vc  (2^{j\alpha}|\psi_j(\sL)f(z)|)^q\dz\\
	&\lesi \sum_{k\ge 0}2^{-k(\lambda q/p-n-2np)}\|f\|^q_{\F_{\vc,q,w}^{\alpha,\psi,L}}\\
	&\lesi \|f\|^q_{\F_{\vc,q,w}^{\alpha,\psi,L}}
	\end{aligned}
	\end{equation*}
	as long as $\lambda>np/q+2np^2/q$.
	
	Combining the estimates for $E_{j,1}$ and $E_{j,2}$ we conclude that 
	$$
	\Big(\f{V(Q)}{w(Q)^2}\int_{Q}\sum_{j\ge -\log_2 r_Q}^\vc (2^{j\alpha}|\psi^{*}_{j, \lambda}(\sL)f(x)|)^q\dx\Big)^{1/q}
	\lesi \|f\|_{\F_{\vc,q,w}^{\alpha,\psi,L}},
	$$
	which implies the desired result when $q<\infty$.

	The proof for the case $q=\vc$ is similar with minor modifications, and hence we leave it to the reader.
\end{proof}

\medskip

\begin{prop}\label{prop2-Fvc}
	Let $0<q\leq \vc, \alpha\in \mathbb{R}$ and $w\in A_p, 1<p<\vc$. Let $\psi$ be a partition of unity. Then for $\lambda>np/q+2np^2/q$ we have
	\[
	\begin{aligned}
	\|f\|_{\F_{\vc,q,w}^{\alpha,L}}&\sim \sup_{x\in X, t>0}\Big(\f{V(x,t)}{w(B(x,t))^2}\int_{B(x,t)}\int_0^{t}(s^{-\alpha}|\psi(s\sL)f(y)|)^q \f{ds}{s}\dy\Big)^{1/q}=:N_{\alpha,q, L}(f)\\
	&\sim \sup_{x\in X, t>0}\Big(\f{V(x,t)}{w(B(x,t))^2}\int_{B(x,t)}\int_0^{t}(s^{-\alpha}|\psi^*_\lambda(s\sL)f(y)|)^q \f{ds}{s}\dy\Big)^{1/q}=:N^*_{\lambda,\alpha,q, L}(f)
	\end{aligned}
	\]
	for all $f\in \SLm$.
\end{prop}
\begin{proof}
	As usual we will only prove the proposition for $0<q<\vc$.  
	We first claim that
	\begin{equation}
	\label{eq1-prop2-Fvc}
	\|f\|_{\F_{\vc,q,w}^{\alpha,L}}\lesi N^*_{\lambda,\alpha,q, L}(f).
	\end{equation}
	
	Indeed, from Proposition \ref{prop3-maximal functions} we have, for every $j\in \mathbb{Z}$ and $x\in X$, 
	$$
	|\psi_j(\sL)f(x)|^q\lesi \int_{2^{-j-2}}^{2^{-j+2}}|\psi^*_{\lambda}(s\sL)f(x)|^q\f{ds}{s}
	$$
	Hence, for any $\nu\in \mathbb{Z}$,
	$$
	\sum_{j\ge \nu}(2^{j\alpha}|\psi_j(\sL)f(x)|)^q\lesi \sum_{j\ge \nu}\int_{2^{-j-2}}^{2^{-j+2}}(s^{-\alpha}|\psi^*_{\lambda}(s\sL)f(x)|)^q\f{ds}{s}=\int_{0}^{2^{-\nu+2}}(s^{-\alpha}|\psi^*_{\lambda}(s\sL)f(x)|)^q\f{ds}{s}.
	$$
	This implies that 
	\[
	\begin{aligned}
	\f{V(Q)}{w(Q)^2}\int_Q&\sum_{j\ge \nu}(2^{j\alpha}|\psi_j(\sL)f(x)|)^q\dx\lesi \f{V(Q)}{w(Q)^2}\int_{0}^{2^{-\nu+2}}\int_Q(s^{-\alpha}|\psi^*_{\lambda}(s\sL)f(x)|)^q\dx\f{ds}{s}
	\end{aligned}
	\]
	for every  ball $Q$. Hence \eqref{eq1-prop2-Fvc} follows.
	
	\bigskip
	
	We next prove that
	\begin{equation}
	\label{eq2-prop2-Fvc}
	N^*_{\lambda,\alpha,q, L}(f)\lesi N_{\alpha,q, L}(f).
	\end{equation}
	To do this, we employ Proposition \ref{prop2-maximal functions} to find that
	\[
	|\psi^*_{\lambda}(s\sL)f(x)|\lesi \Big[\int_X \f{1}{V(z,s)}\Big(1+\f{d(x,z)}{s}\Big)^{-\lambda r}|\psi(s\sL)f(z)|^r\dz\Big]^{1/r}
	\]
	where $r=q/p$. Fix a ball $Q=B(x_1,t)$. Then the above implies that, for every $x\in Q$,
	\[
	\begin{aligned}
\int_0^t(s^{-\alpha}|\psi^*_{\lambda}(s\sL)f(x)|)^q\f{ds}{s}&\lesi \int_{0}^{t}\Big[\int_{4Q} \f{1}{V(z,s)}\Big(1+\f{d(x,z)}{s}\Big)^{-\lambda r}|s^{-\alpha}\psi(s\sL)f(z)|^r\dz\Big]^{q/r}\f{ds}{s}\\
	&+\int_{0}^{t}\Big[\int_{X\backslash 4Q} \f{1}{V(z,s)}\Big(1+\f{d(x,z)}{s}\Big)^{-\lambda r}|s^{-\alpha}\psi(s\sL)f(z)|^r\dz\Big]^{q/r}\f{ds}{s}.
	\end{aligned}
	\]
	Using the above estimates and an argument similar to the proof of Lemma \ref{lem2-Fvc} we obtain \eqref{eq2-prop2-Fvc}.
	
	\medskip
	
	It remains to show that
	\begin{equation}
	\label{eq3-prop2-Fvc}
	N_{\alpha,q, L}(f)\lesi \|f\|_{\F_{\vc,q,w}^{\alpha,L}}.
	\end{equation}
	
	Observe that for $t\in (2^{-\nu-1}, 2^{-\nu}]$ with $\nu \in \mathbb{Z}$,
	\begin{equation}\label{eq1-proof prop2 Fvc}
	\begin{aligned}
	\f{V(x,t)}{w(B(x,t))^2}&\int_{B(x,t)}\int_0^{t}(s^{-\alpha}|\psi(s\sL)f(y)|)^q \f{ds}{s}dy\\
	&\lesi \f{V(x,2^{-\nu})}{w(B(x,2^{-\nu}))^2}\int_{B(x,2^{-\nu})}\int_0^{2^{-\nu}}(s^{-\alpha}|\psi(s\sL)f(y)|)^q \f{ds}{s}dy\\
	&\lesi \f{V(x,2^{-\nu})}{w(B(x,2^{-\nu}))^2}\int_{B(x,2^{-\nu})}\sum_{j\geq \nu}\int_{2^{-j-1}}^{2^{-j}}(s^{-\alpha}|\psi(s\sL)f(y)|)^q\f{ds}{s}dy.
	\end{aligned}
	\end{equation}
	By Proposition \ref{prop1-maximal function}, for $\lambda>0$ and $s\in [2^{-j-1}, 2^{-j}]$ we have
	\begin{equation*}
	|\psi(s\sL)f(x)|\lesi \sum_{k=j-2}^{j+3}\psi_{k,\lambda}^*f(x), \ \ x\in X.
	\end{equation*}
	Substituting this into \eqref{eq1-proof prop2 Fvc} and using Lemma \ref{lem2-Fvc}, we obtain  \eqref{eq3-prop2-Fvc}.
	
	\bigskip
	
	The conclusion of the proposition follows immediately from \eqref{eq1-prop2-Fvc}, \eqref{eq2-prop2-Fvc} and \eqref{eq3-prop2-Fvc}. 
\end{proof}

\bigskip

\begin{thm}\label{prop3-Fvc}
	Let $0<q\leq \vc, \alpha\in \mathbb{R}, m>\alpha/2$, and $w\in A_p, 1\leq p <\infty$. Let $\varphi\in \mathscr{S}_m(\mathbb{R})$. Assume  $\lambda>np/q+2np^2/q$.  	\begin{enumerate}[\rm{(i)}]
	\item If $p=1$, then for each $f\in \mathcal S'$ there exists $\rho\in \mathscr P$ so that 
	\begin{equation}
	\label{eq1- sm pvc}
	K^*_{\lambda,\alpha,q, L}(f-\rho):=\sup_{x\in X, t>0}\Big(\f{V(x,t)}{w(B(x,t))^2}\int_{B(x,t)}\int_0^{t}(s^{-\alpha}|\varphi^*_\lambda(s\sL)(f-\rho)(y)|)^q \f{ds}{s}\dy\Big)^{1/q} \lesi \|f\|_{\F_{\vc,q,w}^{\alpha,L}}
	\end{equation} 
	\item If $p>1$, then 
	\begin{equation}
	\label{eq2- sm pvc}
	\|f\|_{\F_{\vc,q,w}^{\alpha,L}}\lesi \sup_{x\in X, t>0}\Big(\f{V(x,t)}{w(B(x,t))^2}\int_{B(x,t)}\int_0^{t}(s^{-\alpha}|\varphi(s\sL)f(y)|)^q \f{ds}{s}\dy\Big)^{1/q}=:K_{\alpha,q, L}(f)
	\end{equation}
	for every $f\in \mathcal S'$.
	\end{enumerate}
	Moreover, the distribution $\rho$ in \eqref{eq1- sm pvc} can be omitted if $f\in L^2(X)$.
\end{thm}
\begin{proof}
	We will only prove the theorem for $0<q<\vc$, as the case $q=\vc$ is similar and  so we omit the details.
	
	\medskip

	We prove \eqref{eq1- sm pvc} first. Recall from \eqref{eq1-heat kernel charac} that we can find $\rho\in \mathscr P$ such that for each $u\in [2^{-\nu-1},2^{-\nu})$  for some $\nu\in \mathbb{Z}$ and $\lambda>0$,
	\begin{equation*}%\label{eq1-proof pvc}
	\begin{aligned}
	|\varphi^*_\lambda(u\sL)(f-\rho)|&\lesi \sum_{j\in \mathbb Z}2^{-2m|\nu-j|} |\psi^*_{j,\lambda}(\sL)f|.
	\end{aligned}
	\end{equation*}	
	It follows that if $q\le 1$,  then  	we have
	\begin{equation}\label{eq2-proof pvc}
	\begin{aligned}
	\int_{2^{-\nu-1}}^{2^{-\nu}}(u^{-\alpha}|\varphi^*_\lambda(u\sL)(f-\rho)|)^q\f{du}{u}&\lesi \sum_{j\in \mathbb Z}2^{-q(2m-\alpha)|\nu-j|} (2^{j\alpha}|\psi^*_{j,\lambda}(\sL)f|)^q
	\end{aligned}
	\end{equation}
	for every $\nu\in \mathbb{Z}$.
	Therefore, for any $t>0$ so that $t\in [2^{-\nu_0-1}, 2^{-\nu_0})$, and $x\in X$, we have
	\[
	\begin{aligned}
\f{V(x,t)}{w(B(x,t))^2}&\int_{B(x,t)}\int_0^{t}(u^{-\alpha}|\varphi^*_\lambda(u\sL)(f-\rho)|)^q\f{du}{u}d\mu\\
	&\lesi \f{V(x,t)}{w(B(x,t))^2}\int_{B(x,t)}\sum_{\nu\ge \nu_0}\int_{2^{-\nu-1}}^{2^{-\nu}}(u^{-\alpha}|\varphi^*_\lambda(u\sL)(f-\rho)|)^q\f{du}{u}d\mu \\
	&\lesi \f{V(x,t)}{w(B(x,t))^2}\int_{B(x,t)}\sum_{\nu\ge \nu_0}\sum_{j\in \mathbb{Z}}2^{-q(2m-\alpha)|\nu-j|} (2^{j\alpha}|\psi^*_{j,\lambda}(\sL)f|)^qd\mu \\
	&\lesi \f{V(x,t)}{w(B(x,t))^2}\int_{B(x,t)}\sum_{\nu\ge \nu_0}\sum_{j\ge \nu_0}\ldots \, +\f{V(x,t)}{w(B(x,t))^2}\int_{B(x,t)}\sum_{\nu\ge \nu_0}\sum_{j< \nu_0}\ldots \,\\
	&=: I_1 + I_2.
	\end{aligned}
	\]
	
	Assume $w\in A_1$ and $\lambda>n/q+2n/q$.  Choose $p>1$ such that $\lambda>np/q+2np^2/q$. Then, since $w\in A_p$, using Lemma \ref{lem2-Fvc} we deduce that 
	\begin{equation}\label{I1}
	\begin{aligned}
	I_1\lesi \|f\|^q_{\F_{\vc,q,w}^{\alpha,L}}.
	\end{aligned}
	\end{equation}
	
For the term $I_2$ we have
\begin{equation}\label{I2}
	\begin{aligned}
I_2&= \f{V(x,t)}{w(B(x,t))^2}\int_{B(x,t)}\sum_{\nu\ge \nu_0}2^{-q(2m-\alpha)(\nu-\nu_0)}\sum_{j<\nu_0}2^{-q(2m-\alpha)(\nu_0-j)} (2^{j\alpha}|\psi^*_{j,\lambda}(\sL)f|)^qd\mu\\
&\lesi \sum_{j<\nu_0}2^{-q(2m-\alpha)(\nu_0-j)}\f{V(x,t)}{w(B(x,t))^2}\int_{B(x,t)} (2^{j\alpha}|\psi^*_{j,\lambda}(\sL)f|)^qd\mu.
	\end{aligned}
\end{equation}
We now claim that 
\begin{equation}\label{eq-TimBui inequality}
\f{1}{V(x,t)}\int_{B(x,t)}(2^{j\alpha}|\psi^*_{j,\lambda}(\sL)f|)^q\dy\lesi \f{1}{V(x,2^{\nu_0-j}t)}\int_{B(x,2^{\nu_0-j}t)}(2^{j\alpha}|\psi^*_{j,\lambda}(\sL)f|)^q\dy.
\end{equation}
Indeed, for any $y\in B(x,t)$, and $y'\in B(x,2^{\nu_0-j}t)$ with $j<\nu_0$, we have $d(y,y')\lesi 2^{-j}$. Hence,
\[
\begin{aligned}
\psi^*_{j,\lambda}(\sL)f(y)&=\sup_{z\in X}\f{|\psi_{j,\lambda}(\sL)f(z)|}{(1+2^jd(y,z))^\lambda}\sim \sup_{z\in X}\f{|\psi_{j,\lambda}(\sL)f(z)|}{(1+2^jd(y,z)+2^jd(y,y'))^\lambda}\\
&\le \sup_{z\in X}\f{|\psi_{j,\lambda}(\sL)f(z)|}{(1+2^jd(z,y'))^\lambda}\\
&=\psi^*_{j,\lambda}(\sL)f(y'),
\end{aligned}
\]
so that $\sup_{y\in B(x,t)}\psi^*_{j,\lambda}(\sL)f(y) \lesi \inf_{y'\in  B(x,2^{\nu_0-j}t)} \psi^*_{j,\lambda}(\sL)f(y')$.
Thus \eqref{eq-TimBui inequality} follows.

Using \eqref{eq-TimBui inequality} and Lemma \ref{lem2-Fvc}, for each $j<\nu_0$ we have
\begin{equation}\label{I2b}
\begin{aligned}
\f{V(x,t)}{w(B(x,t))^2}\int_{B(x,t)} (2^{j\alpha}|\psi^*_{j,\lambda}(\sL)f|)^qd\mu& \lesi \f{V(x,t)^2}{w(B(x,t))^2}\f{1}{V(x,2^{\nu_0-j}t)}\int_{B(x,2^{\nu_0-j}t)} (2^{j\alpha}|\psi^*_{j,\lambda}(\sL)f|)^qd\mu\\
& \lesi \f{V(x,2^{\nu_0-j}t)}{w(B(x,2^{\nu_0-j}t))^2}\int_{B(x,2^{\nu_0-j}t)} (2^{j\alpha}|\psi^*_{j,\lambda}(\sL)f|)^qd\mu\\
&\lesi \|f\|^q_{\F_{\vc,q,w}^{\alpha,L}}
\end{aligned}
\end{equation}
where in the second inequality we used the fact that 
\[
\f{V(x,t)}{w(B(x,t))}\lesi \f{V(x,2^{\nu_0-j}t)}{w(B(x,2^{\nu_0-j}t))},
\]
which is a consequence the $A_1$-condition (see Lemma~\ref{weightedlemma1}).
Inserting \eqref{I2b} into \eqref{I2} yields
\[
I_2\lesi \|f\|^q_{\F_{\vc,q,w}^{\alpha,L}}.
\]
The above estimate for $I_2$ and \eqref{I1} imply \eqref{eq1- sm pvc}.
	
	If $q>1$, using Young's inequality and a similar argument, we also obtain that
	\[
	\begin{aligned}
	\f{V(x,t)}{w(B(x,t))^2}\int_{B(x,t)}\int_0^{t}(u^{-\alpha}|\varphi^*_\lambda(u\sL)(f-\rho)|)^q\f{du}{u}\dx\lesi \|f\|^q_{\F_{\vc,q,w}^{\alpha,L}}.
	\end{aligned}
	\]
	This completes the proof of  \eqref{eq1- sm pvc} in all cases.
	
	\bigskip

It remains to prove \eqref{eq2- sm pvc}.	 
	We first show that 
	\begin{equation}\label{eq1b-proof pvc}
	\|f\|_{\F_{\vc,q,w}^{\alpha,L}}\lesi K^*_{\lambda,\alpha,q, L}(f).
	\end{equation}
	To do this, let $\psi$ be a partition of unity. Assume that $t\in (2^{-\nu+1}, 2^{-\nu+2}]$ with $\nu\in \mathbb{Z}$.  
		By Proposition \ref{prop3-maximal functions}, 
	for any $\lambda>0$,
	 	\[
	\begin{aligned}
	\sum_{j\geq \nu}(2^{j\alpha}|\psi_j(\sL)f(x)|)^q&\lesi \sum_{j\geq \nu}\int_{2^{-j-2}}^{2^{-j+2}}|t^{-\alpha}\varphi^*_{\lambda}(s\sL)f(x)|^q\f{ds}{s}\\
	&\lesi \int_{0}^{2^{-\nu+2}}|t^{-\alpha}\varphi^*_{\lambda}(s\sL)f(x)|^q\f{ds}{s}.
	\end{aligned}
	\]
	This clearly implies \eqref{eq1b-proof pvc}.
	 
	Therefore,  to complete the proof of \eqref{eq2- sm pvc}, it suffices to prove that
	%We next prove that
	\[
	K^*_{\lambda,\alpha,q,L}(f)\lesi K_{\alpha,q,L}(f).
	\]
	To this end, fix $A>\lambda+\alpha$. Then for each $j\ge \nu$ and $t\in [1,2] $, recall that
	$$
	\varphi_\lambda^{**}(2^{-j}t\sL)(L)f(x)=\sup_{k\ge j}\sup_{y\in X}2^{-(k-j)A}\f{|\varphi(2^{-k}t\sL)f(y)|}{(1+2^{j}d(x,y))^\lambda}.
	$$
	
	Applying \eqref{eq-varphi*} with $r=q/p$ gives
	\[
	\begin{aligned}
	|\varphi^*_{\lambda}(2^{-j}t\sL)f(x)|^{r}&\lesi \sum_{k\geq j}2^{-(k-j)(A-\lambda)r}
	\int_{X} \f{|\varphi(2^{-k}t\sL)f(z)|^r}{V(z,2^{-k})(1+2^kd(x,z))^{\lambda r}}\dz.
	\end{aligned}
	\]
	Since $p=q/r>1$, it follows that
	\begin{equation}\label{eq1-proof lem4 Fvc}
	\begin{aligned}
	\Big[\int_1^2&|\varphi^*_{\lambda}(2^{-j}t\sL)f(x)|^{q}\f{dt}{t}\Big]^{r/q}\\
	&\lesi \sum_{k\geq j}2^{-(k-j)(A-\lambda)r}
	\Big\{\int_1^2\Big[\int_{X} \f{|\varphi(2^{-k}t\sL)f(z)|^r}{V(z,2^{-k})(1+2^kd(x,z))^{\lambda r}}\dz\Big]^{q/r}\f{dt}{t}\Big\}^{r/q}.
	\end{aligned}
	\end{equation}
	
	Fix a ball $Q$ and let $x\in Q$. 
	The argument in the proof of Lemma \ref{lem2-Fvc} then shows that
	\[
	\begin{aligned}
	\Big[\int_{X} &\f{|\varphi(2^{-k}t\sL)f(z)|^r}{V(z,2^{-k})(1+2^kd(x,z))^{\lambda r}}\dz\Big]^{q/r}\\
	&\lesi \left[\mathcal{M}_r(|\varphi(2^{-k}t\sL)f|\chi_{4Q})(x)\right]^q+\sum_{\ell=3}^\vc \f{2^{-\ell(\lambda q/p-n)}}{V(2^\ell Q)}\int_{S_\ell (Q)} |\varphi(2^{-k}t\sL)f(z)|^{q}\dz.
	\end{aligned}
	\]
	Inserting this into \eqref{eq1-proof lem4 Fvc}, we obtain
	\begin{equation*}
	\begin{aligned}
	\Big[\int_1^2&|\varphi^*_{\lambda}(2^{-j}t\sL)f(x)|^{q}\f{dt}{t}\Big]^{r/q}\\
	&\lesi \sum_{k\geq j}2^{-(k-j)(A-\lambda)r}
	\Big\{\int_1^2\left[\mathcal{M}_r(|\varphi(2^{-k}t\sL)f|\chi_{4Q})(x)\right]^q\f{dt}{t}\Big\}^{r/q}\\
	& \ \ \ \ +\sum_{k\geq j}2^{-(k-j)(A-\lambda)r}
	\Big\{\int_1^2\sum_{\ell=3}^\vc \f{2^{-\ell(\lambda q/p-n)}}{V(2^\ell Q)}\int_{S_\ell (Q)} |\varphi(2^{-k}t\sL)f(z)|^{q}\dz\f{dt}{t}\Big\}^{r/q}.
	\end{aligned}
	\end{equation*}
	This implies that, for any $\nu\in \mathbb{Z}$, 
	\begin{equation*}
	\begin{aligned}
	\sum_{j\ge \nu}&\int_{2^{-j}}^{2^{-j+1}}(2^{j\alpha}|\varphi^*_{\lambda}(t\sL)f(x)|)^{q}\f{dt}{t}\\
	&\lesi \sum_{j\ge \nu}\Big(\sum_{k\geq j}2^{-(k-j)(A-\lambda-\alpha)r}
	\Big\{\int_{2^{-k}}^{2^{-k+1}}2^{k\alpha q}\left[\mathcal{M}_r(|\varphi(t\sL)f|\chi_{4Q})(x)\right]^q\f{dt}{t}\Big\}^{r/q}\Big)^{q/r}\\
	& \ \ +\sum_{j\ge \nu}\Big(\sum_{k\geq j}2^{-(k-j)(A-\lambda-\alpha)r}
	\Big\{\int_{2^{-k}}^{2^{-k+1}}\sum_{\ell=3}^\vc \f{2^{-\ell(\lambda q/p-n)}}{V(2^\ell Q)}\int_{S_\ell (Q)} (2^{k\alpha}|\varphi(t\sL)f(z)|)^{q}\dz\f{dt}{t}\Big\}^{r/q}\Big)^{q/r}.	\end{aligned}
	\end{equation*}
	We now use Young's inequality to get the bound:
	\begin{equation*}
	\begin{aligned}
	\sum_{j\ge \nu}&\int_{2^{-j}}^{2^{-j+1}}(t^{-\alpha}|\varphi^*_{\lambda}(t\sL)f(x)|)^{q}\f{dt}{t}\\
	&\lesi \sum_{k\geq \nu}\int_{2^{-k}}^{2^{-k+1}}t^{-\alpha q}\left[\mathcal{M}_r(|\varphi(t\sL)f|\chi_{4Q})(x)\right]^q\f{dt}{t}\\
	& \ \ +\sum_{k\geq \nu}\int_{2^{-k}}^{2^{-k+1}}\sum_{\ell=3}^\vc \f{2^{-\ell(\lambda q/p-n)}}{V(2^\ell Q)}\int_{S_\ell (Q)} (t^{-\alpha}|\varphi(t\sL)f(z)|)^{q}\dz\f{dt}{t}.	\end{aligned}
	\end{equation*}
	It follows that 
	\begin{equation*}
	\begin{aligned}
	\int_{0}^{2^{-\nu+1}}(t^{-\alpha}|\varphi^*_{\lambda}(t\sL)f(x)|)^{q}\f{dt}{t}
	&\lesi \int_{0}^{2^{-\nu+1}}t^{-\alpha q}\left[\mathcal{M}_r(|\varphi(t\sL)f|\chi_{4Q})(x)\right]^q\f{dt}{t}\\
	& \ \ +\int_{0}^{2^{-\nu+1}}\sum_{\ell=3}^\vc \f{2^{-\ell(\lambda q/p-n)}}{V(2^\ell Q)}\int_{S_\ell (Q)} (t^{-\alpha}|\varphi(t\sL)f(z)|)^{q}\dz\f{dt}{t}.	\end{aligned}
	\end{equation*}
	Using the above estimate and arguing similarly to the proof of Lemma \ref{lem2-Fvc}, we deduce that $K^*_{\lambda,\alpha,q,L}(f)\lesi K_{\alpha,q,L}(f)$. This completes the proof of \eqref{eq2- sm pvc}, and hence the theorem follows.
	
\end{proof}

\begin{rem}
It is natural to question if the estimate \eqref{eq1- sm pvc} holds true for $w\in A_\vc$. We note that when $L=-\Delta$ is the Laplacian on $\mathbb{R}^n$, 
we show in Theorem \ref{BMOthm} that the classical weighted BMO space  ${\rm BMO}_{w}(\mathbb{R}^n)$ coincides with $\F^{0,-\Delta}_{\vc,2,w}(\mathbb{R}^n)$ for $w\in A_1\cap RH_2$. In this case, the estimate \eqref{eq1- sm pvc} is known as the Carleson measure condition and to the best of our knowledge, the problem 
of obtaining the estimate \eqref{eq1- sm pvc} on the classical weighted BMO space with all $w\in A_\vc$   is still open. In this sense, the restriction $w\in A_1$ is reasonable.
	
	\medskip

	If we replace the definition of $\|f\|_{\F_{\vc,q,w}^{\alpha,\psi,L}}$ in Definition \ref{de pvc} by the following quantity
		\[
		\sup_{Q: \ {\rm balls}}\Big(\f{1}{w(Q)}\int_{Q}\sum_{j\ge -\log_2 r_Q}^\vc (2^{j\alpha}|\psi_j(\sL)f(x)|)^q w(x)\dx\Big)^{1/q},
		\]
then by a similar argument we can show that \eqref{eq1- sm pvc} holds true for $w\in A_\vc$. However, one of the main drawbacks of this definition is that when  $L=-\Delta$, we have $\F^{0,-\Delta}_{\vc,2,w}(\mathbb{R}^n)\equiv {\rm BMO}(\mathbb{R}^n)$ for every $w\in A_\vc$. See for example \cite{Qui5}.
\end{rem}
\bigskip

We have the following results as a direct consequence of Theorem~ \ref{prop3-Fvc}:
\begin{cor}\label{cor-Fvc}
	Let $0<q\leq \vc, \alpha\in \mathbb{R}, m>\alpha/2$, and $w\in A_p, 1\le p<\infty$. Assume $\lambda>np/q+2np^2/q$. 
	Let $\Psi_{m,t}(L)=(t^2L)^me^{-t^2L}$  and
	\begin{equation}
	\label{eq-PetreeFunction-heatkernel}
	\Psi_{m,t,\lambda}^*(L)f(x)=\sup_{y\in X}\f{|\Psi_{m,t}(L)f(y)|}{(1+d(x,y)/t)^\lambda}.
	\end{equation}
	
	\begin{enumerate}[(i)]
	\item If $p=1$, then for each 
	 $f\in \mathcal S'$, there exists $\rho\in \mathscr P$ so that
	\[
	\sup_{x\in X, t>0}\Big(\f{V(x,t)}{w(B(x,t))^2}\int_{B(x,t)}\int_0^{t}(s^{-\alpha}|\Psi^*_{m,s,\lambda}(L)(f-\rho)(y)|)^q \f{ds}{s}dy\Big)^{1/q}\lesi \|f\|_{\F_{\vc,q,w}^{\alpha,L}}.
	\]
     \item If $p>1$, then 
	\[
	\begin{aligned}
	\|f\|_{\F_{\vc,q,w}^{\alpha,L}}&\lesi \sup_{x\in X, t>0}\Big(\f{V(x,t)}{w(B(x,t))^2}\int_{B(x,t)}\int_0^{t}(s^{-\alpha}|\Psi_{m,s}(L)f(y)|)^q \f{ds}{s}dy\Big)^{1/q}
	\end{aligned}
	\]
	for every $f\in \mathcal{S}'$.
	\end{enumerate}
	 	
	Moreover, the distribution $\rho$ can be removed if $f\in L^2(X)$.
\end{cor}

\section{Atomic decompositions}

We now prove atomic decomposition theorems for our new Besov and Triebel-Lizorkin   spaces. We first introduce the definition of weighted atoms related to $L$.

\begin{defn}\label{defLmol}
	Let $0< p\leq \vc$, $M\in \mathbb{N}_+$ and $w\in A_\vc$. A function $a$ is said to be an $(L, M, p, w)$ atom if there exists a dyadic cube $Q\in \mathcal{D}$ so that
	\begin{enumerate}[{\rm (i)}]
		\item $a=L^{M} b$;
		
		\item ${\rm supp} \,L^{k} b\subset 3B_Q$, $k=0,\ldots , 2M$;
		
		\item $\displaystyle |L^{k} b(x)|\leq \ell(Q)^{2(M-k)}w(Q)^{-1/p}$, $k=0,\ldots , 2M$;
	\end{enumerate}
where $B_Q$ is a ball associated to $Q$ defined in Remark \ref{rem1-Christ}.
\end{defn}

\subsection{Atomic decompositions for Besov spaces $\B^{\alpha,L}_{p,q,w}$.}
Our first main result in this section is the following theorem in which we show that each function in $\B^{\alpha,L}_{p,q,w}$ can be characterized in terms of atomic decompositions.
\begin{thm}\label{thm1- atom Besov}
	Let $\alpha\in \mathbb{R}$, $0<p,q\leq \vc$, $M\in \mathbb{N}_+$ and $w\in A_\infty$. Assume $f\in \B^{\alpha,L}_{p,q,w}$. Then there exist a sequence of $(L,M,p,w)$ atoms $\{a_Q\}_{Q\in \mathcal{D}_\nu, \nu\in \mathbb{Z}}$ and a sequence of coefficients  $\{s_Q\}_{Q\in \mathcal{D}_\nu, \nu\in\mathbb{Z}}$ so that
	$$
	f=\sum_{\nu\in\mathbb{Z}}\sum_{Q\in \mathcal{D}_\nu}s_Qa_Q \ \ \text{in $\mathcal{S}'_\vc$}.
	$$
	Moreover,
	$$
	\Big[\sum_{\nu\in\mathbb{Z}}2^{\nu \alpha q}\Big(\sum_{Q\in \mathcal{D}_\nu}|s_Q|^p\Big)^{q/p}\Big]^{1/q}\lesi \|f\|_{\B^{\alpha,L}_{p,q,w}}.
	$$
\end{thm}

\begin{proof}
	Let $\psi$ be a partition of unity and $\Phi$ be a function as in Lemma \ref{lem:finite propagation}. Due to Proposition \ref{prop-Calderon2}, for $f\in \mathcal{S}'_\vc$ we have
	$$
	f=c\int_0^\vc \psi(t\sL)\Phi(t\sL)f\f{dt}{t}
	$$
	in $\mathcal{S}'_\vc$, where $\displaystyle c=\Big[\int_0^\vc \psi(\xi)\Phi(\xi)\f{d\xi}{\xi}\Big]^{-1}$.
	
	As a consequence, by Lemma \ref{Christ'slemma} we have
	\begin{equation}
	\label{eq1-atom}
	\begin{aligned}
	f&=c\sum_{\nu\in \mathbb{Z}}\int_{2^{-\nu-1}}^{2^{-\nu}}(t^2L)^M\Phi(t\sL)[\psi_{M}(t\sL)f]\f{dt}{t}\\
	&=c\sum_{Q\in \mathcal{D}_\nu}\sum_{\nu\in \mathbb{Z}}\int_{2^{-\nu-1}}^{2^{-\nu}}(t^2L)^M\Phi(t\sL)[\psi_{M}(t\sL)f\cdot\chi_{Q}]\f{dt}{t}
	\end{aligned}
	\end{equation}
	where $\psi_{M}(\xi)=\xi^{-2M}\psi(\xi)$.

	For each $\nu\in \mathbb{Z}$ and $Q\in \mathcal{D}_\nu$, we set
	$$
	s_Q= w(Q)^{1/p}\sup_{y\in Q}\int_{2^{-\nu-1}}^{2^{-\nu}}|\psi_{M}(t\sL)f(y)|\f{dt}{t},
	$$
	and $a_Q=L^{M}b_Q$, where
	\begin{equation}\label{eq-bQ}
	b_Q=\f{1}{s_Q} \int_{2^{-\nu-1}}^{2^{-\nu}} t^{2M}\Phi(t\sqrt{L})[\psi_{M}(t\sL)f.\chi_Q]\f{dt}{t}.
	\end{equation}
	Obviously,  we deduce from \eqref{eq1-atom} that
	$$
	f=\sum_{\nu\in\mathbb{Z}}\sum_{Q\in \mathcal{D}_\nu}s_Qa_Q \ \ \text{in $\mathcal{S}'_\vc$}.
	$$
	For $k=0,\ldots,2M$, we have
	$$
	\begin{aligned}
	L^{k}b_Q(x)&=\f{1}{s_Q} \int_{2^{-\nu-1}}^{2^{-\nu}} t^{2(M-k)}(t^2L)^k\Phi(t\sqrt{L})[\psi_{M}(t\sL)f.\chi_Q]\f{dt}{t}\\
	&=\f{1}{s_Q} \int_{2^{-\nu-1}}^{2^{-\nu}}\int_Q t^{2(M-k)}K_{(t^2L)^k\Phi(t\sqrt{L})}(x,y)\psi_{M}(t\sL)f(y)d\mu(y)\f{dt}{t}.
	\end{aligned}
	$$
	Using the finite propagation property in Lemma \ref{lem:finite propagation} we can see that
	\[
	{\rm supp}\,L^{k}b\subset 3B_Q,
	\]
	and
	$$
	\begin{aligned}
	|L^{k}b_Q(x)|&\lesi \f{2^{-2\nu(M-k)}}{s_Q} \int_{2^{-\nu-1}}^{2^{-\nu}}\int_Q \f{1}{V(y,2^{-\nu})}|\psi_{M}(t\sL)f(y)|d\mu(y)\f{dt}{t}\\
	&\lesi 2^{-\nu(2M-2k)}w(Q)^{-1/p}.
	\end{aligned}
	$$
	It follows that $a_Q$ is (a multiple of) an $(L, M,p, w)$ atom.
	
	We now prove that
	$$
	\Big[\sum_{\nu\in\mathbb{Z}}2^{\nu\alpha q}\Big(\sum_{Q\in \mathcal{D}_\nu}|s_Q|^p\Big)^{q/p}\Big]^{1/q}\sim \|f\|_{\B^{\alpha,L}_{p,q,w}}.
	$$
	
	Indeed, for any $\lambda>0$ we note that
	\[
	\begin{aligned}
	s_Q&\sim w(Q)^{1/p}\sup_{y\in Q}\int_{2^{-\nu-1}}^{2^{-\nu}}|\psi_{M}(t\sL)f(y)|\f{dt}{t}\le \Big[\int_Q |F^*_{M,\lambda}(\sL)f(x)|^pw(x)\dx\Big]^{1/p}
	\end{aligned}
	\]
	where 
	\[
	F^*_{M,\lambda}(\sL)f(x)=\sup_{y\in X}\f{\displaystyle \int_{2^{-\nu-1}}^{2^{-\nu}}|\psi_{M}(t\sL)f(y)|\f{dt}{t}}{(1+2^\nu d(x,y))^\lambda}.
	\]
	As a consequence,
	\[
	\sum_{Q\in \mathcal{D}_\nu}|s_Q|^p\le \int_X |F^*_{M,s,\lambda}(\sL)f(x)|^pw(x)\dx.
	\]
	On the other hand, fixing a $m>\alpha/2$ and arguing similarly to \eqref{eq1-heat kernel charac} we obtain
	\begin{equation*}
	\begin{aligned}
		|F^*_{M,\lambda}(\sL)f|&\lesi \sum_{j\in \mathbb{Z}}2^{-2m|\nu-j|} \psi^*_{j,\lambda}(\sL)f.
	\end{aligned}
	\end{equation*}
	Therefore,
	\begin{equation}
	\label{eq1-atom Besov}
	\begin{aligned}
	\Big[\sum_{\nu\in\mathbb{Z}}2^{\nu\alpha q}\Big(\sum_{Q\in \mathcal{D}_\nu}|s_Q|^p\Big)^{q/p}\Big]^{1/q}&\leq \Big[\sum_{\nu\in\mathbb{Z}}2^{\nu\alpha q}\Big(\Big\|\sum_{j\in \mathbb{Z}}2^{-2m|\nu-j|} \psi^*_{j,\lambda}(\sL)f\Big\|_{p,w}\Big)^{q}\Big]^{1/q}\\
	&\leq \Big[\sum_{\nu\in\mathbb{Z}}\Big(\Big\|\sum_{j\in \mathbb{Z}}2^{-(2m-\alpha)|\nu-j|} 2^{j\alpha}\psi^*_{j,\lambda}(\sL)f\Big\|_{p,w}\Big)^{q}\Big]^{1/q}.
	\end{aligned}
	\end{equation}
	
	\noindent {\bf Case 1: $p\geq 1$.} We have
	\[
	\Big[\sum_{\nu\in\mathbb{Z}}2^{\nu\alpha q}\Big(\sum_{Q\in \mathcal{D}_\nu}|s_Q|^p\Big)^{q/p}\Big]^{1/q}\leq \Big[\sum_{\nu\in\mathbb{Z}}\Big(\sum_{j\in \mathbb{Z}}2^{-(2m-\alpha)|\nu-j|} 2^{j\alpha}\|\psi^*_{j,\lambda}(\sL)f\|_{p,w}\Big)^{q}\Big]^{1/q}.
	\]
	At this stage, if $q\geq 1$ we then use Young's inequality and Proposition \ref{prop2-thm1} to further imply
	\[
	\Big[\sum_{\nu\in\mathbb{Z}}2^{\nu\alpha q}\Big(\sum_{Q\in \mathcal{D}_\nu}|s_Q|^p\Big)^{q/p}\Big]^{1/q}\leq \Big[\sum_{j\in\mathbb{Z}}\Big( 2^{j\alpha}\|\psi^*_{j,\lambda}(\sL)f\|_{p,w}\Big)^{q}\Big]^{1/q}
	\sim \|f\|_{\B^{\alpha,L}_{p,q,w}}.
	\]
	Otherwise if $0<q<1$, we then have
	\[
	\begin{aligned}
	\Big[\sum_{\nu\in\mathbb{Z}}2^{\nu \alpha q}\Big(\sum_{Q\in \mathcal{D}_\nu}|s_Q|^p\Big)^{q/p}\Big]^{1/q}&\leq \Big[\sum_{\nu\in\mathbb{Z}}\sum_{j\in \mathbb{Z}}\Big(2^{-(2m-\alpha)|\nu-j|} 2^{j\alpha}\|\psi^*_{j,\lambda}(\sL)f\|_{p,w}\Big)^{q}\Big]^{1/q}\\
	&\leq \Big[\sum_{j\in\mathbb{Z}}\Big( 2^{j\alpha}\|\psi^*_{j,\lambda}(\sL)f\|_{p,w}\Big)^{q}\Big]^{1/q}
	\sim \|f\|_{\B^{\alpha,L}_{p,q,w}}
	\end{aligned}
	\]
	where in the last line we use Proposition \ref{prop2-thm1}.
	
	\medskip
	
	\noindent {\bf Case 2: $0<p< 1$.} From \eqref{eq1-atom Besov}, we obtain
	\[
	\Big[\sum_{\nu\in\mathbb{Z}}2^{\nu\alpha q}\Big(\sum_{Q\in \mathcal{D}_\nu}|s_Q|^p\Big)^{q/p}\Big]^{1/q}\leq \Big[\sum_{\nu\in\mathbb{Z}}\Big(\sum_{j\in \mathbb{Z}}2^{-p(2m-\alpha)|\nu-j|} 2^{jsp}\|\psi^*_{j,\lambda}(\sL)f\|^p_{p,w}\Big)^{q/p}\Big]^{1/q}.
	\]
	If $q/p\geq 1$, we then use Young's inequality and Proposition \ref{prop2-thm1} to get the estimate
	\[
	\Big[\sum_{\nu\in\mathbb{Z}}2^{\nu\alpha q}\Big(\sum_{Q\in \mathcal{D}_\nu}|s_Q|^p\Big)^{q/p}\Big]^{1/q}\leq \Big[\sum_{j\in\mathbb{Z}}\Big( 2^{j\alpha}\|\psi^*_{j,\lambda}(\sL)f\|_{p,w}\Big)^{q}\Big]^{1/q}
	\sim \|f\|_{\B^{\alpha,L}_{p,q,w}}.
	\]
	Otherwise if $0<q/p<1$, we then have
	\[
	\begin{aligned}
	\Big[\sum_{\nu\in\mathbb{Z}}2^{\nu\alpha q}\Big(\sum_{Q\in \mathcal{D}_\nu}|s_Q|^p\Big)^{q/p}\Big]^{1/q}&\leq \Big[\sum_{\nu\in\mathbb{Z}}\sum_{j\in \mathbb{Z}}\Big(2^{-(2m-\alpha)|\nu-j|} 2^{j\alpha}\|\psi^*_{j,\lambda}(\sL)f\|_{p,w}\Big)^{q}\Big]^{1/q}\\
	&\leq \Big[\sum_{j\in\mathbb{Z}}\Big( 2^{j\alpha}\|\psi^*_{j,\lambda}(\sL)f\|_{p,w}\Big)^{q}\Big]^{1/q}
	\sim \|f\|_{\B^{\alpha,L}_{p,q,w}}
	\end{aligned}
	\]
	where in the last line we use Proposition \ref{prop2-thm1}.

	This completes the  proof.
\end{proof}

\bigskip

Conversely, each atomic decomposition with suitable coefficients belong to the spaces  $\B^{\alpha,L}_{p,q,w}$.

\begin{thm}\label{thm2- atom Besov}
	Let $\alpha\in \mathbb{R}$, $0<p,q\leq \vc$ and $w\in A_\vc$. Assume that 
	$$
	f=\sum_{\nu\in\mathbb{Z}}\sum_{Q\in \mathcal{D}_\nu}s_Qa_Q \ \ \text{in $\mathcal{S}'_\vc$},
	$$
	where $\{a_Q\}_{Q\in \mathcal{D}_\nu, \nu\in \mathbb{Z}}$ is a sequence of $(L,M,p,w)$ atoms and $\{s_Q\}_{Q\in \mathcal{D}_\nu, \nu\in\mathbb{Z}}$ is a sequence of coefficients satisfying
	$$
	\Big[\sum_{\nu\in\mathbb{Z}}2^{\nu\alpha q}\Big(\sum_{Q\in \mathcal{D}_\nu}|s_Q|^p\Big)^{q/p}\Big]^{1/q}<\vc.
	$$
	Then $f\in \B^{\alpha,L}_{p,q,w}$ and
	$$
	\|f\|_{\B^{\alpha,L}_{p,q,w}} \lesi \Big[\sum_{\nu\in\mathbb{Z}}\Big(\sum_{Q\in \mathcal{D}_\nu}|s_Q|^p\Big)^{q/p}\Big]^{1/q}
	$$
	provided $M>\f{n}{2}+\f{1}{2}\max\{\alpha,\f{nq_w}{1\wedge p\wedge q}-\alpha\}$.
\end{thm}	
	
	Before coming to the proof of Theorem \ref{thm2- atom Besov} we need the following technical results.
\begin{lem}\label{lem1- thm2 atom Besov}
		Let $w\in A_{\tilde q}$, $N>n$, $\kappa\in [0,1]$, and $\eta, \nu \in \mathbb{Z}$, $\nu\geq \eta$. Assume that  $\{f_Q\}_{Q\in \mathcal{D}_\nu}$ is a  sequence of functions satisfying
		$$
		|f_Q(x)|\lesi \Big(\f{V(Q)}{V(x_Q,2^{-\eta})}\Big)^\kappa\Big(1+\f{d(x,x_Q)}{2^{-\eta}}\Big)^{-N}.
		$$
		Then for $\f{n\tilde q}{N}<r\leq 1$ and a sequence of numbers $\{s_Q\}_{Q\in \mathcal{D}_\nu}$, we have
		$$
		\sum_{Q\in \mathcal{D}_\nu}|s_Q|\,|f_Q(x)|\lesi 2^{n(\nu-\eta)(\tilde q/r-\kappa)}\mathcal{M}_{w,r}\Big(\sum_{Q\in \mathcal{D}_\nu}|s_Q|\chi_Q\Big)(x).
		$$
	\end{lem}
	\begin{proof}
		In the particular case when $\kappa=0$, $w\equiv 1$ and $X=\mathbb{R}^n$, this lemma was proved in \cite{FJ2}. We adapt some of this argument to our present situation.
		
		Fix $x\in X$. We set
		$$
		\mathcal{B}_0=\{Q\in \mathcal{D}_\nu: d(x,x_Q)\leq 2^{-\eta}\}, Q_0=\bigcup\limits_{Q\in \mathcal{B}_0} Q
		$$
		and
		$$
		\mathcal{B}_k=\{Q\in \mathcal{D}_\nu: 2^{k-\eta-1}<d(x,x_Q)\leq 2^{k-\eta}\}, Q_k=\bigcup\limits_{j: \ j\leq k}\bigcup\limits_{Q\in \mathcal{B}_j} Q, \ \ k\in \mathbb{N}_+.
		$$
		Then we write
		$$
		\begin{aligned}
		\sum_{Q\in \mathcal{D}_\nu}|s_Q|\,|f_Q(x)|&=\sum_{k\in \mathbb{N}}\sum_{Q\in \mathcal{B}_k}|s_Q|\,|f_Q(x)|\leq \sum_{k\in \mathbb{N}}\sum_{Q\in \mathcal{B}_k}|s_Q|\Big(\f{V(Q)}{V(x_Q,2^{-\eta})}\Big)^{\kappa}\Big(1+\f{d(x,x_Q)}{2^{-\eta}}\Big)^{-N}\\
		&=:\sum_{k\in \mathbb{N}}E_k.
		\end{aligned}
		$$
		For each $k\in \mathbb{N}$, we have
		\begin{equation}\label{eq1-sec4}
		\begin{aligned}
		E_k&\lesi \sum_{Q\in \mathcal{B}_k} 2^{-kN}\Big(\f{V(Q)}{V(x_Q,2^{-\eta})}\Big)^{\kappa}|s_Q| \leq 2^{-kN}\left[\sum_{Q\in \mathcal{B}_k} \Big(\f{V(Q)}{V(x_Q,2^{-\eta})}\Big)^{\kappa r}|s_Q|^r\right]^{1/r}\\
		&\lesi 2^{-kN}\left\{\int_{Q_k}\left[\sum_{Q\in \mathcal{B}_k}\Big(\f{V(Q)}{V(x_Q,2^{-\eta})}\Big)^{\kappa} w(Q)^{-1/r}|s_Q|\chi_Q(y)\right]^rw(y)d\mu(y)\right\}^{1/r}\\
		&\lesi 2^{-kN}\left\{\f{1}{w(Q_k)}\int_{Q_k}\left[\sum_{Q\in \mathcal{B}_k} \Big(\f{w(Q_k)}{w(Q)}\Big)^{1/r}\Big(\f{V(Q)}{V(x_Q,2^{-\eta})}\Big)^{\kappa}|s_Q|\chi_Q(y)\right]^rd\mu(y)\right\}^{1/r}.
		\end{aligned}
		\end{equation}
		It is easy to see that $V(Q_k)\sim V(x,2^{-\eta+k})\sim V(x_Q,2^{-\eta+k})$, for each $Q\in \mathcal{B}_k$. Therefore,
		$$
		\begin{aligned}
		\Big(\f{w(Q_k)}{w(Q)}\Big)^{1/r}\Big(\f{V(Q)}{V(x_Q,2^{-\eta})}\Big)^{\kappa}&\lesi \Big(\f{V(Q_k)}{V(Q)}\Big)^{\tilde q/r}\Big(\f{V(Q)}{V(x_Q,2^{-\eta})}\Big)^{\kappa}\\
		& =\Big(\f{V(Q_k)}{V(x_Q,2^{-\eta})}\Big)^{\tilde q/r}\Big(\f{V(x_Q,2^{-\eta})}{V(Q)}\Big)^{\tilde q/r-\kappa}\\
		&\lesi 2^{kn \tilde q/r}2^{n(\nu-\eta)(\tilde q/r-\kappa)}.
		\end{aligned}
		$$
		Inserting this into (\ref{eq1-sec4}) gives
		$$
		\begin{aligned}
		E_k&\lesi 2^{-kN}2^{kn \tilde q/r}2^{n(\nu-\eta)(\tilde q/r-\kappa)}\Big\{\f{1}{w(Q_k)}\int_{Q_k}\Big[\sum_{Q\in \mathcal{B}_k} |s_Q|\chi_Q(y)\Big]^rw(y)d\mu(y)\Big\}^{1/r}\\
		&\lesi 2^{-k(N-n\tilde q/r)}2^{n(\nu-\eta)(\tilde q/r-\kappa)}\mathcal{M}_{w,r}\Big(\sum_{Q\in \mathcal{B}_k} |s_Q|\chi_Q\Big)(x).
		\end{aligned}
		$$
		Since $r>\f{n\tilde q}{N}$, we find that
		$$
		\sum_{k\in \mathbb{N}}E_k\lesi 2^{n(\nu-\eta)(\tilde q/r-\kappa)}\mathcal{M}_{w,r}\Big(\sum_{Q\in \mathcal{D}_\nu} |s_Q|\chi_Q\Big)(x).
		$$
		
		This completes our proof.
	\end{proof}

	\begin{lem}\label{lem2-thm2 atom Besov}
		Let $\psi$ be a partition of unity and  let $a_Q$ be an $(L, M, p, w)$ atom with some $Q\in \mathcal{D}_\nu$. Then for any $t>0$ and  $N>0$ we have:
		\begin{equation}\label{eq- psi atom}
		|\psi(t\sL) a_Q(x)|\lesi \Big(\f{t}{2^{-\nu}}\wedge \f{2^{-\nu}}{t}\Big)^{2M-n}w(Q)^{-1/p} \Big(1+\f{d(x,x_Q)}{2^{-\nu}\vee t}\Big)^{-N}.
		\end{equation}
	\end{lem}
	\begin{proof}
		We now consider two cases: $t\le 2^{-\nu}$ and $t>2^{-\nu}$.

 		\noindent{\bf Case 1: $t\le 2^{-\nu}$.} Observe that
		$$
		\psi(t\sL) a_Q=t^{2M}\psi_M(t\sL)(L^Ma_Q) 
		$$
		where $\psi_M(\lambda)=\lambda^{-2M}\psi(\lambda)$.
		
		This, along with Lemma \ref{lem1} and the definition of the atoms, yields
		$$
		\begin{aligned}
		|\psi(t\sL) a_Q(x)|&\lesi \int_{3B_Q} \f{t^{2M}}{V(y,t)}\Big(1+\f{d(x,y)}{t}\Big)^{-N}|L^{K}a_Q(y)|d\mu(y)\\
		&\lesi \Big(\f{t}{2^{-\nu}}\Big)^{2M}w(Q)^{-1/p} \int_{3B_Q}\f{1}{V(y,t)}\Big(1+\f{d(x,y)}{t}\Big)^{-N}d\mu(y).
		\end{aligned}
		$$
		Note that, for $t\le 2^{-\nu}$ and $y\in 3B_Q$, we have
		\[
		\Big(1+\f{d(x,y)}{t}\Big)^{-N}\le \Big(1+\f{d(x,y)}{2^{-\nu}}\Big)^{-N}\sim \Big(1+\f{d(x,x_Q)}{2^{-\nu}}\Big)^{-N}.
		\]
		Therefore,
		$$
		\begin{aligned}
		|\psi(t\sL) a_Q(x)|&\lesi \Big(\f{t}{2^{-\nu}}\Big)^{2M}w(Q)^{-1/p}\Big(1+\f{d(x,x_Q)}{2^{-\nu}}\Big)^{-N} \f{V(3B_Q)}{V(y,t)}\\
		&\lesi \Big(\f{t}{2^{-\nu}}\Big)^{2M-n}w(Q)^{-1/p}\Big(1+\f{d(x,x_Q)}{2^{-\nu}}\Big)^{-N}
		\end{aligned}
		$$
		where in the last inequality we use \eqref{doubling2}. This leads us to \eqref{eq- psi atom}.
		
		\medskip
		
		\noindent{\bf Case 2: $t> 2^{-\nu}$.} We first write  $a_Q=L^{M}b_Q$. Hence,
		\[
		\psi(t\sL)a_Q=t^{-2M}\tilde \psi_M(t\sL)b_Q
		\]
		where $\tilde \psi_M(\lambda)=\lambda^{2M}\psi(\lambda)$.
		 
		This, along with Lemma \ref{lem1}, implies that
		$$
		\begin{aligned}
		|\psi(t\sL)a_Q(x)|&\lesi \int_{3B_Q}\f{t^{-2M}}{V(y,t)}\Big(1+\f{d(x,y)}{t}\Big)^{-N}|b_Q(y)|d\mu(y)\\
		&\lesi \Big(\f{2^{-\nu}}{t}\Big)^{2M} w(Q)^{-1/p}\int_{3B_Q}\f{1}{V(y,t)}\Big(1+\f{d(x,y)}{t}\Big)^{-N}d\mu(y).
		\end{aligned}
		$$
		Note that for $y\in 3B_Q$ and $t\ge 2^{-\nu}\sim \ell(Q)$ we have
		\[
		\Big(1+\f{d(x,y)}{t}\Big)^{-N}\sim \Big(1+\f{d(x,x_Q)}{t}\Big)^{-N}. 
		\]
		Hence, the above inequality simplifies into 
		$$
		\begin{aligned}
		|\psi(t\sL)a_Q(x)|
		&\lesi \Big(\f{2^{-\nu}}{t}\Big)^{2M} w(Q)^{-1/p}\Big(1+\f{d(x,x_Q)}{t}\Big)^{-N} \f{V(3B_Q)}{V(y,t)}\\
		&\lesi \Big(\f{2^{-\nu}}{t}\Big)^{2M} w(Q)^{-1/p}\Big(1+\f{d(x,x_Q)}{t}\Big)^{-N}.
		\end{aligned}
		$$
		Hence \eqref{eq- psi atom} follows.
	\end{proof}
	\medskip
	
	We are now ready to give the proof for Theorem \ref{thm2- atom Besov}.
	
		\begin{proof}[Proof of Theorem \ref{thm2- atom Besov}:] The proof can be done by using similar arguments to those in \cite{FJ1, FJ2}. However, for the sake of completeness, we will provide the details.
			
		Fix $\tilde q\in (q_w,\vc)$ and $r<\min\{1,p,q\}$ so that $w\in A_{\tilde q}$ and $M>n/2 +n\tilde{q}/r -s$. We now fix $N>\f{n\tilde{q}}{r}$. Let $\psi$ be a partition of unity. Since 
		$$
		f=\sum_{\nu\in\mathbb{Z}}\sum_{Q\in \mathcal{D}_\nu}s_Qa_Q \ \ \text{in $\mathcal{S}'_\vc$},
		$$
		we then, for each $j\in \mathbb{Z}$ have
		\[
		\begin{aligned}
		\psi_j(\sL)f&=\sum_{\nu\in\mathbb{Z}}\sum_{Q\in \mathcal{D}_\nu}s_Q\psi_j(\sL)a_Q\\
		&=\sum_{\nu: \nu\geq j}\sum_{Q\in \mathcal{D}_\nu}s_Q\psi_j(\sL)a_Q+\sum_{\nu: \nu< j}\sum_{Q\in \mathcal{D}_\nu}s_Q\psi_j(\sL)a_Q.
		\end{aligned}
		\]
		Using Lemma \ref{lem1- thm2 atom Besov} and Lemma \ref{lem2-thm2 atom Besov} we see that 
		\begin{equation}\label{eq-proof of reverse Besov}
		\begin{aligned}
		|\psi_j(\sL)f|&\lesi \sum_{\nu\in\mathbb{Z}}\sum_{Q\in \mathcal{D}_\nu}s_Q\psi_j(\sL)a_Q\\
		&=\sum_{\nu: \nu\geq j}2^{-(\nu-j)(2M-n-n\tilde q/r)}\mathcal{M}_{w,r}\Big(\sum_{Q\in \mathcal{D}_\nu}|s_Q|w(Q)^{-1/p}\chi_Q\Big)\\
		& \ \ \ \ +\sum_{\nu: \nu< j}2^{-(2M-n)(j-\nu)}\mathcal{M}_{w,r}\Big(\sum_{Q\in \mathcal{D}_\nu}|s_Q|w(Q)^{-1/p}\chi_Q\Big).
		\end{aligned}
		\end{equation}
		Therefore,
		\[
		\begin{aligned}
		\|f\|_{\B^{\alpha,L}_{p,q,w}}&:=\Big[\sum_{j\in \mathbb{Z}}(2^{j\alpha}\|\psi_j(\sL)f\|_{p,w})^q\Big]^{1/q}\\	&\lesi\left[\sum_{j\in \mathbb{Z}}\left( \Big\|\sum_{\nu: \nu\geq j}2^{-(\nu-j)(2M-n-n\tilde q/r+\alpha)}\mathcal{M}_{w,r}\Big(\sum_{Q\in \mathcal{D}_\nu}2^{\nu \alpha}|s_Q|w(Q)^{-1/p}\chi_Q\Big)\Big\|_{p,w}\right)^q\right]^{1/q}\\
		& \ \ \ \ +\left[\sum_{j\in \mathbb{Z}}\left( \Big\|\sum_{\nu: \nu< j}2^{-(2M-n-\alpha)(j-\nu)}\mathcal{M}_{w,r}\Big(\sum_{Q\in \mathcal{D}_\nu}2^{\nu \alpha}|s_Q|w(Q)^{-1/p}\chi_Q\Big)\Big\|_{p,w}\right)^q\right]^{1/q}\\
		&=:E_1+E_2.
		\end{aligned}
		\]
If $p\ge 1$, then we have
\[
E_1\lesi \left[\sum_{j\in \mathbb{Z}}\left( \sum_{\nu: \nu\geq j}2^{-(\nu-j)(2M-n-n\tilde q/r+\alpha)}\Big\|\mathcal{M}_{w,r}\Big(\sum_{Q\in \mathcal{D}_\nu}2^{\nu \alpha}|s_Q|w(Q)^{-1/p}\chi_Q\Big)\Big\|_{p,w}\right)^q\right]^{1/q}.
\]		
We  now apply Young's inequality when $q\geq 1$ and the inequality $(\sum_j |a_j|)^q\leq \sum_j|a_j|^q$ when $0<q<1$ to simplify
\[
E_1\lesi \left[\sum_{\nu\in \mathbb{Z}}\left(\Big\|\mathcal{M}_{w,r}\Big(\sum_{Q\in \mathcal{D}_\nu}2^{\nu \alpha}|s_Q|w(Q)^{-1/p}\chi_Q\Big)\Big\|_{p,w}\right)^q\right]^{1/q}
\]
as long as $2M>n\tilde q /r-\alpha$.

On the other hand, since the maximal function $\mathcal{M}_{w,r}$ is bounded on $L^{p}_w(X)$ as $p>r$, we have
\[
\Big\|\mathcal{M}_{w,r}\Big(\sum_{Q\in \mathcal{D}_\nu}2^{\nu \alpha}|s_Q|w(Q)^{-1/p}\chi_Q\Big)\Big\|_{p,w}\lesi \Big\|\sum_{Q\in \mathcal{D}_\nu}2^{\nu \alpha}|s_Q|w(Q)^{-1/p}\chi_Q\Big\|_{p,w}\sim \Big(\sum_{Q\in \mathcal{D}_\nu}2^{\nu \alpha p}|s_Q|\Big)^{1/p}.
\]
As a consequence,
\[
E_1\lesi \Big[\sum_{\nu\in \mathbb{Z}}\Big(\sum_{Q\in \mathcal{D}_\nu}2^{\nu \alpha p}|s_Q|\Big)^{q/p}\Big]^{1/q}=\Big[\sum_{\nu\in \mathbb{Z}}2^{\nu\alpha q}\Big(\sum_{Q\in \mathcal{D}_\nu}|s_Q|\Big)^{q/p}\Big]^{1/q}.
\]
Similarly,
\[
E_2\lesi \Big[\sum_{\nu\in \mathbb{Z}}\Big(\sum_{Q\in \mathcal{D}_\nu}2^{\nu \alpha p}|s_Q|\Big)^{q/p}\Big]^{1/q}=\Big[\sum_{\nu\in \mathbb{Z}}2^{\nu\alpha q}\Big(\sum_{Q\in \mathcal{D}_\nu}|s_Q|\Big)^{q/p}\Big]^{1/q}.
\]
Hence,
\[
\|f\|_{\B^{\alpha,L}_{p,q,w}}\lesi \Big[\sum_{\nu\in \mathbb{Z}}2^{\nu\alpha q}\Big(\sum_{Q\in \mathcal{D}_\nu}|s_Q|\Big)^{q/p}\Big]^{1/q}
\]
as desired.

\medskip

If $0<p<1$, then we have
\[
E_1\lesi \left[\sum_{j\in \mathbb{Z}}\left( \sum_{\nu: \nu< j}2^{-p(\nu-j)(2M-n-n\tilde q/r+\alpha)}\Big\|\mathcal{M}_{w,r}\Big(\sum_{Q\in \mathcal{D}_\nu}2^{\nu \alpha}|s_Q|w(Q)^{-1/p}\chi_Q\Big)\Big\|^p_{p,w}\right)^{q/p}\right]^{1/q}.
\]
Arguing similarly to the case $p\ge 1$ by considering two cases $q/p\ge 1$ and $0<q/p<1$, we come up with

\[
E_1\lesi \Big[\sum_{\nu\in \mathbb{Z}}2^{\nu\alpha q}\Big(\sum_{Q\in \mathcal{D}_\nu}|s_Q|\Big)^{q/p}\Big]^{1/q}.
\]
By the same manner, we have
\[
E_2\lesi \Big[\sum_{\nu\in \mathbb{Z}}2^{\nu\alpha q}\Big(\sum_{Q\in \mathcal{D}_\nu}|s_Q|\Big)^{q/p}\Big]^{1/q}.
\]

Therefore,
\[
\|f\|_{\B^{\alpha,L}_{p,q,w}}\lesi \Big[\sum_{\nu\in \mathbb{Z}}2^{\nu\alpha q}\Big(\sum_{Q\in \mathcal{D}_\nu}|s_Q|\Big)^{q/p}\Big]^{1/q}
\]
as desired.

\end{proof}

\subsection{Atomic decompositions for Triebel-Lizorkin spaces $\F^{\alpha,L}_{p,q,w}$.}
Our second main result is the following  atomic decomposition theorem for the spaces $\F^{\alpha,L}_{p,q,w}$. More precisely, we prove the following theorem.
\begin{thm}\label{thm1- atom TL spaces}
	Let $\alpha\in \mathbb{R}$, $0<p<\vc$, $0<q\leq \vc$, $m\in \mathbb{N}_+$ and $w\in A_\vc$. If $f\in \F^{\alpha,L}_{p,q,w}$,  then there exist a sequence of $(L,M,p,w)$ atoms $\{a_Q\}_{Q\in \mathcal{D}_\nu, \nu\in \mathbb{Z}}$ and a sequence of coefficients  $\{s_Q\}_{Q\in \mathcal{D}_\nu, \nu\in\mathbb{Z}}$ so that
	$$
	f=\sum_{\nu\in\mathbb{Z}}\sum_{Q\in \mathcal{D}_\nu}s_Qa_Q \ \ \text{in $\mathcal{S}'_\vc$}.
	$$
	Moreover,
	\begin{equation}\label{eq1-thm1 atom TL space}
		\Big\|\Big[\sum_{\nu\in\mathbb{Z}}2^{\nu\alpha q}\Big(\sum_{Q\in \mathcal{D}_\nu}w(Q)^{-1/p}|s_Q|\chi_Q\Big)^q\Big]^{1/q}\Big\|_{p,w}\lesi \|f\|_{\F^{\alpha,L}_{p,q,w}}.
\end{equation}
\end{thm}
\begin{proof}
	Recall that in the proof of Theorem \ref{thm1- atom Besov}, we have proved the representation
	$$
	f=\sum_{\nu\in\mathbb{Z}}\sum_{Q\in \mathcal{D}_\nu}s_Qa_Q \ \ \text{in $\mathcal{S}'_\vc$},
	$$
	where
	$$
	s_Q= w(Q)^{1/p}\sup_{y\in Q}\int_{2^{-\nu-1}}^{2^{-\nu}}|\psi_{M}(t\sL)f(y)|\f{dt}{t},
	$$
	and $a_Q=L^{M}b_Q$ is an $(L,M,p,w)$ atom defined by
	$$
	b_Q=\f{1}{s_Q} \int_{2^{-\nu-1}}^{2^{-\nu}} t^{2M}\Phi(t\sqrt{L})[\psi_{M}(t\sL)f.\chi_Q]\f{dt}{t}.
	$$
	
	It remains to prove \eqref{eq1-thm1 atom TL space}. Indeed, for any $\lambda>0$ it is easy to see that
	\[
	\begin{aligned}
	w(Q)^{-1/p}s_Q\chi_Q&= \sup_{y\in Q}\int_{2^{-\nu-1}}^{2^{-\nu}}|\psi_{M}(t\sL)f(y)|\f{dt}{t}\cdot \chi_Q\lesi \chi_Q F^*_{M,\lambda}(\sL)f
	\end{aligned}
	\]
	where 
	\[
	F^*_{M,\lambda}(\sL)f(x)=\sup_{y\in X}\f{\displaystyle \int_{2^{-\nu-1}}^{2^{-\nu}}|\psi_{M}(t\sL)f(y)|\f{dt}{t}}{(1+2^\nu d(x,y))^\lambda}.
	\]
	As a consequence,
	\[
	\sum_{Q\in \mathcal{D}_\nu}w(Q)^{-1/p}|s_Q|\chi_Q\lesi F^*_{M,\lambda}(\sL)f.
	\]
	On the other hand, fixing a $m>\alpha/2$ and arguing similarly to \eqref{eq1-heat kernel charac} we show that
	\begin{equation*}
	\begin{aligned}
	|F^*_{M,\lambda}(\sL)f|&\lesi \sum_{j\in \mathbb{Z}}2^{-2m|\nu-j|} \psi^*_{j,\lambda}(\sL)f.
	\end{aligned}
	\end{equation*}
	Therefore,
	\[
	\begin{aligned}
	\Big\|\Big[\sum_{\nu\in\mathbb{Z}}2^{\nu\alpha q}\Big(\sum_{Q\in \mathcal{D}_\nu}w(Q)^{-1/p}|s_Q|\chi_Q\Big)^q\Big]^{1/q}\Big\|_{p,w}&\lesi \Big\|\Big[\sum_{\nu\in\mathbb{Z}}2^{\nu\alpha q}\Big(\sum_{j\in \mathbb{Z}}2^{-2m|\nu-j|} \psi^*_{j,\lambda}(\sL)f\Big)^q\Big]^{1/q}\Big\|_{p,w}\\
	&\lesi \Big\|\Big[\sum_{\nu\in\mathbb{Z}}\Big(\sum_{j\in \mathbb{Z}}2^{-2m|\nu-j|+\alpha(\nu-j)} 2^{j\alpha}\psi^*_{j,\lambda}(\sL)f\Big)^q\Big]^{1/q}\Big\|_{p,w}.
	\end{aligned}
	\]	
	We then apply Young's inequality when $q\geq 1$ and the inequality $(\sum_j |a_j|)^q\leq \sum_j|a_j|^q$ when $0<q<1$ to find that
	\[
	\begin{aligned}
	\Big\|\Big[\sum_{\nu\in\mathbb{Z}}2^{\nu\alpha q}\Big(\sum_{Q\in \mathcal{D}_\nu}w(Q)^{-1/p}|s_Q|\chi_Q\Big)^q\Big]^{1/q}\Big\|_{p,w}
	&\lesi \Big\|\Big[\sum_{j\in\mathbb{Z}}\Big( 2^{j\alpha}\psi^*_{j,\lambda}(\sL)f\Big)^q\Big]^{1/q}\Big\|_{p,w}\\
	&\lesi \|f\|_{\F^{\alpha,L}_{p,q,w}}
	\end{aligned}
	\]
	where in the last inequality we use Proposition \ref{prop2-thm1}.
	
	This completes our proof.
\end{proof}
\bigskip

For the converse direction, we have the following theorem:
\begin{thm}\label{thm2- atom TL spaces}	
	Let $\alpha\in \mathbb{R}$, $0<p<\vc$, $0<q\le \vc$ and $w\in A_\vc$. If
	$$
	f=\sum_{\nu\in\mathbb{Z}}\sum_{Q\in \mathcal{D}_\nu}s_Qa_Q \ \ \text{in $\mathcal{S}'_\vc$}
	$$
	where $\{a_Q\}_{Q\in \mathcal{D}_\nu, \nu\in \mathbb{Z}}$ is a sequence of $(L,M,p,w)$ atoms and $\{s_Q\}_{Q\in \mathcal{D}_\nu, \nu\in\mathbb{Z}}$ is a sequence of coefficients satisfying
	$$
	\Big\|\Big[\sum_{\nu\in\mathbb{Z}}2^{\nu\alpha q}\Big(\sum_{Q\in \mathcal{D}_\nu}w(Q)^{-1/p}|s_Q|\chi_Q\Big)^q\Big]^{1/q}\Big\|_{p,w}<\vc,
	$$
	then $f\in \F^{\alpha,L}_{p,q,w}$ and
	$$
	\|f\|_{\F^{\alpha,L}_{p,q,w}} \lesi \Big\|\Big[\sum_{\nu\in\mathbb{Z}}2^{\nu\alpha q}\Big(\sum_{Q\in \mathcal{D}_\nu}w(Q)^{-1/p}|s_Q|\chi_Q\Big)^q\Big]^{1/q}\Big\|_{p,w}
	$$
	provided $M>\f{n}{2}+\f{1}{2}\max\{\alpha,\f{nq_w}{1\wedge p\wedge q}-\alpha\}$.
\end{thm}
\begin{proof}
	The proof of this theorem is similar to that of  Theorem \ref{thm2- atom Besov}. Hence, we just sketch the main ideas. 
	
	With the same notations as in the proof of Theorem \ref{thm2- atom Besov}, from \eqref{eq-proof of reverse Besov} we have
	$$
	\begin{aligned}
	2^{j\alpha}|\psi_j(\sL)f|
	&\lesi \sum_{\nu: \nu\geq j}2^{-(\nu-j)(2M-n\tilde q/r-\alpha)}\mathcal{M}_{w,r}\Big(\sum_{Q\in \mathcal{D}_\nu}2^{\nu\alpha}|s_Q|w(Q)^{-1/p}\chi_Q\Big)\\
	& \ \ \ \ +\sum_{\nu: \nu< j}2^{-(2M-\alpha)(j-\nu)}\mathcal{M}_{w,r}\Big(\sum_{Q\in \mathcal{D}_\nu}2^{\nu\alpha}|s_Q|w(Q)^{-1/p}\chi_Q\Big).
	\end{aligned}
	$$
	By using \eqref{YFSIn} we conclude that
	\[
		\begin{aligned}
	\|f\|_{\F^{\alpha,L}_{p,q,w}}&=\Big\|\Big[\sum_{j\in \mathbb{Z}}(2^{j\alpha}|\psi_j(\sL)f|)^{q}\Big]^{1/q}\Big\|_{p,w}\\
	&\lesi \Big\|\Big[\sum_{\nu\in\mathbb{Z}}2^{\nu\alpha q}\Big(\sum_{Q\in \mathcal{D}_\nu}w(Q)^{-1/p}|s_Q|\chi_Q\Big)^q\Big]^{1/q}\Big\|_{p,w}.
		\end{aligned}
	\]
	This completes our proof.
\end{proof}

\begin{rem}
	\label{rem1}
	By a careful examination the proof of Theorem \ref{thm1- atom Besov} and Theorem \ref{thm1- atom TL spaces}, it is easy to see  that each atom $a_Q=L^{2M}b_Q$, defined by \eqref{eq-bQ}, belongs to the spaces of test functions $\mathcal{S}_\vc$. As a direct consequence of the atomic decomposition results in these two theorems, the test functions space $\mathcal{S}_\vc$ is dense in both $\B^{\alpha,L}_{p,q,w}$ and $\F^{\alpha,L}_{p,q,w}$ when $0<p,q<\vc$.
	
	\medskip

	We can mimic Definition \ref{defLmol} to define new molecules associated to $L$ as follows: Let $0< p\leq \vc$, $N>0$, $M\in \mathbb{N}_+$ and $w\in A_\vc$. A function $m$ is said to be an $(L, M, N, p, w)$ molecule if there exists a dyadic cube $Q\in \mathcal{D}$ so that
			\begin{enumerate}[{\rm (i)}]
				\item $m=L^{M} b$;
				\item $\displaystyle |L^{k} b(x)|\leq \ell(Q)^{2(M-k)}w(Q)^{-1/p}\Big(1+\f{|x-x_Q|}{\ell(Q)}\Big)^{-N}$, $k=0,\ldots , 2M$;
			\end{enumerate}
			where $B_Q$ is a ball associated to $Q$ defined in Remark \ref{rem1-Christ}.

	Then we can adapt the arguments in the proofs of Theorem \ref{thm1- atom Besov} and Theorem \ref{thm1- atom TL spaces} to obtain the molecular decompositions for our new Besov and Triebel--Lizorkin spaces. However, we do not aim to present the results in this paper and  leave the details to the interested reader. We note that the molecular decomposition theorem for the unweighted case was obtained in \cite{G.etal2} under the additional assumptions (H) and (C)   by using a different approach. See  Remark \ref{rem 6.3} below.
	
\end{rem}
\section{Identifications of our new Besov and Triebel-Lizorkin spaces with known function spaces}

\subsection{Coincidence with $L^p_w(X)$ spaces}
We have the following results.
\begin{thm}\label{equiv Lpw} For $1<p<\vc$ and $w\in A_p$, we have
	\begin{equation}\label{eq-Lpw}
	\F^{0,L}_{p,2,w}(X)\equiv L^p_w(X).
	\end{equation}
\end{thm}
\begin{proof}
	Arguing similarly to the proof of \cite[Theorem 7.2]{AM} we prove that  for any $m\in \mathbb{N}^+$,
	\begin{equation}
	\label{square functions}
	\Big\|\Big[\int_0^\vc|\Psi_{m,t}(L)f|^2\f{dt}{t}\Big]^{1/2}\Big\|_{p,w}\sim c_{p,w}\|f\|_{p,w}.
	\end{equation}
	This, along with Corollary \ref{cor1}, implies \eqref{eq-Lpw}.
\end{proof}

\subsection{Coincidence with the weighted Hardy spaces $H^p_{L,w}$.}\label{sub-weightedHardy} Let $0<p\le 1$ and $w\in A_\vc$. The weighted Hardy space $H^p_{L,w}$ is defined as the completion of the set
\begin{equation*}
\left\{f\in L^2: \mathcal{S}_Lf\in L^p_w \right\}
\end{equation*}
under the norm $\|f\|_{H^p_{L,w}}=\|\mathcal{S}_Lf\|_{p,w}$ where
\[
\mathcal{S}_Lf(x)=\Big[\int_0^\vc\int_{d(x,y)<t}|t^2Le^{-tL^2}f(y)|^2\f{\dy dt}{tV(x,t)}\Big]^{1/2}.
\]
The Hardy spaces $H^1_L$ was initiated in \cite{ADM}. See also \cite{DY}. The theory of Hardy spaces associated to operators satisfying Davies--Gaffney estimates  $H^1_L$ was established in \cite{HLMMY}. The weighted version for $H^p_{L,w}$ was investigated in \cite{BCKYY1}. From Proposition \ref{prop4.1-heat kernel} we obtain:

\begin{thm}
	\label{equiv-Hardy}
	Let $0<p\le 1$ and $w\in A_\vc$. Then we have
	\[
	H^p_{L,w}\equiv \F^{0,L}_{p,2,w}.
	\]
\end{thm}

Moreover, we have an interesting characterization for the weighted Hardy spaces:
\begin{prop}\label{prop-equiv Hardy square functions}
	Let $0<p\le 1$ and $w\in A_\vc$ and let $\psi$ be a partition of unity and $\varphi\in \mathscr{S}_1(\mathbb{R})$. Then for any $f\in L^2\cap H^p_{L,w}$ we have
	\[
	\begin{aligned}
	\|f\|_{H^p_{L,w}}&\sim \Big\|\Big[\int_{0}^\vc|\psi(t\sL)f|^2\f{dt}{t} \Big]^{1/2}\Big\|_{p,w}\sim \Big\|\Big[\int_{0}^\vc|\varphi(t\sL)f|^2\f{dt}{t} \Big]^{1/2}\Big\|_{p,w}\\
	&\sim \|\mathcal{G}_{\lambda, 2}(\psi(t\sL)f)\|_{p,w}\sim \|\mathcal{S}_{2}(\psi(t\sL)f)\|_{p,w}\\
	&\sim \|\mathcal{G}_{\lambda, 2}(\varphi(t\sL)f)\|_{p,w}\sim \|\mathcal{S}_{2}(\varphi(t\sL)f)\|_{p,w}
	\end{aligned}
	\]
	where $\mathcal{S}_{2}$ and $\mathcal{G}_{\lambda, 2}$ are square functions defined in \eqref{g-function} and \eqref{lusin-function}.
\end{prop}
\begin{proof}
	The proposition follows immediately from Theorem \ref{thm1-continuouscharacter}, Theorem \ref{thm2}, Proposition \ref{prop4.1b} and Proposition \ref{prop4.1-heat kernel}.
\end{proof}
\begin{rem}
These equivalent norm estimates for weighted Hardy spaces $H^p_{L,w}$ are new. It is worth noticing that in the particular case when $\varphi(\xi)=\xi^{2M}e^{-\xi^2}$ with $M\ge 1$ the equivalent estimate in Proposition \ref{prop-equiv Hardy square functions} 
\[
\|f\|_{H^p_{L,w}}\sim \Big\|\Big[\int_{0}^\vc|\varphi(t\sL)f|^2\f{dt}{t} \Big]^{1/2}\Big\|_{p,w}
\]
reads
\[
\|f\|_{H^p_{L,w}}\sim \Big\|\Big[\int_{0}^\vc|(t^2L)^{M}e^{-t^2L}f|^2\f{dt}{t} \Big]^{1/2}\Big\|_{p,w}.
\]
Note that this estimate was proved in \cite{G} (see also \cite{DJL}) for $M=1$.
\end{rem}

\subsection{Coincidence with the weighted BMO spaces ${\rm BMO}_{L,w}(X)$} 
\begin{defn}
	Let $w\in A_\vc$. The function $f\in \mathcal S'\cap L^1_{{\rm loc}}(X)$ is said to be in ${\rm BMO}_{L,w}(X)$, the weighted space of functions of bounded mean oscillation associated to $L$, if 
	\begin{equation}\label{BMOdef}
	\|f\|_{{\rm BMO}_{L,w}(X)}:=\sup_{B: \  {\rm balls}}\f{1}{w(B)}\int_B|(I-e^{-r^2_BL})f(x)|\dx<\vc.
	\end{equation}
\end{defn}

The unweighted BMO space ${\rm BMO}_L(X)$ associated to operators $L$ was first introduced by \cite{DY1}. The weighted version was studied in \cite{BD3, HYY}. 

It was proved in Theorem 5.5 \cite{HYY} that when $L=-\Delta$ on $\mathbb{R}^n$, we have
\begin{equation}
\label{eq-BMO equiv}
{\rm BMO}_{-\Delta,w}(\mathbb{R}^n)={\rm BMO}_{w}(\mathbb{R}^n) 
\end{equation} 
for all $w\in A_1\cap RH_2$, where 
\[
{\rm BMO}_{w}(\mathbb{R}^n)=\left\{f\in L^1_{{\rm loc}}:  \|f\|_{{\rm BMO}_{w}}:=\sup_{B: \  {\rm balls}}\f{1}{w(B)}\int_B|f-f_B|dx<\vc \right\}.
\]

We now prove the coincidence between the weighted BMO space ${\rm BMO}_{L,w}(X)$ and the weighted Triebel--Lizorkin space $\F^{0,L}_{\vc,2,w}(X)$.
\begin{thm}\label{BMOthm} We have the following identities.
	
	\noindent{\rm (a)} ${\rm BMO}_{L}(X)\equiv \F^{0,L}_{\vc,2}(X)$ in the sense that if $f\in {\rm BMO}_{L}(X)$ then $f\in \F^{0,L}_{\vc,2}(X)$; conversely, if $f\in \F^{0,L}_{\vc,2}(X)$ then there exists $\rho\in \mathscr P$ so that $f-\rho\in {\rm BMO}_{L}(X)$.
	
	\noindent{\rm (b)} Let $w\in A_1\cap RH_2$. Then we have 
	${\rm BMO}_{w}(\mathbb{R}^n)\equiv \F^{0,-\Delta}_{\vc,2,w}(\mathbb{R}^n)$ in the similar sense to that of {\rm (a)}.
\end{thm}
\begin{proof}
	(a) It was proved in Theorem 4.2 \cite{JY} that 
	\[
	\sup_{x\in X, t>0}\Big(\f{1}{V(x,t)}\int_{B(x,t)}\int_0^{t}|s^2Le^{-s^2L} e^{-s^2L}f|^2 \f{ds}{s}\dy\Big)^{1/2}\sim  \|f\|_{{\rm BMO}_{L,w}(X)}.
	\]
	This, along with Theorem \ref{prop3-Fvc}, implies the assertion.
	
	\medskip
	
	(b) It was proved in \cite{HYY} that 
	\[
	\|f\|_{{\rm BMO}_{w}(\mathbb{R}^n)}\sim \sup_{x\in X, t>0}\Big(\f{|B(x,t)|}{w(B(x,t))^2}\int_{B(x,t)}\int_0^{t}|s^2\Delta e^{s^2\Delta}(I- e^{s^2\Delta})f|^2 \f{ds}{s}dy\Big)^{1/2}.
	\]
	Using this and Theorem \ref{prop3-Fvc}, we derive part (b).
\end{proof}
\subsection{Coincidence with the weighted Sobolev spaces $\dot{W}^{s,L}_{p,w}$}\label{sub-fractional} For $s\in \mathbb{R}$, we define $L^{s/2}: \SL\to \SL$ by setting: 
\begin{equation}
\label{eq-Ls/2}
L^{s/2}f=\f{1}{\Gamma(m-s/2)}\int_0^\vc t^{-s/2}(tL)^m e^{-tL}f\f{dt}{t}
\end{equation}
for any $m\in \mathbb{N}, m>s/2$.

Arguing similarly to Proposition \ref{prop-Calderon1}, we can prove that the right hand side in \eqref{eq-Ls/2} converges in $\SL$. Moreover, by integration by part we can see that
\[
\f{1}{\Gamma(m-s/2)}\int_0^\vc t^{-s/2}(tL)^m e^{-tL}f\f{dt}{t}=\f{1}{\Gamma(\ell-s/2)}\int_0^\vc t^{-s/2}(tL)^\ell e^{-tL}f\f{dt}{t}
\]
for any $m,\ell\in \mathbb{N}, m,\ell>s/2$.

Therefore, $L^{s/2}$  given by \eqref{eq-Ls/2} is well-defined as an operator from $\SL$ into $\SL$. Moreover, it is easy to check that 
\begin{equation}
\label{eq-prop Ls/2}
L^\alpha[L^\beta f]=L^{\alpha+\beta}f, \ \ \ \forall f\in \SL.
\end{equation} 

We now define the weighted Sobolev spaces $\dot{W}^{s,L}_{p,w}$ as follows: Let $s\in\mathbb{R}$, $1<p<\vc$ and $w\in A_\vc$. The weighted Sobolev space $\dot{W}^{s,L}_{p,w}$ is defined as the completion of the set
\[
\left\{f\in \mathcal{S}_\vc: \|L^{s/2}f\|_{p,w}<\vc \right\}
\]
under the norm $\|f\|_{\dot{W}^{s,L}_{p,w}}=\|L^{s/2}f\|_{p,w}$.

\begin{thm}
	\label{equiv Sobolev}
	Let $s\in\mathbb{R}$, $1<p<\vc$ and $w\in A_\vc$. Then we have
	\[
	\dot{W}^{s,L}_{p,w}\equiv \dot{F}^{s,L}_{p,2,w}.
	\]
\end{thm}
\begin{proof}
	From Remark \ref{rem1}, we need only to show that
	\[
	\dot{W}^{s,L}_{p,w}\cap \mathcal{S}_{\vc}\equiv \dot{F}^{s,L}_{p,2,w}\cap \mathcal{S}_{\vc}.
	\]
	Indeed, let $f\in \dot{F}^{s,L}_{p,2,w}\cap \mathcal{S}_{\vc}$ and $\psi$ be a partition of unity. Then by Theorem \ref{equiv Lpw} and the spectral theory, for $g\in L^{p'}_{w^{1-p'}}$,
	\[
	\begin{aligned}
	|\langle L^{s/2}f,g\rangle| &=\Big|c_\psi \int_X\int_0^\vc \psi^3(t\sL)L^{s/2}f(x)g(x)\f{dt}{t}\dx \Big|\\
	&= \Big|c_\psi \int_X\int_0^\vc \psi^2(t\sL)L^{s/2}f(x)\psi(t\sL)g(x)\f{dt}{t}\dx\Big|\\
	&\lesi c_\psi \Big\|\Big[\int_{0}^\vc|\psi(t\sL)^2L^{s/2}f|^2\f{dt}{t} \Big]^{1/2}\Big\|_{p,w}\Big\|\Big[\int_{0}^\vc|{\psi}(t\sL)g|^2\f{dt}{t} \Big]^{1/2}\Big\|_{p',w^{1-p'}},
	\end{aligned}
	\]
	where $c_\psi =\Big[\int_0^\vc \psi^3(s)\f{ds}{s}\Big]^{-1}$, and
	we use H\"older's inequality in the last inequality.

On the other hand, since $w\in A_p$, $w^{1-p'}\in A_{p'}$. By Theorem \ref{equiv Lpw} and Theorem \ref{thm1-continuouscharacter} we have
\[
\Big\|\Big[\int_{0}^\vc|{\psi}(t\sL)g|^2\f{dt}{t} \Big]^{1/2}\Big\|_{p',w^{1-p'}}\sim \|g\|_{p',w^{1-p'}}.
\]
Therefore,
\begin{equation}\label{eq0-proof Ls}
|\langle L^{s/2}f,g\rangle| \lesi \|g\|_{p',w^{1-p'}}\Big\|\Big[\int_{0}^\vc|\psi(t\sL)L^{s/2}f|^2\f{dt}{t} \Big]^{1/2}\Big\|_{p,w}.
\end{equation}

Using \eqref{eq-Ls/2} we have
\[
\begin{aligned}
\psi^2(t\sL)L^{s/2}f(x)=&\f{1}{\Gamma(m-s/2)}\int_0^\vc u^{-s/2}\psi(t\sL)(uL)^m e^{-uL}(\psi(t\sL)f)\f{du}{u}\\
=&\int_0^{t^2}\ldots+\int^\vc_{t^2}\ldots=:E_1(x,t)+E_2(x,t).
\end{aligned}
\]
Fix $\lambda>\max\{n/q,nq_w/p \}$ and $M>(\lambda+s)/2$. By Lemma \ref{lem1} we have, for $N>n$,
\[
\begin{aligned}
|E_1(x,t)|&\lesi \int_0^{t^2}\int_X u^{-s/2}\Big(\f{t^2}{u}\Big)^M\f{1}{V(x,u)}\Big(1+\f{d(x,y)}{u}\Big)^{-N-\lambda}| \psi(t\sL)f(y)|\dy\f{dt}{t}\\
&\lesi \int_0^{t^2}\int_X u^{-s/2}\Big(\f{t^2}{u}\Big)^{M-\lambda/2}\f{1}{V(x,u)}\Big(1+\f{d(x,y)}{u}\Big)^{-N}\psi^*_\lambda(t\sL)f(x)\dy\f{dt}{t}\\
&\lesi t^{-s}\psi^*_\lambda(t\sL)f(x).
\end{aligned}
\]
Similarly, we have
\[
|E_2(x,t)|\lesi t^{-s}\psi^*_\lambda(t\sL)f(x).
\]
Hence,
\begin{equation}\label{eq1- proof Ls/2}
	|\psi^2(t\sL)L^{s/2}f(x)|\lesi t^{-s}\psi^*_\lambda(t\sL)f(x).
\end{equation}
Inserting this into \eqref{eq0-proof Ls}, then using Theorem \ref{thm1-continuouscharacter} we get that
\[
\begin{aligned}
|\langle L^{s/2}f,g\rangle| &\lesi \|g\|_{p',w^{1-p'}}\Big\|\Big[\int_{0}^\vc(t^{-s}\psi^*_\lambda(t\sL)f)^2\f{dt}{t} \Big]^{1/2}\Big\|_{p,w}\\
&\lesi \|g\|_{p',w^{1-p'}}\|f\|_{\F^{0,L}_{p,2,w}}.
\end{aligned}
\]
This implies $\|L^{s/2}f\|_{p,w}\lesi \|f\|_{\F^{0,L}_{p,2,w}}$. Hence, $\dot{F}^{s,L}_{p,2,w}\cap \mathcal{S}_{\vc} \hookrightarrow \dot{W}^{s,L}_{p,w}\cap \mathcal{S}_{\vc}$.\\

Conversely, let $\psi\in \SR$ so that $\psi^3$ is a partition of unity. Then for $f\in \dot{W}^{s,L}_{p,w}\cap \mathcal{S}_{\vc}$, by \eqref{eq-prop Ls/2} we have
\[
\begin{aligned}
\|f\|_{\dot{F}^{s,L}_{p,2,w}}&:=\Big\|\Big[\sum_{j\in \mathbb{Z}}(2^{-js}|\psi^3(2^{-j}\sL)f|)^2\Big]^{1/2}\Big\|_{p,w}\\
&\lesi\Big\|\Big[\sum_{j\in \mathbb{Z}}(2^{-js}|L^{-s/2}\psi^3(2^{-j}\sL)(L^{s/2}f)|)^2\Big]^{1/2}\Big\|_{p,w}.
\end{aligned}
\]
Arguing similarly to the proof of \eqref{eq1- proof Ls/2} we have, for each $j\in \mathbb{Z}$,
$$
|L^{-s/2}\psi^3(2^{-j}\sL)(L^{s/2}f)|\lesi 2^{js}\psi^*_{j,\lambda}(\sL)(L^{s/2}f).
$$
Therefore, 
\[
\begin{aligned}
\|f\|_{\dot{F}^{s,L}_{p,2,w}}&\lesi\Big\|\Big[\sum_{j\in \mathbb{Z}}|\psi^*_{j,\lambda}(\sL)(L^{s/2}f)|)^2\Big]^{1/2}\Big\|_{p,w}.
\end{aligned}
\]
This, together with Proposition \ref{prop1-maximal function} and Proposition \ref{prop2-thm1}, yields
\[
\begin{aligned}
\|f\|_{\dot{F}^{s,L}_{p,2,w}}&\lesi \|L^{s/2}f\|_{\F^{0,L}_{p,2,w}}.
\end{aligned}
\]
On the other hand, from Theorem \ref{equiv Lpw} we have $\|L^{s/2}f\|_{\F^{0,L}_{p,2,w}}\sim \|L^{s/2}f\|_{p,w}$. Hence,
\[
\begin{aligned}
\|f\|_{\dot{F}^{s,L}_{p,2,w}}&\lesi \|L^{s/2}f\|_{p,w}.
\end{aligned}
\]
This completes our proof.
\end{proof}

\subsection{Coincidence with weighted Hardy--Sobolev spaces $H\dot{S}^{s,L}_{p,w}$.} For $s\in\mathbb{R}$, $0<p\le 1$ and $w\in A_\vc$ the weighted Hardy-Sobolev space 
$H\dot{S}^{s,L}_{p,w}$ is defined as the completion of the set
\[
\left\{f\in \mathcal{S}_\vc: \|L^{s/2}f\|_{H^p_{L,w}}<\vc \right\}
\]
under the norm $\|f\|_{H\dot{S}^{s,L}_{p,w}}=\|L^{s/2}f\|_{H^p_{L,w}}$, where $H^p_{L,w}$ is the weighted Hardy space defined as in Subsection \ref{sub-weightedHardy}.

\begin{thm}
	\label{equiv Hardy-Sobolev}
	Let $s\in\mathbb{R}$, $0<p\le 1$ and $w\in A_\vc$. Then we have
	\[
	H\dot{S}^{s,L}_{p,w}\equiv \dot{F}^{s,L}_{p,2,w}.
	\]
\end{thm}
\begin{proof}
	Since $\mathcal{S}_\vc$ is dense in both spaces, we need only to verify that
	\[
	H\dot{S}^{s,L}_{p,w}\cap \mathcal{S}_\vc \equiv \dot{F}^{s,L}_{p,2,w}\cap \mathcal{S}_\vc.
	\] 
	We first verify that $\dot{F}^{s,L}_{p,2,w}\cap \mathcal{S}_\vc\subset H\dot{S}^{s,L}_{p,w}\cap \mathcal{S}_\vc$. Indeed, if $f\in \dot{F}^{s,L}_{p,2,w}\cap \mathcal{S}_\vc$, then by Theorem \ref{equiv-Hardy} and Theorem \ref{thm1-continuouscharacter} we have
	\[
	\begin{aligned}
	\|f\|_{H\dot{S}^{s,L}_{p,w}}=\|L^{s/2}f\|_{H^p_{L,w}}\sim\Big\|\Big[\int_0^\vc |\psi(t\sL)(L^{s/2}f)|\Big]^{1/2}\Big\|_{p,w}
	\end{aligned}
	\]
	where $\psi\in \SR$ so that $\psi^3$ is a partition of unity. 
	
	At this stage, we argue similarly to the proof of Theorem \ref{equiv Hardy-Sobolev} to obtain that
	$$
	\|f\|_{H\dot{S}^{s,L}_{p,w}}\lesi \|f\|_{\dot{F}^{s,L}_{p,2,w}}
	$$
	which implies
	$\dot{F}^{s,L}_{p,2,w}\cap \mathcal{S}_\vc\subset H\dot{S}^{s,L}_{p,w}\cap \mathcal{S}_\vc$.
	
	The converse direction is similar and we omit the details.
	
\end{proof}

\section{Comparison with classical Besov and Triebel-Lizorkin spaces}

\begin{defn}\label{def:7.1}
	Let $ 0<p \leq \infty, w \in A_\infty $ and $\epsilon >0$. A function $a$ is said to be a $(p,w,\epsilon)$ atom if there exists a dyadic cube $Q \in \mathcal{D}_{\nu}$ such that \\
	\begin{enumerate}[\rm (i)]
		\item $\operatorname{supp}\,a \subset 3B_Q$;
		\item $|a(x)| \leq w(Q)^{-1/p}$;
		\item $|a(x)-a(y)| \leq w(Q)^{1/p}\left(\dfrac{d(x,y)}{2^{-\nu}}\right)^\epsilon$;
		\item $\displaystyle \int a(x)d\mu (x)=0$.
	\end{enumerate}
\end{defn}

Let $0<p< \infty, w \in A_\infty$ and $\epsilon >0$. We say that a function $f$ has a $(p,w,\epsilon)$ atomic representation of Besov type if \\
$$f= \sum\limits_{\nu  \in \mathbb{Z}} {\sum\limits_{Q \in {\mathcal{D}_\nu }} {{s_Q}{a_Q}} } \text{ in }L^2 $$ 
where $\{a_Q\}$ is a sequence of $(p,w,\epsilon)$ atoms and $s:=\{s_Q\}$ is a sequence of numbers satisfying 
$${\Big[ {\sum\limits_{\nu  \in Z} {{2^{\nu \alpha q}}{{\Big( {\sum\limits_{Q \in {\mathcal{D}_\nu }} {|{s_Q}{|^p}} } \Big)}^{q/p}}} } \Big]^{1/q}} < \infty. $$ 
Then the weighted Besov space is defined as follows\\

\begin{defn}
	Let $ \alpha \in (-1,1),0<p,q<\infty $ and  $ w \in A_\infty $. The weighted Besov space $B^{\alpha}_{p,q,w} $ 	is defined as the completion of the set of all $L^2$-functions having a $ (p,w,\epsilon)$ atomic representation of Besov type under the norm\\
	$$ {\| f\|_{\dot B_{p,q,w}^\alpha }} = \inf \Big\{ {{{\Big[ {\sum\limits_{\nu  \in \mathbb{Z}} {{2^{\nu \alpha q}}{{\Big( {\sum\limits_{Q \in {\mathcal{D}_\nu }} {{{| {{s_Q}}|}^p}} } \Big)}^{q/p}}} } \Big]}^{1/q}}:f = \sum\limits_{\nu  \in \mathbb{Z}} {\sum\limits_{Q \in {\mathcal{D}_\nu }} {{s_Q}{a_Q}} } } \Big\}.$$
\end{defn}
Similarly, the weighted Triebel-Lizorkin spaces $ F^{\alpha}_{p,q,w} $ is defined by replacing the quantity

$$ {\Big[ {\sum\limits_{\nu  \in \mathbb{Z}} {{2^{\nu \alpha q}}{{\Big( {\sum\limits_{Q \in {\mathcal{D}_\nu }} {{{| {{s_Q}}|}^p}} } \Big)}^{q/p}}} } \Big]^{1/q}},$$
by 
$${\Big\| {{{\Big[ {\sum\limits_{\nu  \in \mathbb{Z}} {{2^{\nu \alpha q}}{{\Big( {\sum\limits_{Q \in {\mathcal{D}_\nu }} {w{{(Q)}^{ - 1/p}}| {{s_Q}}|{\chi _Q}} } \Big)}^q}} } \Big]}^{1/q}}} \Big\|_{p,w}}.$$

\begin{rem}\label{rem 6.3}
    Up to now we have stated and proved our results under rather mild assumptions on $L$; namely, we have assumed that $L$ is a nonnegative self-adjoint operator for which its heat kernel $p_t(x,y)$ satisfies the Gaussian upper bound (GE). In some applications below we will require $L$ also to satisfy one or both of the following additional conditions:
	\begin{enumerate}
		\item[(H)] There exists $\delta_0 \in (0,1]$  so that \\
		\begin{equation}\label{eq:H}
		| p_t(x,y) - p_t(\bar x,y) | \lesssim \Big(\f{d(x,\bar x)}{\sqrt{t}}\Big)^{\delta _0}\frac{1}{{V\left( {x,\sqrt t } \right)}}\exp\Big(-\f{d(x,y)^2}{ct}\Big)
		\end{equation}
		whenever $d(x,\bar x)< \sqrt{t}.$ (H\"older Continuity Property.)
		\item[(C)] $\displaystyle \int_X {{p_t}\left( {x,y} \right)d \mu \left( x \right)}  = 1$ for all $y \in X$ and $t>0$. (Conservation Property.)
	\end{enumerate}
\end{rem}
We have the following estimate:
\begin{lem}\label{lem6.4}
	 Let $\Phi$ be a function as in Lemma 2.1. Also assume that $L$ satisfies (H).
	Then for any $M \in \mathbb{N}$ we have
	$$\left| {{K_{{{( {{t^2}L})}^M}\Phi ( {t\sqrt L })}} ( {x,y}) - {K_{{{( {{t^2}L})}^M}\Phi( {t\sqrt L })}}( {\bar x,y})} \right| \lesssim \Big(\f{d(x,\bar x)}{t}\Big)^{\delta _0}\frac{1}{{V( {x,t})}}$$
	whenever $d(x,\bar x)<t$.
\end{lem}
\begin{proof}
	The proof of this lemma is similar to that of Theorem 1 in \cite{S}. We leave the details to the interested reader.
\end{proof}
The following theorem is the main result of this section.

\begin{thm}
	Assume that $L$ also satisfies (H) and (C). Then we have
	\begin{equation}\label{eq:73}
	\dot B_{p,q,w}^{\alpha ,L} \equiv \dot B_{p,q,w}^\alpha 
	\end{equation}
	for all $\alpha \in (-\delta_0,\delta_0),0<p,q<\infty$ and $w \in A_\infty$ such that $\alpha+n+\delta_0>\frac{nqw}{p}$; and\\
	\begin{equation}\label{eq:74}
	\dot F_{p,q,w}^{\alpha ,L} \equiv \dot F_{p,q,w}^\alpha 
	\end{equation}
	for all $\alpha \in (-\delta_0,\delta_0),0<p,q<\infty$ and $w \in A_\infty$ such that $\alpha+n+\delta_0>\frac{nqw}{p}$.
\end{thm}

\begin{proof}
	We will only prove $\dot F_{p,q,w}^{\alpha ,L} \equiv \dot F_{p,q,w}^\alpha $  , since the proof of $	\dot B_{p,q,w}^{\alpha ,L} \equiv \dot B_{p,q,w}^\alpha$  
	can be done similarly.\\
	We split the proof into 2 steps:\\
	\textbf{Step 1: Proof of}  $\dot F_{p,q,w}^{\alpha ,L} \cap L^2 \hookrightarrow \dot F_{p,q,w}^\alpha \cap L^2 $. To do this, we will employ the same notations as in 	the proof of Theorem 5.6. If $f \in \dot F_{p,q,w}^{\alpha ,L} \cap L^2$, then repeating the proof of Theorem 5.6 we can find a sequence of $ (L, M, p, w) $ atoms  $ \{a_Q\}_{Q\in \mathcal{D}_\nu,\nu \in \mathbb{Z}}$ and a sequence of coefficients $ \{s_Q\}_{Q\in \mathcal{D}_\nu,\nu \in \mathbb{Z}}$
	so that\\
	$$f = \sum\limits_{\nu  \in \mathbb{Z}} {\sum\limits_{Q \in {\mathcal{D}_\nu }} {{s_Q}{a_Q}} } \text{ in } L^2$$
	where\\
	$${a_Q} = \frac{1}{{{s_Q}}}\int_{{2^{ - \nu  - 1}}}^{2^{ - \nu }} {{{( {{t^2}L})}^M}\Phi ( {t\sqrt L })\left[ {{\psi _M}( {t\sqrt L })f.{\chi _Q}} \right]\frac{{dt}}{t},}$$
	and moreover,\\
	$$\Big\| \Big[ \sum\limits_{\nu  \in \mathbb{Z}} {{2^{\nu \alpha q}}{{\Big( {\sum\limits_{Q \in {\mathcal{D}_\nu }} {w{{(Q)}^{ - 1/p}}| {{s_Q}} |{\chi _Q}} } \Big)}^q}} \Big]^{1/q} \Big\|_{p,w} \lesssim \| f \|_{\dot F_{p,q,w}^{\alpha ,L}}.$$
	We now claim that each $ (L, M, p, w) $ atom is also $ (p, w, \delta_0) $ atom. Indeed, it is clear that each
	$ (L, M, p, w) $ atom satisfies (i) and (ii) in Definition~\ref{def:7.1}. The argument as in Lemma 9.1 in [26] implies that an $(L, M, p, w)$ atom satisfies (iv) in Definition~\ref{def:7.1}. The condition (iii) in Definition~\ref{def:7.1} can be verified by making use of Lemma~\ref{lem6.4} and hence we omit the details. This completes the
	first step.\\
	\textbf{Step 2: Proof of}  $\dot F_{p,q,w}^{\alpha} \cap L^2 \hookrightarrow \dot F_{p,q,w}^{\alpha ,L} \cap L^2 $. To do this, we need the following estimates:\\
	\begin{lem}\label{lem1-comparison}
		Let $ a_Q $ be a $ (p, w, \delta_0) $ atom associated to some dyadic cube $Q \in \mathcal{D}_\nu, \nu \in \mathbb{Z} $. Then,
		for any $  N > 0 $, we have: \\
		\begin{enumerate}[(i)]
			\item $| {{t^2}L{e^{ - {t^2}L}}{a_Q}(x)} | \lesssim w( Q)^{ - 1/p}{\left( {\frac{t}{2^{ - \nu }}} \right)^{\delta _0}}{\left( {1 + \frac{d( x,x_Q)}{2^{ - \nu }}} \right)^{ - N}}$, for all $ t \leq 2^{-\nu}$;
			\item $| {{t^2}L{e^{ - {t^2}L}}{a_Q}(x)} | \lesssim {\Big( {\frac{2^{ - \nu }}{t}} \Big)^{\delta _0}}\frac{{V\left( Q \right)}}{{V\left( {{x_Q},t} \right)}}{\left( {1 + \frac{d( x,x_Q)}{2^{ - \nu }}} \right)^{ - N}}$, for all $ 2^{-\nu} \leq t$.
		\end{enumerate}
	\end{lem}
	
	\begin{proof}
		We just sketch the main ideas. Denote by $ q_t(x; y) $ the kernel of $tLe^{-tL}$.
		
		(i) If $x \in 6B_Q$, then from (A4), we have
			\begin{align*}
			\left| {{t^2}L{e^{ - {t^2}L}}{a_Q}\left( x \right)} \right| &\lesssim w( Q)^{ - 1/p}\int_X {\frac{1}{V(y,t)}{{\left( {1 + \frac{d(x,y)}{t}} \right)}^{ - K - {\delta _0}}}{{\left( {\frac{d(x,y)}{2^{ - \nu }}} \right)}^{\delta _0}}d\mu (y)} \\
			&\lesssim w( Q)^{ - 1/p}{\left( {\frac{t}{2^{ - \nu }}} \right)^{\delta _0}} \\
			&\sim w( Q)^{ - 1/p}{\left( {\frac{t}{2^{ - \nu }}} \right)^{\delta _0}} {\left( {1 + \frac{d( x,x_Q)}{2^{ - \nu }}} \right)^{ - N}}  (\text{since }y \in 6B_Q)				
			\end{align*}
			If $ x \notin 6B_Q $, then we have $ d(x, y) \sim  d(x, x_Q) \geq 2^{-\nu} $. This, together with Lemma 2.2, implies that
			\begin{align*}
			\left| {{t^2}L{e^{ - {t^2}L}}{a_Q}\left( x \right)} \right| &\lesssim w( Q)^{ - 1/p}\int_{3B_Q} {\frac{1}{V(y,t)}{{\left( {1 + \frac{d(x,y)}{t}} \right)}^{ - K -N- {\delta _0}}} \left| a_Q(y) \right| d\mu (y)} \\
			&\lesssim w( Q)^{ - 1/p}{\left( {\frac{t}{2^{ - \nu }}} \right)^{\delta _0}}{{\left( {1 + \frac{d(x,y)}{2^{-\nu}}} \right)}^{ -N}}\int_{3B_Q} {\frac{1}{V(y,t)}{{\left( {1 + \frac{d(x,y)}{t}} \right)}^{ - K}}  d\mu (y)} 
			\\
			&\lesssim w( Q)^{ - 1/p}{\left( {\frac{t}{2^{ - \nu }}} \right)^{\delta _0}} {\left( {1 + \frac{d( x,y)}{2^{ - \nu }}} \right)^{ - N}}  				
			\end{align*}
			where in the last inequality we use Lemma \ref{lem-elementary}. \\
			This completes the proof of (i).
			
			\medskip
			
			(ii) By (iv) in Definition \ref{def:7.1}, we can write
			$$\left|t^2Le^{-t^2L}a_Qa_Q(x)\right|=\left|\int_{3B_Q}{[q_{t^2}(x,y)-q_{t^2}(x,x_Q)]a_Q(y)d\mu (y)}\right|.$$
			Since $q_t(x,y)=2\int_{X}{q_{t/2}(x,z)p_{t/2}(z,y)s d\mu (z)}$, we see that $q_t(x,y)$ satisfies (\ref{eq:H}). Hence,
			\begin{align*}
			\left|t^2Le^{-t^2L}a_Q(x)\right|&\lesssim \int_{3B_Q}{\left(\dfrac{d(y,x_Q)}{t}\right)^{\delta_0}\dfrac{1}{V(x_Q,t)}\left(1+\dfrac{d(x,x_Q)}{t}\right)^{-N}|a_Q(y)|d\mu (y)}\\&\lesssim w(Q)^{-1/p}\left(\dfrac{2^{-\nu}}{t}\right)^{\delta_0}\dfrac{V(Q)}{V(x_Q,t)}\left(1+\dfrac{d(x,x_Q)}{t}\right)^{-N}.
			\end{align*}
			This completes the proof of (ii).
	\end{proof}
	
	We now turn to the proof of  Step 2. Using Lemma \ref{lem1-comparison}, and arguing similarly to \eqref{eq-proof of reverse Besov} we have
	\begin{multline*}
	|\psi_j(\sqrt{L})f|\lesssim \sum_{\nu : \nu\geq j}{2^{-(\nu -j)(\delta_0 +n-n\tilde{q}/r)}\mathcal{M}_{w,r}\Big(\sum_{Q\in \mathcal{D}_\nu}{|s_Q|w(Q)^{-1/p}\chi_Q}\Big)}\\
	+\sum_{\nu : \nu <j}{2^{-\delta_0(j-\nu)}\mathcal{M}_{w,r}\Big(\sum_{Q\in \mathcal{D}_\nu}{|s_Q|w(Q)^{-1/p}\chi _Q}\Big)}.
	\end{multline*}
	At this stage, the argument in the proof of Theorem \ref{thm2- atom TL spaces} (see also Theorem \ref{thm2- atom Besov}) shows that $ f \in \dot F_{p,q,w}^{\alpha ,L}$ provided that $\alpha+n+\delta_0>\frac{nq_w}{p\wedge q}$.
\end{proof}

\begin{rem}
	Some comments for the condition $\alpha+n+\delta_0>\frac{nq_w}{p\wedge q}$ are in order:
	\begin{enumerate}[(i)]
		\item If $1\leq p,q < \infty$, the identities (\ref{eq:73}) and (\ref{eq:74}) holds true for all $\alpha \in (-\delta_0,\delta_0)$ and $w \in A_\infty$ with $q_\infty < \min \{p,q\}\times \frac{n+\delta_0+\alpha}{n}$.\\
		\item If $w\equiv 1$,  the identities (\ref{eq:73}) and (\ref{eq:74}) holds true for all $\frac{n}{n+\delta_0}<p,q<\infty$ and $\frac{n}{p\wedge q}-n-\delta_0<\alpha<\delta_0$
	\end{enumerate}
\end{rem}

\section{Applications}

The theory of Besov and Triebel--Lizorkin spaces have a wide range of applications. See for example \cite{BBD, BDH, DP} and the references therein. In this section, we just give two applications to the fractional power and the spectral multipliers. Further application would be an upcoming project and will be investigated in the future.  
\subsection{Fractional powers}
\begin{thm}
	Let $s\in \mathbb{R}$ and let $L^{s/2}$ be defined as in \eqref{eq-Ls/2}. Then for $\alpha \in \mathbb{R}$ and $ w \in A_\infty$, the fractional integral $L^{s/2}$ maps continuously from $ \dot B_{p,q,w}^{\alpha ,L}$ into $ \dot B_{p,q,w}^{\alpha +s,L}$ for $0<p,q \leq \infty$ and from $\dot F_{p,q,w}^{\alpha ,L}$ into $ \dot F_{p,q,w}^{\alpha +s,L}$ for $0<p<\vc$ and $0<q \leq \infty$.
\end{thm}
\begin{proof}
	
	Let $\psi$ be a partition of unity and let $\varphi\in \mathscr S(\mathbb{R})$ be supported in $[1/4,4]$ such that $\varphi=1$ on $[1/2,2]$. For  $f \in \dot F_{p,q,w}^{\alpha ,L}(X)$, 
	using \eqref{eq-Ls/2} we have
	\[
	\begin{aligned}
	\psi(t\sL)L^{s/2}f(x)=&\f{1}{\Gamma(m-s/2)}\int_0^\vc u^{-s/2}\varphi(t\sL)(uL)^m e^{-uL}(\psi(t\sL)f)\f{du}{u}\\
	=&\int_0^{t^2}\ldots+\int^\vc_{t^2}\ldots=:I_1(x,t)+I_2(x,t).
	\end{aligned}
	\]
	Fix $\lambda>\max\{n/q,nq_w/p \}$ and $M>(\lambda+s)/2$. By Lemma \ref{lem1} we have, for $N>n$,
	\[
	\begin{aligned}
	|I_1(x,t)|&\lesi \int_0^{t^2}\int_X u^{-s/2}\Big(\f{t^2}{u}\Big)^M\f{1}{V(x,u)}\Big(1+\f{d(x,y)}{u}\Big)^{-N-\lambda}| \psi(t\sL)f(y)|\dy\f{dt}{t}\\
	&\lesi \int_0^{t^2}\int_X u^{-s/2}\Big(\f{t^2}{u}\Big)^{M-\lambda/2}\f{1}{V(x,u)}\Big(1+\f{d(x,y)}{u}\Big)^{-N}\psi^*_\lambda(t\sL)f(x)\dy\f{dt}{t}\\
	&\lesi t^{-s}\psi^*_\lambda(t\sL)f(x).
	\end{aligned}
	\]
	Similarly,
	\[
	|I_2(x,t)|\lesi t^{-s}\psi^*_\lambda(t\sL)f(x).
	\]
	Hence,
	\begin{equation*}
	|\psi(t\sL)L^{s/2}f(x)|\lesi t^{-s}\psi^*_\lambda(t\sL)f(x).
	\end{equation*}
	This, along with Theorem \ref{thm1-continuouscharacter}, implies 
	\[
	\begin{aligned}
	\|L^{s/2}f\|_{\B^{\alpha,L}_{p,q,w}(X)}&\sim \Big(\int_0^\vc \Big[t^{-\alpha}\|\psi(t\sL)L^{s/2}f\|_{p,w}\Big]^q\f{dt}{t}\Big)^{1/q}\\
	&\lesi  \Big(\int_0^\vc \Big[t^{-\alpha-s}\|\psi^*_\lambda(t\sL)f\|_{p,w}\Big]^q\f{dt}{t}\Big)^{1/q}\\
	& \sim \|f\|_{\B^{\alpha+s,L}_{p,q,w}(X)}.
	\end{aligned}
	\]
	Arguing similarly by using the item (b) in Theorem \ref{thm1-continuouscharacter} we obtain
	\[
	\|L^{s/2}f\|_{\F^{\alpha,L}_{p,q,w}(X)}\lesi \|f\|_{\F^{\alpha+s,L}_{p,q,w}(X)}.
	\]
	This completes our proof.
\end{proof}

\subsection{Spectral multiplier of Laplace transform type.} Let $ m:[0, \vc) \rightarrow \mathbb{C}$ be a bounded function. We now define\\
\begin{equation}\label{eq:75}
\tilde m(L)=\int_{0}^{\infty}{tLe^{-t^2L}m(t^2)dt}
\end{equation}
to be the spectral multiplier of Laplace transform type of $L$. We have the following result:

\begin{thm}\label{thm-sm}
	Let $\alpha \in \mathbb{R}$ and $ w \in A_\infty$. Then the spectral multiplier of Laplace transform type $\tilde m(L)$ defined by (\ref{eq:75}) is bounded on $ \dot B_{p,q,w}^{\alpha ,L}(X) $ for $0<p,q \leq \infty$ and is bounded on $\dot F_{p,q,w}^{\alpha ,L}(X)$ for $0<p<\vc$ and $0<q \leq \infty$.
\end{thm}

\begin{proof}
	We will provide the proof for the Triebel-Lizorkin spaces. The boundedness on the Besov spaces can be proved similarly.
	
	Let $\psi$ be a partition of unity. For  $f \in \dot F_{p,q,w}^{\alpha ,L}(X)$ we have\\
	\begin{align*}
	\psi(s\sqrt{L})\tilde m(L)f(x)&=c_\psi \int_{s/4}^{4s}{\psi(u\sqrt{L}) \tilde m(L)\psi(s\sqrt{L})f\frac{du}{u}}\\
	&=c_\psi \int_{s/4}^{4s}{\int_{0}^{\infty}{m(t^2)(t^2L)e^{-t^2L}\psi(u\sqrt{L})\psi(s\sqrt{L})f(x)\frac{dt}{t}}\frac{du}{u}}\\
	&=c_\psi \int_{s/4}^{4s}{\int_{0}^{u}{...}}+ c_\psi \int_{s/4}^{4s}{\int_{u}^{+\infty}{...}}\\
	&=:E(x)+F(x)
	\end{align*}
	Fix $\lambda > \max \{n/q,nq_w/p\}$ and $N>n$. Lemma \ref{lem1} and Lemma \ref{lem-elementary}, we have
	\begin{align*}
	|E(x)|&\lesssim \int_{s/4}^{4s}{\int_{0}^{u}{\dfrac{t^2}{u^2}\dfrac{1}{V(y,u)}\left(1+\dfrac{d(x,y)}{u}\right)^{-N-\lambda}|\psi(s\sqrt{L})f(y)|d\mu(y)\dfrac{dt}{t}}\dfrac{du}{u}}\\
	&\sim \int_{s/4}^{4s}{\int_{0}^{u}\int_{X}{{\dfrac{t^2}{u^2}\dfrac{1}{V(y,s)}\left(1+\dfrac{d(x,y)}{s}\right)^{-N-\lambda}|\psi(s\sqrt{L})f(y)|d\mu(y)}\dfrac{dt}{t}}\dfrac{du}{u}}\\
	&\sim \int_{s/4}^{4s}{\int_{0}^{u}{\dfrac{t^2}{u^2}\psi_{\lambda}^*(s\sqrt{L})f(x)\dfrac{dt}{t}}\dfrac{du}{u}} \sim \psi_{\lambda}^*(s\sqrt{L})f(x)
	\end{align*}
	Similarly, for $ M > \lambda/2 $,
	\begin{align*}
	|F(x)| &\lesssim \int_{s/4}^{4s}{\int_{u}^{\infty}{\left(\dfrac{u}{t}\right)^{2M}\dfrac{1}{V(y,t)}\left(1+\dfrac{d(x,y)}{t}\right)^{-N-\lambda}|\psi(s\sqrt{L})f(y)|d\mu(y)\dfrac{dt}{t}}\dfrac{du}{u}}\\
	&\sim \int_{s/4}^{4s}{\int_{u}^{\infty}\int_{X}{{\left(\dfrac{u}{t}\right)^{2M-\lambda}\dfrac{1}{V(y,t)}\left(1+\dfrac{d(x,y)}{t}\right)^{-N}\left(1+\dfrac{d(x,y)}{u}\right)^{-\lambda}|\psi(s\sqrt{L})f(y)|d\mu(y)}\dfrac{dt}{t}}\dfrac{du}{u}}\\
	&\sim \int_{s/4}^{4s}{\int_{u}^{\infty}{\left(\dfrac{u}{t}\right)^{2M-\lambda}\psi_{\lambda}^*(s\sqrt{L})f(x)\dfrac{dt}{t}}\dfrac{du}{u}}\\
	& \sim \psi_{\lambda}^*(s\sqrt{L})f(x).
	\end{align*}
	As a consequence,\\
	$$|\psi(s\sqrt{L})\tilde m(L)f(x)| \lesssim \psi_{\lambda}^*(s\sqrt{L})f(x).$$
	Therefore, the conclusion of the theorem follows immediately from Theorem \ref{thm1-continuouscharacter}. 
\end{proof}
\bigskip
\begin{rem}
	Theorem \ref{thm-sm} only requires the Gaussian upper bound condition for the operator $L$. This is a very mild condition and allows us to apply the results to a large number of applications such as the sub-Laplacians on Lie groups of polynomial growth, the Laplacians on the Heisenberg groups or the Laplace-Beltrami operators on certain Riemannian manifolds. For further details concerning examples satisfying this condition we refer to \cite[Section 7]{DOS} and the references therein. It is natural to ask the question on the sharp estimate for the general spectral multipliers of $L$. This problem is more complicated and we leave it as an upcoming project. 
	
\end{rem}

{\bf Acknowledgement.}  Xuan Thinh Duong was supported by Australian Research Council through the ARC grant DP160100153
and Macquarie University Research Grant 82184614.

\end{document}